\documentclass{amsart}

\usepackage{amssymb}
\usepackage{amsmath}
\usepackage{amsthm}
\usepackage{amsfonts}

\usepackage{a4wide}
\usepackage{longtable}
\usepackage{multirow}
\usepackage{setspace}

\newtheorem{theorem}{Theorem}[section]
\newtheorem{corollary}{Corollary}[section]

\theoremstyle{definition}
\newtheorem{definition}{Definition}[section]
\newtheorem{algorithm}{Algorithm}[section]
\newtheorem{remark}{Remark}[section]
\newtheorem{warning}{Warning}[section]
\newtheorem{example}{Example}[section]

\numberwithin{equation}{section}

\begin{document}


\title{Tables of pure quintic fields}

\author{Daniel C. Mayer}
\address{Naglergasse 53\\8010 Graz\\Austria}
\email{algebraic.number.theory@algebra.at}
\urladdr{http://www.algebra.at}

\thanks{Research supported by the Austrian Science Fund (FWF): projects J 0497-PHY and P 26008-N25}


\date{December 05, 2018}

\dedicatory{Dedicated to the memory of Charles John Parry}


\begin{abstract}
By making use of our generalization of
Barrucand and Cohn's theory of principal factorizations
in pure cubic fields \(\mathbb{Q}(\sqrt[3]{D})\)
and their Galois closures \(\mathbb{Q}(\zeta_3,\sqrt[3]{D})\)
with \(3\) possible types
to pure quintic fields \(L=\mathbb{Q}(\sqrt[5]{D})\)
and their pure metacyclic normal fields \(N=\mathbb{Q}(\zeta_5,\sqrt[5]{D})\)
with \(13\) possible types,
we compile an extensive database with arithmetical invariants
of the \(900\) pairwise non-isomorphic fields \(N\)
having normalized radicands in the range \(2\le D<10^3\).
Our classification is based on the Galois cohomology of the unit group \(U_N\),
viewed as a module over the automorphism group \(\mathrm{Gal}(N/K)\)
of \(N\) over the cyclotomic field \(K=\mathbb{Q}(\zeta_5)\),
by employing theorems of Hasse and Iwasawa on the Herbrand quotient
of the unit norm index \((U_K:N_{N/K}(U_N))\)
by the number \(\#(\mathcal{P}_{N/K}/\mathcal{P}_K)\) of primitive ambiguous principal ideals,
which can be interpreted as principal factors of the different \(\mathfrak{D}_{N/K}\).
The precise structure of the \(\mathbb{F}_5\)-vector space of differential principal factors
is expressed in terms of norm kernels
and central orthogonal idempotents.
A connection with integral representation theory is established via
class number relations by Parry and Walter
involving the index of subfield units \((U_N:U_0)\).
The statistical distribution of the \(13\) principal factorization types
and their refined splitting into similarity classes with representative prototypes
is discussed thoroughly.
\end{abstract}

\maketitle


\section{Introduction}
\label{s:Intro}
\noindent
At the end of his \(1975\) article on class numbers of pure quintic fields,
Parry suggested verbatim
\lq\lq In conclusion the author would like to say that he believes
a numerical study of pure quintic fields would be most interesting\rq\rq\
\cite[p. 484]{Pa}.
Of course, it would have been rather difficult
to realize Parry's desire in \(1975\).
But now, \(40\) years later, we are in the position to use
the powerful computer algebra systems PARI/GP
\cite{PARI}
and MAGMA
\cite{BCP,BCFS,MAGMA}
for starting an attack against this hard problem.
Prepared by
\cite{Ma2a,Ma2b},
this will actually be done in the present paper.

Even in \(1991\), when we generalized
Barrucand and Cohn's theory
\cite{BaCo2}
of principal factorization types
from pure cubic fields \(\mathbb{Q}(\sqrt[3]{D})\)
to pure quintic fields \(L=\mathbb{Q}(\sqrt[5]{D})\)
and their pure metacyclic normal closures \(N=\mathbb{Q}(\zeta_5,\sqrt[5]{D})\)
\cite{Ma0},
it was still impossible to verify our hypothesis
about the distinction between
absolute, intermediate and relative \textit{differential principal factors} (DPF)
\cite[(6.3)]{Ma2a}
and about the values of the \textit{unit norm index} \((U_K:N_{N/K}(U_N))\)
\cite[(1.3)]{Ma2a}
by actual computations.


All these conjectures have been proven by our most recent numerical investigations.
Our classification is based on the Hasse-Iwasawa theorem
about the Herbrand quotient of the unit group \(U_N\) of the Galois closure \(N\) of \(L\)
as a module over the relative group \(G=\mathrm{Gal}(N/K)\)
with respect to the cyclotomic subfield \(K=\mathbb{Q}(\zeta_5)\).
It only involves the unit norm index \((U_K:N_{N/K}(U_N))\) and
our \(13\) types of differential principal factorizations
\cite[Thm. 1.3]{Ma2a},
but not the index of subfield units \((U_N:U_0)\)
\cite[\S\ 5]{Ma2a}
in Parry's class number formula
\cite[(5.1)]{Ma2a}.

We begin with a collection of explicit multiplicity formulas in \S\
\ref{s:Formulas}
which are required for understanding
the subsequent extensive presentation of our computational results
in twenty tables of crucial invariants in \S\
\ref{s:Tables}.
This information admits the classification
of all \(900\) pure quintic fields \(L=\mathbb{Q}(\root{5}\of{D})\)
with normalized radicands \(2\le D<10^3\)
into \(13\) DPF types
and the refined classification into
similarity classes with representative prototypes in \S\
\ref{s:Conclusions}.

In these final conclusions,
we collect theoretical consequences of our experimental results
and draw the attention to remaining open questions.


\section{Collection of Multiplicity Formulas}
\label{s:Formulas}
\noindent
For the convenience of the reader,
we provide a summary of formulas
for calculating invariants of
pure quintic fields \(L=\mathbb{Q}(\sqrt[5]{D})\)
with normalized fifth power free radicands \(D>1\)
and their associated pure metacyclic normal fields \(N=\mathbb{Q}(\zeta,\sqrt[5]{D})\)
with a primitive fifth root of unity \(\zeta=\zeta_5\).

Let \(f\) be the class field theoretic conductor of the
relatively quintic Kummer extension \(N/K\)
over the cyclotomic field \(K=\mathbb{Q}(\zeta)\).
It is also called the conductor of the pure quintic field \(L\).
The \textit{multiplicity} \(m=m(f)\) of the conductor \(f\) indicates
the number of non-isomorphic pure metacyclic fields \(N\) sharing the common conductor \(f\),
or also, according to
\cite[Prop. 2.1]{Ma2a},
the number of normalized fifth power free radicands \(D>1\)
whose fifth roots generate 
non-isomorphic pure quintic fields \(L\) sharing the common conductor \(f\).

We adapt the general multiplicity formulas in
\cite[Thm. 2, p. 104]{Ma1}
to the quintic case \(p=5\).
If \(L\) is a field of species \(1\mathrm{a}\)
\cite[(2.6) and Exm. 2.2]{Ma2a},
i.e. \(f^4=5^6\cdot q_1^4\cdots q_t^4\),
then \(m=4^t\) where
\(t:=\#\lbrace q\in\mathbb{P}\mid q\ne 5,\ q\mid f\rbrace\).
The explicit values of \(m\) in dependence on \(t\) are given in Table
\ref{tbl:Multiplicity1a}.


\renewcommand{\arraystretch}{1.0}

\begin{table}[ht]
\caption{Multiplicity of fields of species \(1\mathrm{a}\)}
\label{tbl:Multiplicity1a}
\begin{center}
\begin{tabular}{|c||rrrrrr|}
\hline
 \(t\) & \(0\) & \(1\) &  \(2\) &  \(3\) &   \(4\) &    \(5\) \\
\hline
 \(m\) & \(1\) & \(4\) & \(16\) & \(64\) & \(256\) & \(1024\) \\
\hline
\end{tabular}
\end{center}
\end{table}


\noindent
If \(L\) is a field of species \(1\mathrm{b}\)
\cite[(2.6) and Exm. 2.2]{Ma2a},
i.e. \(f^4=5^2\cdot q_1^4\cdots q_t^4\),
then \(m=4^u\cdot X_v\) where
\(u:=\#\lbrace q\in\mathbb{P}\mid q\equiv\pm 1,\pm 7\,(\mathrm{mod}\,25),\ q\mid f\rbrace\),
\(v:=t-u\)
and \(X_j:=\frac{1}{5}\lbrack 4^j-(-1)^j\rbrack\),
that is \((X_j)_{j\ge -1}=(\frac{1}{4},0,1,3,13,51,205,\ldots)\).
The explicit values of \(m\) in dependence on \(u\) and \(v\) are given in Table
\ref{tbl:Multiplicity1b}.


\renewcommand{\arraystretch}{1.0}

\begin{table}[ht]
\caption{Multiplicity of fields of species \(1\mathrm{b}\)}
\label{tbl:Multiplicity1b}
\begin{center}
\begin{tabular}{|c||c|rrrrrr|}
\hline
 \(u\) & \(v\) & \(0\) &   \(1\) &   \(2\) &   \(3\) &   \(4\) &   \(5\) \\
\hline
 \(0\) &  \(\) & \(0\) &   \(1\) &   \(3\) &  \(13\) &  \(51\) & \(205\) \\
 \(1\) &  \(\) & \(0\) &   \(4\) &  \(12\) &  \(52\) & \(204\) & \(820\) \\
 \(2\) &  \(\) & \(0\) &  \(16\) &  \(48\) & \(208\) & \(816\) &    \(\) \\
 \(3\) &  \(\) & \(0\) &  \(64\) & \(192\) & \(832\) &    \(\) &    \(\) \\
 \(4\) &  \(\) & \(0\) & \(256\) & \(768\) &    \(\) &    \(\) &    \(\) \\
\hline
\end{tabular}
\end{center}
\end{table}


\noindent
If \(L\) is a field of species \(2\)
\cite[(2.6) and Exm. 2.2]{Ma2a},
i.e. \(f^4=5^0\cdot q_1^4\cdots q_t^4\),
then \(m=4^u\cdot X_{v-1}\) where
\(u:=\#\lbrace q\in\mathbb{P}\mid q\equiv\pm 1,\pm 7\,(\mathrm{mod}\,25),\ q\mid f\rbrace\),
\(v:=t-u\)
and \(X_j:=\frac{1}{5}\lbrack 4^j-(-1)^j\rbrack\),
that is \((X_j)_{j\ge -1}=(\frac{1}{4},0,1,3,13,51,205,\ldots)\).
The explicit values of \(m\) in dependence on \(u\) and \(v\) are given in Table
\ref{tbl:Multiplicity2}.


\renewcommand{\arraystretch}{1.0}

\begin{table}[ht]
\caption{Multiplicity of fields of species \(2\)}
\label{tbl:Multiplicity2}
\begin{center}
\begin{tabular}{|c||c|rrrrrr|}
\hline
 \(u\) & \(v\) &  \(0\) & \(1\) &   \(2\) &   \(3\) &   \(4\) &   \(5\) \\
\hline
 \(0\) &  \(\) &  \(0\) & \(0\) &   \(1\) &   \(3\) &  \(13\) &  \(51\) \\
 \(1\) &  \(\) &  \(1\) & \(0\) &   \(4\) &  \(12\) &  \(52\) & \(204\) \\
 \(2\) &  \(\) &  \(4\) & \(0\) &  \(16\) &  \(48\) & \(208\) & \(816\) \\
 \(3\) &  \(\) & \(16\) & \(0\) &  \(64\) & \(192\) & \(832\) &    \(\) \\
 \(4\) &  \(\) & \(64\) & \(0\) & \(256\) & \(768\) &    \(\) &    \(\) \\
\hline
\end{tabular}
\end{center}
\end{table}


\section{Classification by DPF types in \(20\) numerical tables}
\label{s:Tables}

\subsection{DPF types}
\label{ss:DPFTypes}

The following twenty Tables
\ref{tbl:PureQuinticFields50}
--
\ref{tbl:PureQuinticFields1000}
establish a complete classification of all \(900\)
pure metacyclic fields \(N=\mathbb{Q}(\zeta,\root{5}\of{D})\)
with normalized radicands in the range \(2\le D\le 10^3\).
With the aid of PARI/GP
\cite{PARI}
and MAGMA
\cite{MAGMA}
we have determined the \textit{differential principal factorization type}, T,
of each field \(N\)
by means of other invariants \(U,A,I,R\)
\cite[Thm. 6.1]{Ma2a}.
After several weeks of CPU time,
the date of completion was September \(17\), \(2018\).

The possible DPF types are listed in dependence on \(U,A,I,R\) in Table
\ref{tbl:DPFTypes},
where the symbol \(\times\) in the column \(\eta\), resp. \(\zeta\),
indicates the existence of a unit \(H\in U_N\), resp. \(Z\in U_N\),
such that \(\eta=N_{N/K}(H)\), resp. \(\zeta=N_{N/K}(Z)\).
The \(5\)-valuation of the \textit{unit norm index} \((U_K:N_{N/K}{U_N})\) is abbreviated by \(U\)
\cite[(1.3), (6.3)]{Ma2a}.


\renewcommand{\arraystretch}{1.0}

\begin{table}[ht]
\caption{Differential principal factorization types, T, of pure metacyclic fields \(N\)}
\label{tbl:DPFTypes}
\begin{center}
\begin{tabular}{|c|ccc|ccc|}
\hline
 T               & \(U\) & \(\eta\)   & \(\zeta\)  & \(A\) & \(I\) & \(R\) \\
\hline
 \(\alpha_1\)    & \(2\) & \(-\)      & \(-\)      & \(1\) & \(0\) & \(2\) \\
 \(\alpha_2\)    & \(2\) & \(-\)      & \(-\)      & \(1\) & \(1\) & \(1\) \\
 \(\alpha_3\)    & \(2\) & \(-\)      & \(-\)      & \(1\) & \(2\) & \(0\) \\
 \(\beta_1\)     & \(2\) & \(-\)      & \(-\)      & \(2\) & \(0\) & \(1\) \\
 \(\beta_2\)     & \(2\) & \(-\)      & \(-\)      & \(2\) & \(1\) & \(0\) \\
 \(\gamma\)      & \(2\) & \(-\)      & \(-\)      & \(3\) & \(0\) & \(0\) \\
\hline
 \(\delta_1\)    & \(1\) & \(\times\) & \(-\)      & \(1\) & \(0\) & \(1\) \\
 \(\delta_2\)    & \(1\) & \(\times\) & \(-\)      & \(1\) & \(1\) & \(0\) \\
 \(\varepsilon\) & \(1\) & \(\times\) & \(-\)      & \(2\) & \(0\) & \(0\) \\
\hline
 \(\zeta_1\)     & \(1\) & \(-\)      & \(\times\) & \(1\) & \(0\) & \(1\) \\
 \(\zeta_2\)     & \(1\) & \(-\)      & \(\times\) & \(1\) & \(1\) & \(0\) \\
 \(\eta\)        & \(1\) & \(-\)      & \(\times\) & \(2\) & \(0\) & \(0\) \\
\hline
 \(\vartheta\)   & \(0\) & \(\times\) & \(\times\) & \(1\) & \(0\) & \(0\) \\
\hline
\end{tabular}
\end{center}
\end{table}


\subsection{Justification of the computational techniques}
\label{ss:Techniques}

The steps of the following classification algorithm
are ordered by increasing requirements of CPU time.
To avoid unnecessary time consumption,
the algorithm stops at early stages already,
as soon as the DPF type is determined unambiguously.
The illustrating subfield lattice of \(N\) is drawn in Figure
\ref{fig:GaloisCorrespondence}
at the end of the paper.

\begin{algorithm}
\label{alg:Classification}
(Classification into \(13\) DPF types.)\\
\textbf{Input:}
a normalized fifth power free radicand \(D\ge 2\). \\
\textbf{Step 1:}
By purely \textit{rational} methods, without any number field constructions,
the prime factorization of the radicand \(D\) 
(including the counters \(t,u,v;n,s_2,s_4\), \S\
\ref{ss:Prototypes})
is determined.
If \(D=q\in\mathbb{P}\), \(q\equiv\pm 2\,(\mathrm{mod}\,5)\), \(q\not\equiv\pm 7\,(\mathrm{mod}\,25)\),
then \(N\) is a Polya field of type \(\varepsilon\); stop.
If \(D=q\in\mathbb{P}\), \(q=5\) or \(q\equiv\pm 7\,(\mathrm{mod}\,25)\),
then \(N\) is a Polya field of type \(\vartheta\); stop. \\
\textbf{Step 2:}
The field \(L\) of degree \(5\) is constructed.
The primes \(q_1,\ldots,q_T\) dividing the conductor \(f\) of \(N/K\) are determined,
and their overlying prime ideals \(\mathfrak{q}_1,\ldots,\mathfrak{q}_T\) in \(L\) are computed.
By means of at most \(5^T\) principal ideal tests of the elements of
\(\mathcal{I}_{L/\mathbb{Q}}/\mathcal{I}_{\mathbb{Q}}=\bigoplus_{i=1}^T\,\mathbb{F}_5\,\mathfrak{q}_i\),
the number
\(5^A:=\#\lbrace (v_1,\ldots,v_T)\in\mathbb{F}_5^T\mid\prod_{i=1}^T\,\mathfrak{q}_i^{v_i}\in\mathcal{P}_{L}\rbrace\),
that is the cardinality of \(\mathcal{P}_{L/\mathbb{Q}}/\mathcal{P}_{\mathbb{Q}}\), is determined.
If \(A=T\), then \(N\) is a Polya field.
If \(A=3\), then \(N\) is of type \(\gamma\); stop.
If \(A=2\), \(s_2=s_4=0\), \(v\ge 1\), then \(N\) is of type \(\varepsilon\); stop.
If \(A=1\), \(s_2=s_4=0\), then \(N\) is of type \(\vartheta\); stop. \\
\textbf{Step 3:}
If \(s_2\ge 1\) or \(s_4\ge 1\),
then the field \(M\) of degree \(10\) is constructed.
For the \(2\)-split primes \(\ell_1,\ldots,\ell_{s_2+s_4}\equiv\pm 1\,(\mathrm{mod}\,5)\)
among the primes \(q_1,\ldots,q_T\) dividing the conductor \(f\) of \(N/K\),
the overlying prime ideals
\(\mathcal{L}_1,\mathcal{L}_1^\tau,\ldots,\mathcal{L}_{s_2+s_4},\mathcal{L}_{s_2+s_4}^\tau\)
in \(M\) are computed.
By means of at most \(5^{s_2+s_4}\) principal ideal tests of the elements of
\((\mathcal{I}_{M/K^+}/\mathcal{I}_{K^+})\bigcap\ker(N_{M/L})
=\bigoplus_{i=1}^{s_2+s_4}\,\mathbb{F}_5\,\mathcal{K}_{(\ell_i)}\),
where \(\mathcal{K}_{(\ell_i)}=\mathcal{L}_i^{1+4\tau}\)
for \(1\le i\le s_2+s_4\),
the number
\(5^I:=\#\lbrace (v_1,\ldots,v_{s_2+s_4})\in\mathbb{F}_5^{s_2+s_4}\mid\prod_{i=1}^{s_2+s_4}\,\mathcal{K}_{(\ell_i)}^{v_i}\in\mathcal{P}_{M}\rbrace\),
that is the cardinality of \((\mathcal{P}_{M/K^+}/\mathcal{P}_{K^+})\bigcap\ker(N_{M/L})\), is determined.
If \(I=2\), then \(N\) is of type \(\alpha_3\); stop.
If \(I=1\), \(A=2\), then \(N\) is of type \(\beta_2\); stop. \\
\textbf{Step 4:}
If \(s_4\ge 1\),
then the field \(N\) of degree \(20\) is constructed.
For all \(4\)-split primes \(\ell_{s_2+1},\ldots,\ell_{s_2+s_4}\) \(\equiv +1\,(\mathrm{mod}\,5)\)
among the primes \(q_1,\ldots,q_T\) dividing the conductor \(f\) of \(N/K\),
the overlying prime ideals
\(\mathfrak{L}_{s_2+1},\mathfrak{L}_{s_2+1}^{\tau^2},\mathfrak{L}_{s_2+1}^{\tau},\mathfrak{L}_{s_2+1}^{\tau^3},
\ldots,\mathfrak{L}_{s_2+s_4},\mathfrak{L}_{s_2+s_4}^{\tau^2},\mathfrak{L}_{s_2+s_4}^{\tau},\mathfrak{L}_{s_2+s_4}^{\tau^3}\)
in \(N\) are computed.
By means of at most \(5^{2s_4}\) principal ideal tests of the elements of
\((\mathcal{I}_{N/K}/\mathcal{I}_{K})\bigcap\ker(N_{N/M})
=\bigoplus_{i=s_2+1}^{s_2+s_4}\,\left(\mathbb{F}_5\,\mathfrak{K}_{1,(\ell_i)}\right.\) \(\left.\oplus\mathbb{F}_5\,\mathfrak{K}_{2,(\ell_i)}\right)\),
where
\(\mathfrak{K}_{1,(\ell_i)}=\mathfrak{L}_i^{1+4\tau^2+2\tau+3\tau^3}\)
and
\(\mathfrak{K}_{2,(\ell_i)}=\mathfrak{L}_i^{1+4\tau^2+3\tau+2\tau^3}\)
for \(s_2+1\le i\le s_2+s_4\),
the number
\(5^R:=\#\lbrace (v_{1,s_2+1},v_{2,s_2+1},\ldots,v_{1,s_2+s_4},v_{2,s_2+s_4})\in\mathbb{F}_5^{2s_4}\mid
\prod_{i=s_2+1}^{s_2+s_4}\,\left(\mathfrak{K}_{1,(\ell_i)}^{v_{1,i}}\mathfrak{K}_{2,(\ell_i)}^{v_{2,i}}\right)\in\mathcal{P}_{N}\rbrace\),
that is the cardinality of \((\mathcal{P}_{N/K}/\mathcal{P}_{K})\bigcap\ker(N_{N/M})\), is determined.
If \(R=2\), then \(N\) is of type \(\alpha_1\); stop.
If \(R=1\), \(I=1\), then \(N\) is of type \(\alpha_2\); stop.
If \(R=1\), \(A=2\), then \(N\) is of type \(\beta_1\); stop. \\
\textbf{Step 5:}
If the type of the field \(N\) is not yet determined uniquely,
then \(U=1\) and there remain the following possibilities.
If \(v\ge 1\), then
\(N\) is of type \(\delta_1\), if \(R=1\),
of type \(\delta_2\), if \(I=1\),
and of type \(\varepsilon\), if \(R=I=0\).
If \(v=0\), then
a fundamental system \((E_j)_{1\le j\le 9}\) of units is constructed
for the unit group \(U_N\) of the field \(N\) of degree \(20\),
and all relative norms of these units with respect to the cyclotomic subfield \(K\) are computed.
If \(N_{N/K}(E_j)=\zeta_5^k\) for some \(1\le j\le 9\), \(1\le k\le 4\), then
\(N\) is of type \(\zeta_1\), if \(R=1\),
of type \(\zeta_2\), if \(I=1\),
and of type \(\eta\), if \(R=I=0\).
Otherwise the conclusions are the same as for \(v\ge 1\).
\\
\textbf{Output:}
the DPF type of the field \(N=\mathbb{Q}(\zeta_5,\sqrt[5]{D})\)
and the decision about its Polya property.
\end{algorithm}

\begin{proof}
The claims of Step 1 concerning the types \(\varepsilon,\vartheta\) are proved in items (1) and (2) of
\cite[Thm. 10.1]{Ma2a}.

For Step 2, the formulas (4.1) and (4.2) in
\cite[Thm. 4.1]{Ma2a}
give an \(\mathbb{F}_5\)-basis of the space of absolute differential factors,
and the formulas (4.3) and (4.4) in
\cite[Cor. 4.1]{Ma2a}
determine bounds for the \(\mathbb{F}_5\)-dimension \(A\)
of the space of \textit{absolute} DPF
in the field \(L\) of degree \(5\).
The Polya property was characterized in
\cite[Thm. 10.5]{Ma2a},
the claim concerning type \(\gamma\) follows from
\cite[Thm. 6.1]{Ma2a},
and the claims about the types \(\varepsilon,\vartheta\) from
\cite[Thm. 8.1 and Thm. 6.1]{Ma2a}.

For Step 3, the formulas (4.5) and (4.6) in
\cite[Thm. 4.3]{Ma2a}
give an \(\mathbb{F}_5\)-basis of the space of intermediate differential factors,
and the formulas (4.7) and (4.8) in
\cite[Cor. 4.2]{Ma2a}
determine bounds for the \(\mathbb{F}_5\)-dimension \(I\)
of the space of \textit{intermediate} DPF
in the field \(M\) of degree \(10\).
The claims concerning the types \(\alpha_3,\beta_2\) are consequences of
\cite[Thm. 6.1]{Ma2a},

For Step 4, the formulas (4.9) and (4.10) in
\cite[Thm. 4.4]{Ma2a}
give an \(\mathbb{F}_5\)-basis of the space of relative differential factors,
and the formulas (4.11) and (4.12) in
\cite[Cor. 4.3]{Ma2a}
determine bounds for the \(\mathbb{F}_5\)-dimension \(R\)
of the space of \textit{relative} DPF
in the field \(N\) of degree \(20\).
The claims concerning the types \(\alpha_1,\alpha_2,\beta_1\) are consequences of
\cite[Thm. 6.1]{Ma2a}.

Concerning Step 5,
the signature of \(N\) is \((r_1,r_2)=(0,10)\),
whence the torsion free Dirichlet unit rank of \(N\) is given by \(r=r_1+r_2-1=9\). 
The claims about all types are consequences of
\cite[Thm. 6.1]{Ma2a},
including information on the constitution of the norm group \(N_{N/K}(U_N)\).
\end{proof}

\begin{remark}
\label{rmk:Classification}
Whereas the execution of Step 1 and 2 in our Algorithm
\ref{alg:Classification},
implemented as a Magma program script
\cite{MAGMA},
is a matter of a few seconds
on a machine with clock frequency at least \(2\,\)GHz,
the CPU time for Step 3 lies in the range of several minutes.
The time requirement for Step 4 and 5 can reach hours or even days
in spite of code optimizations for the calculation of units,
in particular the use of the Magma procedures
\texttt{IndependentUnits()} and \texttt{SetOrderUnitsAreFundamental()}
prior to the call of \texttt{UnitGroup()}.
\end{remark}


\subsection{Open problems}
\label{ss:OpenQuestions}

We conjecture that considerable amounts of CPU time can be saved in our Algorithm
\ref{alg:Classification}
by computing the logarithmic \(5\)-class numbers \(V_F:=v_5(h_F)\)
of the fields \(F\in\lbrace L,M,N\rbrace\),
which admit the determination of the
logarithmic indices \(E\), resp. \(E^+\), of subfield units
in the Parry
\cite{Pa}, 
resp. Kobayashi
\cite{Ky1,Ky2},
class number relation, according to the formulas
\begin{equation}
\label{eqn:LogInd}
E=5+V_N-4\cdot V_L, \qquad E^+=2+V_M-2\cdot V_L.
\end{equation}
However, first there would be required rigorous proofs of
the heuristic connections between \(E,E^+\) and the DPF types in Table
\ref{tbl:LogInd},
where \((E,E^+)=(1,0)\) implies type \(\alpha_2\),
\((E,E^+)=(2,0)\) implies type \(\alpha_3\),
\((E,E^+)=(4,2)\) implies type \(\delta_1\),
but
\((E,E^+)=(2,1)\) admits types \(\alpha_1,\delta_1\),
\((E,E^+)=(3,1)\) admits types \(\alpha_2,\beta_1,\delta_2\),
\((E,E^+)=(4,1)\) admits types \(\beta_2,\zeta_2\),
\((E,E^+)=(5,2)\) admits types \(\beta_1,\varepsilon,\zeta_1,\vartheta\), and
\((E,E^+)=(6,2)\) admits types \(\gamma,\eta\).
\((E,E^+)=(0,0)\) seems to be impossible.


\renewcommand{\arraystretch}{1.0}

\begin{table}[ht]
\caption{Logarithmic indices \(E,E^+\) of subfield units for DPF types, T}
\label{tbl:LogInd}
\begin{center}
\begin{tabular}{|c|cc|c|cc|}
\hline
 T               & \(E\) & \(E^+\) & or & \(E\) & \(E^+\) \\
\hline
 \(\alpha_1\)    & \(2\) & \(1\)   &    & \(\)  & \(\)    \\
 \(\alpha_2\)    & \(1\) & \(0\)   &    & \(3\) & \(1\)   \\
 \(\alpha_3\)    & \(2\) & \(0\)   &    & \(\)  & \(\)    \\
 \(\beta_1\)     & \(3\) & \(1\)   &    & \(5\) & \(2\)   \\
 \(\beta_2\)     & \(4\) & \(1\)   &    & \(\)  & \(\)    \\
 \(\gamma\)      & \(6\) & \(2\)   &    & \(\)  & \(\)    \\
\hline
 \(\delta_1\)    & \(2\) & \(1\)   &    & \(4\) & \(2\)   \\
 \(\delta_2\)    & \(3\) & \(1\)   &    & \(\)  & \(\)    \\
 \(\varepsilon\) & \(5\) & \(2\)   &    & \(\)  & \(\)    \\
\hline
 \(\zeta_1\)     & \(5\) & \(2\)   &    & \(\)  & \(\)    \\
 \(\zeta_2\)     & \(4\) & \(1\)   &    & \(\)  & \(\)    \\
 \(\eta\)        & \(6\) & \(2\)   &    & \(\)  & \(\)    \\
\hline
 \(\vartheta\)   & \(5\) & \(2\)   &    & \(\)  & \(\)    \\
\hline
\end{tabular}
\end{center}
\end{table}


\subsection{Conventions and notation in the tables}
\label{ss:Conventions}

The \textit{normalized} radicand \(D=q_1^{e_1}\cdots q_s^{e_s}\)
of a pure metacyclic field \(N\) of degree \(20\)
is minimal among the powers \(D^n\), \(1\le n\le 4\),
with corresponding exponents \(e_j\) reduced modulo \(5\). 
The normalization of the radicands \(D\) provides a warranty that
all fields are pairwise non-isomorphic
\cite[Prop. 2.1]{Ma2a}.

Prime factors are given for composite radicands \(D\) only. 
Dedekind's \textit{species}, S, of radicands is refined by
distinguishing \(5\mid D\) (species 1a) and \(\gcd(5,D) = 1\) (species 1b)
among radicands \(D\not\equiv\pm 1,\pm 7\,(\mathrm{mod}\,25)\) (species 1).
By the species and factorization of \(D\),
the shape of the \textit{conductor} \(f\) is determined.
We give the fourth power \(f^4\) to avoid fractional exponents.
Additionally, the \textit{multiplicity} \(m\) indicates
the number of non-isomorphic fields sharing a common conductor \(f\) (\S
\ref{s:Formulas}).
The symbol \(V_F\) briefly denotes the \(5\)-valuation of the order \(h_F=\#\mathrm{Cl}(F)\)
of the class group \(\mathrm{Cl}(F)\) of a number field \(F\).
By \(E\) we denote the exponent of the power in the \textit{index of subfield units} \((U_N:U_0)=5^E\). 

An asterisk denotes the smallest radicand
with given Dedekind kind, DPF type and \(5\)-class groups \(\mathrm{Cl}_5(F)\), \(F\in\lbrace L,M,N\rbrace\).
The latter are usually elementary abelian, except for the cases indicated by an additional asterisk (see \S\
\ref{ss:NonElem}).

Principal factors, P, are listed
when their constitution is not a consequence of the other information.
According to
\cite[Thm. 7.2., item (1)]{Ma2a}
it suffices to give the rational integer norm of \textit{absolute} principal factors.
For \textit{intermediate} principal factors, we use the symbols
\(\mathcal{K}:=\mathcal{L}^{1-\tau}=\alpha\mathcal{O}_M\) with \(\alpha\in M\)
or \(\mathcal{L}=\lambda\mathcal{O}_M\) with a prime element \(\lambda\in M\)
(which implies \(\mathcal{L}^\tau=\lambda^\tau\mathcal{O}_M\) and thus also \(\mathcal{K}=\lambda^{1-\tau}\mathcal{O}_M\)).
Here, \((\mathcal{L}^{1+\tau})^5=\ell\mathcal{O}_M\)
when a prime \(\ell\equiv\pm 1\,(\mathrm{mod}\,5)\) divides the radicand \(D\).
For \textit{relative} principal factors, we use the symbols
\(\mathfrak{K}_1:=\mathfrak{L}^{1+4\tau^2+2\tau+3\tau^3}=A_1\mathcal{O}_N\)
and
\(\mathfrak{K}_2:=\mathfrak{L}^{1+4\tau^2+3\tau+2\tau^3}=A_2\mathcal{O}_N\)
with \(A_1,A_2\in N\).
Here, \((\mathfrak{L}^{1+\tau+\tau^2+\tau^3})^5=\ell\mathcal{O}_N\)
when a prime number \(\ell\equiv +1\,(\mathrm{mod}\,5)\) divides the radicand \(D\).
(Kernel ideals in
\cite[\S\ 7]{Ma2a}.)

The quartet \((1,2,4,5)\) indicates conditions which
either enforce a reduction of possible DPF types
or enable certain DPF types.
The lack of a prime divisor \(\ell\equiv\pm 1\,(\mathrm{mod}\,5)\)
together with the existence of a prime divisor \(q\not\equiv\pm 7\,(\mathrm{mod}\,25)\) and \(q\ne 5\) of \(D\)
is indicated by a symbol \(\times\) for the component \(1\).
In these cases, only the two DPF types \(\gamma\) and \(\varepsilon\) can occur
\cite[Thm. 8.1]{Ma2a}.

A symbol \(\times\) for the component \(2\)
emphasizes a prime divisor \(\ell\equiv -1\,(\mathrm{mod}\,5)\) of \(D\)
and the possibility of intermediate principal factors in \(M\), like \(\mathcal{L}\) and \(\mathcal{K}\).
A symbol \(\times\) for the component \(4\)
emphasizes a prime divisor \(\ell\equiv +1\,(\mathrm{mod}\,5)\) of \(D\)
and the possibility of relative principal factors in \(N\), like \(\mathfrak{K}_1\) and \(\mathfrak{K}_2\).
The \(\times\) symbol is replaced by \(\otimes\) if the facility is used completely,
and by \((\times)\) if the facility is only used partially.

If \(D\) has only prime divisors \(q\equiv\pm 1,\pm 7\,(\mathrm{mod}\,25)\) or \(q=5\),
a symbol \(\times\) is placed in component \(5\).
In these cases, \(\zeta\) can occur as a norm \(N_{N/K}(Z)\) of some unit in \(Z\in U_N\).
If it actually does, the \(\times\) is replaced by \(\otimes\)
\cite[\S\ 8]{Ma2a}.

\newpage

\renewcommand{\arraystretch}{1.0}

\begin{table}[ht]
\caption{\(40\) pure metacyclic fields with normalized radicands \(2\le D\le 52\)}
\label{tbl:PureQuinticFields50}
\begin{center}
\begin{tabular}{|r|rc|ccr|cccc|ccc|}
\hline
 No. &   \(D\) &             Factors & S  &             \(f^4\) &  \(m\) & \(V_L\) & \(V_M\) & \(V_N\) & \(E\) & \((1,2,4,5)\)        & T               & P               \\
\hline
   1 &  *\(2\) &                     & 1b &         \(5^2 2^4\) &  \(1\) &   \(0\) &   \(0\) &   \(0\) & \(5\) & \((\times,-,-,-)\)   & \(\varepsilon\) &                 \\
   2 &   \(3\) &                     & 1b &         \(5^2 3^4\) &  \(1\) &   \(0\) &   \(0\) &   \(0\) & \(5\) & \((\times,-,-,-)\)   & \(\varepsilon\) &                 \\
   3 &  *\(5\) &                     & 1a &             \(5^6\) &  \(1\) &   \(0\) &   \(0\) &   \(0\) & \(5\) & \((-,-,-,\otimes)\)  & \(\vartheta\)   &                 \\
   4 &  *\(6\) &        \(2\cdot 3\) & 1b &     \(5^2 2^4 3^4\) &  \(3\) &   \(0\) &   \(0\) &   \(1\) & \(6\) & \((\times,-,-,-)\)   & \(\gamma\)      &                 \\
   5 &  *\(7\) &                     & 2  &             \(7^4\) &  \(1\) &   \(0\) &   \(0\) &   \(0\) & \(5\) & \((-,-,-,\otimes)\)  & \(\vartheta\)   &                 \\
   6 & *\(10\) &        \(2\cdot 5\) & 1a &         \(5^6 2^4\) &  \(4\) &   \(0\) &   \(0\) &   \(0\) & \(5\) & \((\times,-,-,-)\)   & \(\varepsilon\) &                 \\
   7 & *\(11\) &                     & 1b &        \(5^2 11^4\) &  \(1\) &   \(1\) &   \(1\) &   \(2\) & \(3\) & \((-,-,\otimes,-)\)  & \(\alpha_2\)    & \(\mathcal{L},\mathfrak{K}_1\) \\
\hline
\end{tabular}
\end{center}
\end{table}

\newpage

\renewcommand{\arraystretch}{1.0}

\begin{table}[ht]
\caption{\(45\) pure metacyclic fields with normalized radicands \(53\le D\le 104\)}
\label{tbl:PureQuinticFields100}
\begin{center}
\begin{tabular}{|r|rc|ccr|cccc|ccc|}
\hline
 No. &   \(D\) &               Factors & S  &              \(f^4\) &  \(m\) & \(V_L\) & \(V_M\) & \(V_N\) & \(E\) & \((1,2,4,5)\)        & T               & P               \\
\hline
  41 &  \(53\) &                       & 1b &         \(5^2 53^4\) &  \(1\) &   \(0\) &   \(0\) &   \(0\) & \(5\) & \((\times,-,-,-)\)   & \(\varepsilon\) &                 \\
  42 & *\(55\) &         \(5\cdot 11\) & 1a &         \(5^6 11^4\) &  \(4\) &   \(1\) &   \(1\) &   \(2\) & \(3\) & \((-,-,\otimes,-)\)  & \(\alpha_2\)    & \(\mathcal{K},\mathfrak{K}_1\) \\
  43 &  \(56\) &        \(2^3\cdot 7\) & 1b &      \(5^2 2^4 7^4\) &  \(4\) &   \(0\) &   \(0\) &   \(1\) & \(6\) & \((\times,-,-,-)\)   & \(\gamma\)      &                 \\
  44 & *\(57\) &         \(3\cdot 19\) & 2  &         \(3^4 19^4\) &  \(1\) &   \(1\) &   \(1\) &   \(2\) & \(3\) & \((-,\otimes,-,-)\)  & \(\delta_2\)    & \(\mathcal{K}\) \\
  45 &  \(58\) &         \(2\cdot 29\) & 1b &     \(5^2 2^4 29^4\) &  \(3\) &   \(1\) &   \(1\) &   \(3\) & \(4\) & \((-,\otimes,-,-)\)  & \(\beta_2\)     & \(29\cdot 5^2,\mathcal{K}\) \\
  46 &  \(59\) &                       & 1b &         \(5^2 59^4\) &  \(1\) &   \(1\) &   \(1\) &   \(2\) & \(3\) & \((-,\otimes,-,-)\)  & \(\delta_2\)    & \(\mathcal{L}\) \\
  47 &  \(60\) & \(2^2\cdot 3\cdot 5\) & 1a &      \(5^6 2^4 3^4\) & \(16\) &   \(0\) &   \(0\) &   \(1\) & \(6\) & \((\times,-,-,-)\)   & \(\gamma\)      &                 \\
\hline
\end{tabular}
\end{center}
\end{table}

\newpage

\renewcommand{\arraystretch}{1.0}

\begin{table}[ht]
\caption{\(45\) pure metacyclic fields with normalized radicands \(105\le D\le 155\)}
\label{tbl:PureQuinticFields150}
\begin{center}
\begin{tabular}{|r|rc|ccr|cccc|ccc|}
\hline
 No. &    \(D\) &                Factors & S  &              \(f^4\) &  \(m\) & \(V_L\) & \(V_M\) & \(V_N\) & \(E\) & \((1,2,4,5)\)        & T               & P               \\
\hline
  86 &  \(105\) &    \(3\cdot 5\cdot 7\) & 1a &      \(5^6 3^4 7^4\) & \(16\) &   \(0\) &   \(0\) &   \(1\) & \(6\) & \((\times,-,-,-)\)   & \(\gamma\)      &                 \\
  87 &  \(106\) &          \(2\cdot 53\) & 1b &     \(5^2 2^4 53^4\) &  \(3\) &   \(0\) &   \(0\) &   \(1\) & \(6\) & \((\times,-,-,-)\)   & \(\gamma\)      &                 \\
  88 &  \(107\) &                        & 2  &            \(107^4\) &  \(1\) &   \(0\) &   \(0\) &   \(0\) & \(5\) & \((-,-,-,\otimes)\)  & \(\vartheta\)   &                 \\
  89 &  \(109\) &                        & 1b &        \(5^2 109^4\) &  \(1\) &   \(1\) &   \(1\) &   \(2\) & \(3\) & \((-,\otimes,-,-)\)  & \(\delta_2\)    & \(\mathcal{L}\) \\
  90 & *\(110\) &   \(2\cdot 5\cdot 11\) & 1a &     \(5^6 2^4 11^4\) & \(16\) &   \(1\) &   \(1\) &   \(3\) & \(4\) & \((-,-,(\times),-)\) & \(\beta_2\)     & \(11,\mathcal{L}\) \\
  91 &  \(111\) &          \(3\cdot 37\) & 1b &     \(5^2 3^4 37^4\) &  \(3\) &   \(0\) &   \(0\) &   \(1\) & \(6\) & \((\times,-,-,-)\)   & \(\gamma\)      &                 \\
  92 &  \(112\) &         \(2^4\cdot 7\) & 1b &      \(5^2 2^4 7^4\) &  \(4\) &   \(0\) &   \(0\) &   \(1\) & \(6\) & \((\times,-,-,-)\)   & \(\gamma\)      &                 \\
\hline
\end{tabular}
\end{center}
\end{table}

\newpage

\renewcommand{\arraystretch}{1.0}

\begin{table}[ht]
\caption{\(45\) pure metacyclic fields with normalized radicands \(156\le D\le 207\)}
\label{tbl:PureQuinticFields200}
\begin{center}
\begin{tabular}{|r|rc|ccr|cccc|ccc|}
\hline
 No. &    \(D\) &                Factors & S  &              \(f^4\) &  \(m\) & \(V_L\) & \(V_M\) & \(V_N\) & \(E\) & \((1,2,4,5)\)        & T               & P               \\
\hline
 131 &  \(156\) & \(2^2\cdot 3\cdot 13\) & 1b & \(5^2 2^4 3^4 13^4\) & \(13\) &   \(1\) &   \(2\) &   \(5\) & \(6\) & \((\times,-,-,-)\)   & \(\gamma\)      & \(3,13\cdot 5^2\) \\
 132 &  \(157\) &                        & 2  &            \(157^4\) &  \(1\) &   \(0\) &   \(0\) &   \(0\) & \(5\) & \((-,-,-,\otimes)\)  & \(\vartheta\)   &                 \\
 133 &  \(158\) &          \(2\cdot 79\) & 1b &     \(5^2 2^4 79^4\) &  \(3\) &   \(1\) &   \(1\) &   \(3\) & \(4\) & \((-,\otimes,-,-)\)  & \(\beta_2\)     & \(5,\mathcal{K}\) \\
 134 &  \(159\) &          \(3\cdot 53\) & 1b &     \(5^2 3^4 53^4\) &  \(3\) &   \(1\) &   \(2\) &   \(4\) & \(5\) & \((\times,-,-,-)\)   & \(\varepsilon\) & \(5\)           \\
 135 &  \(161\) &          \(7\cdot 23\) & 1b &     \(5^2 7^4 23^4\) &  \(4\) &   \(0\) &   \(0\) &   \(1\) & \(6\) & \((\times,-,-,-)\)   & \(\gamma\)      &                 \\
 136 &  \(163\) &                        & 1b &        \(5^2 163^4\) &  \(1\) &   \(0\) &   \(0\) &   \(0\) & \(5\) & \((\times,-,-,-)\)   & \(\varepsilon\) &                 \\
 137 &  \(164\) &        \(2^2\cdot 41\) & 1b &     \(5^2 2^4 41^4\) &  \(3\) &   \(1\) &   \(1\) &   \(3\) & \(4\) & \((-,-,(\times),-)\) & \(\beta_2\)     & \(41\cdot 5,\mathcal{K}\) \\
\hline
\end{tabular}
\end{center}
\end{table}

\newpage

\renewcommand{\arraystretch}{1.0}

\begin{table}[ht]
\caption{\(45\) pure metacyclic fields with normalized radicands \(208\le D\le 259\)}
\label{tbl:PureQuinticFields250}
\begin{center}
\begin{tabular}{|r|rc|ccr|cccc|ccc|}
\hline
 No. &    \(D\) &                    Factors & S  &              \(f^4\) &  \(m\) & \(V_L\) & \(V_M\) & \(V_N\) & \(E\) & \((1,2,4,5)\)        & T               & P               \\
\hline
 176 &  \(208\) &            \(2^4\cdot 13\) & 1b &     \(5^2 2^4 13^4\) &  \(3\) &   \(0\) &   \(0\) &   \(1\) & \(6\) & \((\times,-,-,-)\)   & \(\gamma\)      &                 \\
 177 & *\(209\) &             \(11\cdot 19\) & 1b &    \(5^2 11^4 19^4\) &  \(3\) &   \(2\) &   \(3\) &   \(7\) & \(4\) & \((-,\otimes,\times,-)\) & \(\beta_2\) & \(11\cdot 5^2,\mathcal{K}_{(19)}\) \\
 178 & *\(210\) & \(2\cdot 3\cdot 5\cdot 7\) & 1a &  \(5^6 2^4 3^4 7^4\) & \(64\) &   \(1\) &   \(2\) &   \(5\) & \(6\) & \((\times,-,-,-)\)   & \(\gamma\)      & \(7,3^2\cdot 5\) \\
 179 & *\(211\) &                            & 1b &        \(5^2 211^4\) &  \(1\) &   \(3\) &   \(5\) &   \(9\) & \(2\) & \((-,-,\otimes,-)\)  & \(\delta_1\)    & \(\mathfrak{K}_2\) \\
 180 &  \(212\) &            \(2^2\cdot 53\) & 1b &     \(5^2 2^4 53^4\) &  \(3\) &   \(0\) &   \(0\) &   \(1\) & \(6\) & \((\times,-,-,-)\)   & \(\gamma\)      &                 \\
 181 &  \(213\) &              \(3\cdot 71\) & 1b &     \(5^2 3^4 71^4\) &  \(3\) &   \(1\) &   \(1\) &   \(3\) & \(4\) & \((-,-,(\times),-)\) & \(\beta_2\)     & \(3\cdot 5^3,\mathcal{K}\) \\
 182 &  \(214\) &             \(2\cdot 107\) & 1b &    \(5^2 2^4 107^4\) &  \(4\) &   \(0\) &   \(0\) &   \(1\) & \(6\) & \((\times,-,-,-)\)   & \(\gamma\)      &                 \\
\hline
\end{tabular}
\end{center}
\end{table}

\newpage

\renewcommand{\arraystretch}{1.0}

\begin{table}[ht]
\caption{\(45\) pure metacyclic fields with normalized radicands \(260\le D\le 307\)}
\label{tbl:PureQuinticFields300}
\begin{center}
\begin{tabular}{|r|rc|ccr|cccc|ccc|}
\hline
 No. &    \(D\) &                Factors & S  &              \(f^4\) &  \(m\) & \(V_L\) & \(V_M\) & \(V_N\) & \(E\) & \((1,2,4,5)\)        & T               & P               \\
\hline
 221 &  \(260\) & \(2^2\cdot 5\cdot 13\) & 1a &     \(5^6 2^4 13^4\) & \(16\) &   \(0\) &   \(0\) &   \(1\) & \(6\) & \((\times,-,-,-)\)   & \(\gamma\)      &                 \\
 222 &  \(261\) &        \(3^2\cdot 29\) & 1b &     \(5^2 3^4 29^4\) &  \(3\) &   \(1\) &   \(1\) &   \(3\) & \(4\) & \((-,\otimes,-,-)\)  & \(\beta_2\)     & \(29,\mathcal{L}\) \\
 223 &  \(262\) &         \(2\cdot 131\) & 1b &    \(5^2 2^4 131^4\) &  \(3\) &   \(1\) &   \(1\) &   \(3\) & \(4\) & \((-,-,(\times),-)\) & \(\beta_2\)     & \(131\cdot 5,\mathcal{K}\) \\
 224 &  \(263\) &                        & 1b &        \(5^2 263^4\) &  \(1\) &   \(0\) &   \(0\) &   \(0\) & \(5\) & \((\times,-,-,-)\)   & \(\varepsilon\) &                 \\
 225 &  \(264\) & \(2^3\cdot 3\cdot 11\) & 1b & \(5^2 2^4 3^4 11^4\) & \(13\) &   \(1\) &   \(2\) &   \(5\) & \(6\) & \((-,-,\times,-)\)   & \(\gamma\)      & \(2\cdot 5,2\cdot 11\) \\
 226 &  \(265\) &          \(5\cdot 53\) & 1a &         \(5^6 53^4\) &  \(4\) &   \(0\) &   \(0\) &   \(0\) & \(5\) & \((\times,-,-,-)\)   & \(\varepsilon\) &                 \\
 227 & *\(266\) &   \(2\cdot 7\cdot 19\) & 1b & \(5^2 2^4 7^4 19^4\) & \(12\) &   \(1\) &   \(2\) &   \(5\) & \(6\) & \((-,\times,-,-)\)   & \(\gamma\)      & \(5,2^3\cdot 7\) \\
\hline
\end{tabular}
\end{center}
\end{table}

\newpage

\renewcommand{\arraystretch}{1.0}

\begin{table}[ht]
\caption{\(45\) pure metacyclic fields with normalized radicands \(308\le D\le 357\)}
\label{tbl:PureQuinticFields350}
\begin{center}
\begin{tabular}{|r|rc|ccr|cccc|ccc|}
\hline
 No. &    \(D\) &                     Factors & S  &              \(f^4\) &  \(m\) & \(V_L\) & \(V_M\) & \(V_N\) & \(E\) & \((1,2,4,5)\)        & T               & P               \\
\hline
 266 &  \(308\) &      \(2^2\cdot 7\cdot 11\) & 1b & \(5^2 2^4 7^4 11^4\) & \(12\) &   \(1\) &   \(2\) &   \(5\) & \(6\) & \((-,-,\times,-)\)   & \(\gamma\)      & \(11\cdot 5,7\cdot 5^2\) \\
 267 &  \(309\) &              \(3\cdot 103\) & 1b &    \(5^2 3^4 103^4\) &  \(3\) &   \(1\) &   \(2\) &   \(4\) & \(5\) & \((\times,-,-,-)\)   & \(\varepsilon\) & \(103\cdot 5^2\)\\
 268 &  \(310\) &        \(2\cdot 5\cdot 31\) & 1a &     \(5^6 2^4 31^4\) & \(16\) &   \(1\) &   \(1\) &   \(3\) & \(4\) & \((-,-,(\times),-)\) & \(\beta_2\)     & \(5,\mathcal{K}\) \\
 269 &  \(311\) &                             & 1b &        \(5^2 311^4\) &  \(1\) &   \(1\) &   \(1\) &   \(2\) & \(3\) & \((-,-,\otimes,-)\)  & \(\alpha_2\)    & \(\mathcal{L},\mathfrak{K}_2\) \\
 270 &  \(312\) &      \(2^3\cdot 3\cdot 13\) & 1b & \(5^2 2^4 3^4 13^4\) & \(13\) &   \(1\) &   \(2\) &   \(5\) & \(6\) & \((\times,-,-,-)\)   & \(\gamma\)      & \(5,2^4\cdot 13\) \\
 271 &  \(313\) &                             & 1b &        \(5^2 313^4\) &  \(1\) &   \(0\) &   \(0\) &   \(0\) & \(5\) & \((\times,-,-,-)\)   & \(\varepsilon\) &                 \\
 272 &  \(314\) &              \(2\cdot 157\) & 1b &    \(5^2 2^4 157^4\) &  \(4\) &   \(0\) &   \(0\) &   \(1\) & \(6\) & \((\times,-,-,-)\)   & \(\gamma\)      &                 \\
\hline
\end{tabular}
\end{center}
\end{table}

\newpage

\renewcommand{\arraystretch}{1.0}

\begin{table}[ht]
\caption{\(45\) pure metacyclic fields with normalized radicands \(358\le D\le 408\)}
\label{tbl:PureQuinticFields400}
\begin{center}
\begin{tabular}{|r|rc|ccr|cccc|ccc|}
\hline
 No. &    \(D\) &                     Factors & S  &              \(f^4\) &  \(m\) & \(V_L\) & \(V_M\) & \(V_N\) & \(E\) & \((1,2,4,5)\)        & T               & P               \\
\hline
 311 &  \(358\) &              \(2\cdot 179\) & 1b &    \(5^2 2^4 179^4\) &  \(3\) &   \(1\) &   \(1\) &   \(3\) & \(4\) & \((-,\otimes,-,-)\)  & \(\beta_2\)     & \(2,\mathcal{L}\) \\
 312 &  \(359\) &                             & 1b &        \(5^2 359^4\) &  \(1\) &   \(1\) &   \(1\) &   \(3\) & \(4\) & \((-,\otimes,-,-)\)  & \(\beta_2\)     & \(5,\mathcal{L}\) \\
 313 &  \(360\) &     \(2^3\cdot 3^2\cdot 5\) & 1a &      \(5^6 2^4 3^4\) & \(16\) &   \(0\) &   \(0\) &   \(1\) & \(6\) & \((\times,-,-,-)\)   & \(\gamma\)      &                 \\
 314 &  \(362\) &              \(2\cdot 181\) & 1b &    \(5^2 2^4 181^4\) &  \(3\) &   \(1\) &   \(1\) &   \(3\) & \(4\) & \((-,-,(\times),-)\) & \(\beta_2\)     & \(181\cdot 5,\mathcal{K}\) \\
 315 &  \(364\) &      \(2^2\cdot 7\cdot 13\) & 1b & \(5^2 2^4 7^4 13^4\) & \(12\) &   \(1\) &   \(2\) &   \(5\) & \(6\) & \((\times,-,-,-)\)   & \(\gamma\)      & \(2,7\cdot 13\) \\
 316 &  \(365\) &               \(5\cdot 73\) & 1a &         \(5^6 73^4\) &  \(4\) &   \(0\) &   \(0\) &   \(0\) & \(5\) & \((\times,-,-,-)\)   & \(\varepsilon\) &                 \\
 317 &  \(366\) &        \(2\cdot 3\cdot 61\) & 1b & \(5^2 2^4 3^4 61^4\) & \(13\) &   \(2\) &   \(3\) &   \(6\) & \(3\) & \((-,-,\otimes,-)\)  & \(\beta_1\)     & \(3\cdot 61^2\cdot 5^4,\mathfrak{K}_1\) \\
\hline
\end{tabular}
\end{center}
\end{table}

\newpage

\renewcommand{\arraystretch}{1.0}

\begin{table}[ht]
\caption{\(45\) pure metacyclic fields with normalized radicands \(409\le D\le 458\)}
\label{tbl:PureQuinticFields450}
\begin{center}
\begin{tabular}{|r|rc|ccr|cccc|ccc|}
\hline
 No. &    \(D\) &                      Factors & S  &             \(f^4\) &  \(m\) & \(V_L\) & \(V_M\) & \(V_N\) & \(E\) & \((1,2,4,5)\)        & T               & P               \\
\hline
 356 &  \(409\) &                              & 1b &       \(5^2 409^4\) &  \(1\) &   \(1\) &   \(1\) &   \(2\) & \(3\) & \((-,\otimes,-,-)\)  & \(\delta_2\)    & \(\mathcal{L}\) \\
 357 &  \(410\) &         \(2\cdot 5\cdot 41\) & 1a &    \(5^6 2^4 41^4\) & \(16\) &   \(1\) &   \(1\) &   \(3\) & \(4\) & \((-,-,(\times),-)\) & \(\beta_2\)     & \(2\cdot 41^2,\mathcal{K}\) \\
 358 &  \(411\) &               \(3\cdot 137\) & 1b &   \(5^2 3^4 137^4\) &  \(3\) &   \(0\) &   \(0\) &   \(1\) & \(6\) & \((\times,-,-,-)\)   & \(\gamma\)      &                 \\
 359 &  \(412\) &             \(2^2\cdot 103\) & 1b &   \(5^2 2^4 103^4\) &  \(3\) &   \(1\) &   \(2\) &   \(4\) & \(5\) & \((\times,-,-,-)\)   & \(\varepsilon\) & \(2\cdot 5^2\)  \\
 360 &  \(413\) &                \(7\cdot 59\) & 1b &    \(5^2 7^4 59^4\) &  \(4\) &   \(1\) &   \(1\) &   \(3\) & \(4\) & \((-,\otimes,-,-)\)  & \(\beta_2\)     & \(59\cdot 5,\mathcal{K}\) \\
 361 &  \(414\) &       \(2\cdot 3^2\cdot 23\) & 1b &\(5^2 2^4 3^4 23^4\) & \(13\) &   \(1\) &   \(2\) &   \(5\) & \(6\) & \((\times,-,-,-)\)   & \(\gamma\)      & \(2,3\) \\
 362 &  \(415\) &                \(5\cdot 83\) & 1a &        \(5^6 83^4\) &  \(4\) &   \(0\) &   \(0\) &   \(0\) & \(5\) & \((\times,-,-,-)\)   & \(\varepsilon\) &                 \\
\hline
\end{tabular}
\end{center}
\end{table}

\newpage

\renewcommand{\arraystretch}{1.0}

\begin{table}[ht]
\caption{\(45\) pure metacyclic fields with normalized radicands \(459\le D\le 508\)}
\label{tbl:PureQuinticFields500}
\begin{center}
\begin{tabular}{|r|rc|ccr|cccc|ccc|}
\hline
 No. &    \(D\) &                     Factors & S  &                  \(f^4\) &  \(m\) & \(V_L\) & \(V_M\) & \(V_N\) & \(E\) & \((1,2,4,5)\)        & T               & P               \\
\hline
 401 &  \(459\) &             \(3^3\cdot 17\) & 1b &         \(5^2 3^4 17^4\) &  \(3\) &   \(0\) &   \(0\) &   \(1\) & \(6\) & \((\times,-,-,-)\)   & \(\gamma\)      &                 \\
 402 &  \(460\) &      \(2^2\cdot 5\cdot 23\) & 1a &         \(5^6 2^4 23^4\) & \(16\) &   \(0\) &   \(0\) &   \(1\) & \(6\) & \((\times,-,-,-)\)   & \(\gamma\)      &                 \\
 403 &  \(461\) &                             & 1b &            \(5^2 461^4\) &  \(1\) &   \(1\) &   \(2\) &   \(3\) & \(4\) & \((-,-,\otimes,-)\)  & \(\delta_1\)    & \(\mathfrak{K}_1\) \\
 404 & *\(462\) & \(2\cdot 3\cdot 7\cdot 11\) & 1b & \(5^2 2^4 3^4 7^4 11^4\) & \(52\) &   \(2\) &   \(4\) &   \(9\) & \(6\) & \((-,-,\times,-)\)   & \(\gamma\)      & \(5^2\cdot 7,3\cdot 5\cdot 11^2\)         \\
 405 &  \(463\) &                             & 1b &            \(5^2 463^4\) &  \(1\) &   \(0\) &   \(0\) &   \(0\) & \(5\) & \((\times,-,-,-)\)   & \(\varepsilon\) &                 \\
 406 &  \(464\) &             \(2^4\cdot 29\) & 1b &         \(5^2 2^4 29^4\) &  \(3\) &   \(1\) &   \(1\) &   \(3\) & \(4\) & \((-,\otimes,-,-)\)  & \(\beta_2\)     & \(5,\mathcal{K}\) \\
 407 & *\(465\) &        \(3\cdot 5\cdot 31\) & 1a &         \(5^6 3^4 31^4\) & \(16\) &   \(2\) &   \(3\) &  \(7*\) & \(4\) & \((-,-,(\times),-)\) & \(\beta_2\)     & \(5,\mathcal{K}\) \\
\hline
\end{tabular}
\end{center}
\end{table}

\newpage

\renewcommand{\arraystretch}{1.0}

\begin{table}[ht]
\caption{\(45\) pure metacyclic fields with normalized radicands \(509\le D\le 556\)}
\label{tbl:PureQuinticFields550}
\begin{center}
\begin{tabular}{|r|rc|ccr|cccc|ccc|}
\hline
 No. &    \(D\) &                     Factors & S  &                  \(f^4\) &  \(m\) & \(V_L\) & \(V_M\) & \(V_N\) & \(E\) & \((1,2,4,5)\)        & T               & P               \\
\hline
 446 &  \(509\) &                             & 1b &            \(5^2 509^4\) &  \(1\) &   \(1\) &   \(1\) &   \(2\) & \(3\) & \((-,\otimes,-,-)\)  & \(\delta_2\)    & \(\mathcal{L}\) \\
 447 &  \(510\) & \(2\cdot 3\cdot 5\cdot 17\) & 1a &     \(5^6 2^4 3^4 17^4\) & \(64\) &   \(1\) &   \(2\) &   \(5\) & \(6\) & \((\times,-,-,-)\)   & \(\gamma\)      & \(2\cdot 3,2\cdot 17\) \\
 448 &  \(511\) &               \(7\cdot 73\) & 1b &         \(5^2 7^4 73^4\) &  \(4\) &   \(0\) &   \(0\) &   \(1\) & \(6\) & \((\times,-,-,-)\)   & \(\gamma\)      &                 \\
 449 &  \(513\) &             \(3^3\cdot 19\) & 1b &         \(5^2 3^4 19^4\) &  \(3\) &   \(1\) &   \(2\) &   \(4\) & \(5\) & \((-,\times,-,-)\)   & \(\varepsilon\) & \(3\cdot 5^2\)  \\
 450 &  \(514\) &              \(2\cdot 257\) & 1b &        \(5^2 2^4 257^4\) &  \(4\) &   \(0\) &   \(0\) &   \(1\) & \(6\) & \((\times,-,-,-)\)   & \(\gamma\)      &                 \\
 451 &  \(515\) &              \(5\cdot 103\) & 1a &            \(5^6 103^4\) &  \(4\) &   \(0\) &   \(0\) &   \(0\) & \(5\) & \((\times,-,-,-)\)   & \(\varepsilon\) &                 \\
 452 &  \(516\) &      \(2^2\cdot 3\cdot 43\) & 1b &     \(5^2 2^4 3^4 43^4\) & \(12\) &   \(1\) &   \(2\) &   \(5\) & \(6\) & \((\times,-,-,-)\)   & \(\gamma\)      & \(2\cdot 3^2,2\cdot 5^4\) \\
\hline
\end{tabular}
\end{center}
\end{table}

\newpage

\renewcommand{\arraystretch}{1.0}

\begin{table}[ht]
\caption{\(45\) pure metacyclic fields with normalized radicands \(557\le D\le 604\)}
\label{tbl:PureQuinticFields600}
\begin{center}
\begin{tabular}{|r|rc|ccr|cccc|ccc|}
\hline
 No. &    \(D\) &                     Factors & S  &                  \(f^4\) &  \(m\) & \(V_L\) & \(V_M\) & \(V_N\) & \(E\) & \((1,2,4,5)\)        & T               & P               \\
\hline
 491 &  \(557\) &                             & 2  &                \(557^4\) &  \(1\) &   \(0\) &   \(0\) &   \(0\) & \(5\) & \((-,-,-,\otimes)\)  & \(\vartheta\)   &                 \\
 492 &  \(558\) &      \(2\cdot 3^2\cdot 31\) & 1b &     \(5^2 2^4 3^4 31^4\) & \(13\) &   \(1\) &   \(2\) &   \(5\) & \(6\) & \((-,-,\times,-)\)   & \(\gamma\)      & \(5,3\cdot 31\) \\
 493 &  \(559\) &              \(13\cdot 43\) & 1b &        \(5^2 13^4 43^4\) &  \(4\) &   \(0\) &   \(0\) &   \(1\) & \(6\) & \((\times,-,-,-)\)   & \(\gamma\)      &                 \\
 494 &  \(560\) &       \(2^4\cdot 5\cdot 7\) & 1a &          \(5^6 2^4 7^4\) & \(16\) &   \(0\) &   \(0\) &   \(1\) & \(6\) & \((\times,-,-,-)\)   & \(\gamma\)      &                 \\
 495 &  \(561\) &       \(3\cdot 11\cdot 17\) & 1b &    \(5^2 3^4 11^4 17^4\) & \(13\) &   \(1\) &   \(2\) &   \(5\) & \(6\) & \((-,-,\times,-)\)   & \(\gamma\)      & \(11\cdot 5^3,17\cdot 5^4\) \\
 496 &  \(562\) &              \(2\cdot 281\) & 1b &        \(5^2 2^4 281^4\) &  \(3\) &   \(1\) &   \(2\) &   \(4\) & \(5\) & \((-,-,\times,-)\)   & \(\varepsilon\) & \(2\cdot 5\)    \\
 497 &  \(563\) &                             & 1b &            \(5^2 563^4\) &  \(1\) &   \(0\) &   \(0\) &   \(0\) & \(5\) & \((\times,-,-,-)\)   & \(\varepsilon\) &                 \\
\hline
\end{tabular}
\end{center}
\end{table}

\newpage

\renewcommand{\arraystretch}{1.0}

\begin{table}[ht]
\caption{\(45\) pure metacyclic fields with normalized radicands \(605\le D\le 653\)}
\label{tbl:PureQuinticFields650}
\begin{center}
\begin{tabular}{|r|rc|ccr|cccc|ccc|}
\hline
 No. &    \(D\) &                     Factors & S  &                  \(f^4\) &  \(m\) & \(V_L\) & \(V_M\) & \(V_N\) & \(E\) & \((1,2,4,5)\)        & T               & P               \\
\hline
 536 &  \(605\) &             \(5\cdot 11^2\) & 1a &             \(5^6 11^4\) &  \(4\) &   \(1\) &   \(1\) &   \(2\) & \(3\) & \((-,-,\otimes,-)\)  & \(\alpha_2\)    & \(\mathcal{K},\mathfrak{K}_1\) \\
 537 & *\(606\) &       \(2\cdot 3\cdot 101\) & 1b &    \(5^2 2^4 3^4 101^4\) & \(12\) &   \(1\) &   \(2\) &   \(5\) & \(6\) & \((-,-,\times,-)\)   & \(\gamma\)      & \(2\cdot 5,101\) \\
 538 &  \(607\) &                             & 2  &                \(607^4\) &  \(1\) &   \(0\) &   \(0\) &   \(0\) & \(5\) & \((-,-,-,\otimes)\)  & \(\vartheta\)   &                 \\
 539 & *\(609\) &        \(3\cdot 7\cdot 29\) & 1b &     \(5^2 3^4 7^4 29^4\) & \(12\) &   \(2\) &   \(3\) &   \(7\) & \(4\) & \((-,\otimes,-,-)\)  & \(\beta_2\)     & \(7^2\cdot 5,\mathcal{K}\) \\
 540 &  \(610\) &        \(2\cdot 5\cdot 61\) & 1a &         \(5^6 2^4 61^4\) & \(16\) &   \(1\) &   \(1\) &   \(3\) & \(4\) & \((-,-,(\times),-)\) & \(\beta_2\)     & \(2\cdot 5^2,\mathcal{K}\) \\
 541 &  \(611\) &              \(13\cdot 47\) & 1b &        \(5^2 13^4 47^4\) &  \(3\) &   \(0\) &   \(0\) &   \(1\) & \(6\) & \((\times,-,-,-)\)   & \(\gamma\)      &                 \\
 542 &  \(612\) &    \(2^2\cdot 3^2\cdot 17\) & 1b &     \(5^2 2^4 3^4 17^4\) & \(13\) &   \(1\) &   \(2\) &   \(5\) & \(6\) & \((\times,-,-,-)\)   & \(\gamma\)      & \(17,3^2\cdot 5\) \\
\hline
\end{tabular}
\end{center}
\end{table}

\newpage

\renewcommand{\arraystretch}{1.0}

\begin{table}[ht]
\caption{\(44\) pure metacyclic fields with normalized radicands \(654\le D\le 701\)}
\label{tbl:PureQuinticFields700}
\begin{center}
\begin{tabular}{|r|rc|ccr|cccc|ccc|}
\hline
 No. &    \(D\) &                     Factors & S  &                  \(f^4\) &  \(m\) & \(V_L\) & \(V_M\) & \(V_N\) & \(E\) & \((1,2,4,5)\)        & T               & P               \\
\hline
 581 &  \(654\) &       \(2\cdot 3\cdot 109\) & 1b &    \(5^2 2^4 3^4 109^4\) & \(13\) &   \(1\) &   \(2\) &   \(5\) & \(6\) & \((-,\times,-,-)\)   & \(\gamma\)      & \(3\cdot 5,109\) \\
 582 &  \(655\) &              \(5\cdot 131\) & 1a &            \(5^6 131^4\) &  \(4\) &   \(1\) &   \(1\) &   \(2\) & \(3\) & \((-,-,\otimes,-)\)  & \(\alpha_2\)    & \(\mathcal{K},\mathfrak{K}_1\) \\
 583 &  \(656\) &             \(2^4\cdot 41\) & 1b &         \(5^2 2^4 41^4\) &  \(3\) &   \(2\) &   \(2\) &   \(4\) & \(1\) & \((-,-,\otimes,-)\)  & \(\alpha_2\)    & \(\mathcal{K},\mathfrak{K}_1\) \\
 584 &  \(657\) &             \(3^2\cdot 73\) & 2  &             \(3^4 73^4\) &  \(1\) &   \(0\) &   \(0\) &   \(0\) & \(5\) & \((\times,-,-,-)\)   & \(\varepsilon\) &                 \\
 585 &  \(658\) &        \(2\cdot 7\cdot 47\) & 1b &     \(5^2 2^4 7^4 47^4\) & \(12\) &   \(1\) &   \(2\) &   \(5\) & \(6\) & \((\times,-,-,-)\)   & \(\gamma\)      & \(2\cdot 5,47\cdot 5\) \\
 586 &  \(659\) &                             & 1b &            \(5^2 659^4\) &  \(1\) &   \(1\) &   \(1\) &   \(2\) & \(3\) & \((-,\otimes,-,-)\)  & \(\delta_2\)    & \(\mathcal{L}\) \\
 587 & *\(660\) &\(2^2\cdot 3\cdot 5\cdot 11\)& 1a &     \(5^6 2^4 3^4 11^4\) & \(64\) &   \(1\) &   \(2\) &   \(5\) & \(6\) & \((-,-,\times,-)\)   & \(\gamma\)      & \(2\cdot 5,5\cdot 11\) \\
\hline
\end{tabular}
\end{center}
\end{table}

\newpage

\renewcommand{\arraystretch}{1.0}

\begin{table}[ht]
\caption{\(47\) pure metacyclic fields with normalized radicands \(702\le D\le 753\)}
\label{tbl:PureQuinticFields750}
\begin{center}
\begin{tabular}{|r|rc|ccr|cccc|ccc|}
\hline
 No. &    \(D\) &                     Factors & S  &                  \(f^4\) &  \(m\) & \(V_L\) & \(V_M\) & \(V_N\) & \(E\) & \((1,2,4,5)\)        & T               & P               \\
\hline
 625 & *\(702\) &      \(2\cdot 3^3\cdot 13\) & 1b &     \(5^2 2^4 3^4 13^4\) & \(13\) &   \(2\) &   \(4\) &   \(8\) & \(5\) & \((\times,-,-,-)\)   & \(\varepsilon\) & \(13\cdot 5^4\) \\
 626 &  \(703\) &              \(19\cdot 37\) & 1b &        \(5^2 19^4 37^4\) &  \(3\) &   \(1\) &   \(1\) &   \(3\) & \(4\) & \((-,\otimes,-,-)\)  & \(\beta_2\)     & \(19\cdot 5^2,\mathcal{K}\) \\
 627 &  \(705\) &        \(3\cdot 5\cdot 47\) & 1a &         \(5^6 3^4 47^4\) & \(16\) &   \(0\) &   \(0\) &   \(1\) & \(6\) & \((\times,-,-,-)\)   & \(\gamma\)      &                 \\
 628 &  \(706\) &              \(2\cdot 353\) & 1b &        \(5^2 2^4 353^4\) &  \(3\) &   \(0\) &   \(0\) &   \(1\) & \(6\) & \((\times,-,-,-)\)   & \(\gamma\)      &                 \\
 629 & *\(707\) &              \(7\cdot 101\) & 2  &            \(7^4 101^4\) &  \(4\) &   \(1\) &   \(1\) &   \(3\) & \(4\) & \((-,-,(\times),\otimes)\) & \(\zeta_2\) & \(\mathcal{K}\) \\
 630 &  \(708\) &      \(2^2\cdot 3\cdot 59\) & 1b &     \(5^2 2^4 3^4 59^4\) & \(13\) &   \(1\) &   \(2\) &   \(5\) & \(6\) & \((-,\times,-,-)\)   & \(\gamma\)      & \(3\cdot 5,59\cdot 5\) \\
 631 &  \(709\) &                             & 1b &            \(5^2 709^4\) &  \(1\) &   \(1\) &   \(1\) &   \(2\) & \(3\) & \((-,\otimes,-,-)\)  & \(\delta_2\)    & \(\mathcal{L}\) \\
\hline
\end{tabular}
\end{center}
\end{table}

\newpage

\renewcommand{\arraystretch}{1.0}

\begin{table}[ht]
\caption{\(44\) pure metacyclic fields with normalized radicands \(754\le D\le 799\)}
\label{tbl:PureQuinticFields800}
\begin{center}
\begin{tabular}{|r|rc|ccr|cccc|ccc|}
\hline
 No. &    \(D\) &                     Factors & S  &                  \(f^4\) &  \(m\) & \(V_L\) & \(V_M\) & \(V_N\) & \(E\) & \((1,2,4,5)\)        & T               & P               \\
\hline
 672 &  \(754\) &       \(2\cdot 13\cdot 29\) & 1b &    \(5^2 2^4 13^4 29^4\) & \(13\) &   \(1\) &   \(2\) &   \(5\) & \(6\) & \((-,\times,-,-)\)   & \(\gamma\)      & \(13\cdot 5,29\cdot 5\) \\
 673 &  \(755\) &              \(5\cdot 151\) & 1a &            \(5^6 151^4\) &  \(4\) &   \(1\) &   \(1\) &   \(3\) & \(4\) & \((-,-,(\times),\otimes)\) & \(\zeta_2\) & \(\mathcal{K}\) \\
 674 &  \(756\) &     \(2^2\cdot 3^3\cdot 7\) & 1b &      \(5^2 2^4 3^4 7^4\) & \(12\) &   \(1\) &   \(2\) &   \(5\) & \(6\) & \((\times,-,-,-)\)   & \(\gamma\)      & \(7,2^3\cdot 5^2\) \\
 675 &  \(757\) &                             & 2  &                \(757^4\) &  \(1\) &   \(0\) &   \(0\) &   \(0\) & \(5\) & \((-,-,-,\otimes)\)  & \(\vartheta\)   &                 \\
 676 &  \(758\) &              \(2\cdot 379\) & 1b &        \(5^2 2^4 379^4\) &  \(3\) &   \(1\) &   \(1\) &   \(3\) & \(4\) & \((-,\otimes,-,-)\)  & \(\beta_2\)     & \(2\cdot 5^2,\mathcal{K}\) \\
 677 &  \(759\) &       \(3\cdot 11\cdot 23\) & 1b &    \(5^2 3^4 11^4 23^4\) & \(13\) &   \(2\) &   \(3\) &   \(6\) & \(3\) & \((-,-,\otimes,-)\)  & \(\beta_1\)     & \(3^2\cdot 23\cdot 5,\mathfrak{K}_2\) \\
 678 &  \(760\) &      \(2^3\cdot 5\cdot 19\) & 1a &         \(5^6 2^4 19^4\) & \(16\) &   \(1\) &   \(1\) &   \(3\) & \(4\) & \((-,\otimes,-,-)\)  & \(\beta_2\)     & \(2^4\cdot 19,\mathcal{K}\) \\
\hline
\end{tabular}
\end{center}
\end{table}

\newpage

\renewcommand{\arraystretch}{1.0}

\begin{table}[ht]
\caption{\(46\) pure metacyclic fields with normalized radicands \(801\le D\le 848\)}
\label{tbl:PureQuinticFields850}
\begin{center}
\begin{tabular}{|r|rc|ccr|cccc|ccc|}
\hline
 No. &    \(D\) &                     Factors & S  &                  \(f^4\) &  \(m\) & \(V_L\) & \(V_M\) & \(V_N\) & \(E\) & \((1,2,4,5)\)        & T               & P               \\
\hline
 716 &  \(801\) &             \(3^2\cdot 89\) & 2  &             \(3^4 89^4\) &  \(1\) &   \(1\) &   \(1\) &   \(2\) & \(3\) & \((-,\otimes,-,-)\)  & \(\delta_2\)    & \(\mathcal{K}\) \\
 717 &  \(802\) &              \(2\cdot 401\) & 1b &        \(5^2 2^4 401^4\) &  \(4\) &   \(1\) &   \(1\) &   \(3\) & \(4\) & \((-,-,(\times),-)\) & \(\beta_2\)     & \(2^2\cdot 5,\mathcal{K}\) \\
 718 &  \(803\) &              \(11\cdot 73\) & 1b &        \(5^2 11^4 73^4\) &  \(3\) &   \(1\) &   \(1\) &   \(3\) & \(4\) & \((-,-,(\times),-)\) & \(\beta_2\)     & \(73\cdot 5^2,\mathcal{K}\) \\
 719 &  \(804\) &      \(2^2\cdot 3\cdot 67\) & 1b &     \(5^2 2^4 3^4 67^4\) & \(13\) &   \(1\) &   \(2\) &   \(5\) & \(6\) & \((\times,-,-,-)\)   & \(\gamma\)      & \(3,2\cdot 5^3\) \\
 720 &  \(805\) &        \(5\cdot 7\cdot 23\) & 1a &         \(5^6 7^4 23^4\) & \(16\) &   \(0\) &   \(0\) &   \(1\) & \(6\) & \((\times,-,-,-)\)   & \(\gamma\)      &                 \\
 721 &  \(806\) &       \(2\cdot 13\cdot 31\) & 1b &    \(5^2 2^4 13^4 31^4\) & \(13\) &   \(2\) &   \(3\) &   \(6\) & \(3\) & \((-,-,\otimes,-)\)  & \(\beta_1\)     & \(31\cdot 5^2,\mathfrak{K}_2\) \\
 722 &  \(807\) &              \(3\cdot 269\) & 2  &            \(3^4 269^4\) &  \(1\) &   \(1\) &   \(1\) &   \(2\) & \(3\) & \((-,\otimes,-,-)\)  & \(\delta_2\)    & \(\mathcal{K}\) \\
\hline
\end{tabular}
\end{center}
\end{table}

\newpage

\renewcommand{\arraystretch}{1.0}

\begin{table}[ht]
\caption{\(47\) pure metacyclic fields with normalized radicands \(849\le D\le 901\)}
\label{tbl:PureQuinticFields900}
\begin{center}
\begin{tabular}{|r|rc|ccr|cccc|ccc|}
\hline
 No. &    \(D\) &                     Factors & S  &                  \(f^4\) &  \(m\) & \(V_L\) & \(V_M\) & \(V_N\) & \(E\) & \((1,2,4,5)\)        & T               & P               \\
\hline
 762 &  \(849\) &              \(3\cdot 283\) & 2  &            \(3^4 283^4\) &  \(1\) &   \(0\) &   \(0\) &   \(0\) & \(5\) & \((\times,-,-,-)\)   & \(\varepsilon\) &                 \\
 763 &  \(850\) &      \(2\cdot 5^2\cdot 17\) & 1a &         \(5^6 2^4 17^4\) & \(16\) &   \(0\) &   \(0\) &   \(1\) & \(6\) & \((\times,-,-,-)\)   & \(\gamma\)      &                 \\
 764 &  \(851\) &              \(23\cdot 37\) & 2  &            \(23^4 37^4\) &  \(1\) &   \(0\) &   \(0\) &   \(0\) & \(5\) & \((\times,-,-,-)\)   & \(\varepsilon\) &                 \\
 765 &  \(852\) &      \(2^2\cdot 3\cdot 71\) & 1b &     \(5^2 2^4 3^4 71^4\) & \(13\) &   \(1\) &   \(2\) &   \(5\) & \(6\) & \((-,-,\times,-)\)   & \(\gamma\)      & \(2\cdot 5^3,3\cdot 5^3\) \\
 766 &  \(853\) &                             & 1b &            \(5^2 853^4\) &  \(1\) &   \(0\) &   \(0\) &   \(0\) & \(5\) & \((\times,-,-,-)\)   & \(\varepsilon\) &                 \\
 767 &  \(854\) &        \(2\cdot 7\cdot 61\) & 1b &     \(5^2 2^4 7^4 61^4\) & \(12\) &   \(1\) &   \(2\) &   \(5\) & \(6\) & \((-,-,\times,-)\)   & \(\gamma\)      & \(2\cdot 5^2,61\) \\
 768 &  \(855\) &      \(3^2\cdot 5\cdot 19\) & 1a &         \(5^6 3^4 19^4\) & \(16\) &   \(1\) &   \(1\) &   \(3\) & \(4\) & \((-,\otimes,-,-)\)  & \(\beta_2\)     & \(3,\mathcal{K}\) \\
\hline
\end{tabular}
\end{center}
\end{table}

\newpage

\renewcommand{\arraystretch}{1.0}

\begin{table}[ht]
\caption{\(47\) pure metacyclic fields with normalized radicands \(902\le D\le 949\)}
\label{tbl:PureQuinticFields950}
\begin{center}
\begin{tabular}{|r|rc|ccr|cccc|ccc|}
\hline
 No. &    \(D\) &                     Factors & S  &                  \(f^4\) &  \(m\) & \(V_L\) & \(V_M\) & \(V_N\) & \(E\) & \((1,2,4,5)\)        & T               & P               \\
\hline
 809 & *\(902\) &       \(2\cdot 11\cdot 41\) & 1b &    \(5^2 2^4 11^4 41^4\) & \(13\) &   \(2\) &   \(3\) &   \(7\) & \(4\) & \((-,-,(\times),-)\) & \(\beta_2\) & \(11\cdot 5,\mathcal{K}_{(11)}\cdot\mathcal{K}_{(41)}\) \\
 810 &  \(903\) &        \(3\cdot 7\cdot 43\) & 1b &     \(5^2 3^4 7^4 43^4\) & \(16\) &   \(1\) &   \(2\) &   \(5\) & \(6\) & \((\times,-,-,-)\)   & \(\gamma\)      & \(5,7^2\cdot 43\) \\
 811 &  \(904\) &            \(2^3\cdot 113\) & 1b &        \(5^2 2^4 113^4\) &  \(3\) &   \(0\) &   \(0\) &   \(1\) & \(6\) & \((\times,-,-,-)\)   & \(\gamma\)      &                 \\
 812 &  \(905\) &              \(5\cdot 181\) & 1a &            \(5^6 181^4\) &  \(4\) &   \(1\) &   \(1\) &   \(2\) & \(3\) & \((-,-,\otimes,-)\)  & \(\alpha_2\)    & \(\mathcal{K},\mathfrak{K}_1\) \\
 813 &  \(906\) &       \(2\cdot 3\cdot 151\) & 1b &    \(5^2 2^4 3^4 151^4\) & \(12\) &   \(1\) &   \(2\) &   \(5\) & \(6\) & \((-,-,\times,-)\)   & \(\gamma\)      & \(2,3\cdot 5^2\) \\
 814 &  \(907\) &                             & 2  &                \(907^4\) &  \(1\) &   \(0\) &   \(0\) &   \(0\) & \(5\) & \((-,-,-,\otimes)\)  & \(\vartheta\)   &                 \\
 815 &  \(908\) &            \(2^2\cdot 227\) & 1b &        \(5^2 2^4 227^4\) &  \(3\) &   \(0\) &   \(0\) &   \(1\) & \(6\) & \((\times,-,-,-)\)   & \(\gamma\)      &                 \\
\hline
\end{tabular}
\end{center}
\end{table}

\newpage

\renewcommand{\arraystretch}{1.0}

\begin{table}[ht]
\caption{\(45\) pure metacyclic fields with normalized radicands \(950\le D\le 999\)}
\label{tbl:PureQuinticFields1000}
\begin{center}
\begin{tabular}{|r|rc|ccr|cccc|ccc|}
\hline
 No. &    \(D\) &                     Factors & S  &                  \(f^4\) &  \(m\) & \(V_L\) & \(V_M\) & \(V_N\) & \(E\) & \((1,2,4,5)\)        & T               & P               \\
\hline
 856 &  \(950\) &      \(2\cdot 5^2\cdot 19\) & 1a &         \(5^6 2^4 19^4\) & \(16\) &   \(1\) &   \(1\) &   \(3\) & \(4\) & \((-,\otimes,-,-)\)  & \(\beta_2\)     & \(5,\mathcal{K}\) \\
 857 &  \(951\) &              \(3\cdot 317\) & 2  &            \(3^4 317^4\) &  \(1\) &   \(0\) &   \(0\) &   \(0\) & \(5\) & \((\times,-,-,-)\)   & \(\varepsilon\) &                 \\
 858 &  \(952\) &      \(2^3\cdot 7\cdot 17\) & 1b &     \(5^2 2^4 7^4 17^4\) & \(12\) &   \(1\) &   \(2\) &   \(5\) & \(6\) & \((\times,-,-,-)\)   & \(\gamma\)      & \(2\cdot 5^4,7\cdot 5^4\) \\
 859 &  \(953\) &                             & 1b &            \(5^2 953^4\) &  \(1\) &   \(0\) &   \(0\) &   \(0\) & \(5\) & \((\times,-,-,-)\)   & \(\varepsilon\) &                 \\
 860 &  \(954\) &      \(2\cdot 3^2\cdot 53\) & 1b &     \(5^2 2^4 3^4 53^4\) & \(13\) &   \(1\) &   \(2\) &   \(5\) & \(6\) & \((\times,-,-,-)\)   & \(\gamma\)      & \(3\cdot 5,2\cdot 5^2\) \\
 861 & *\(955\) &              \(5\cdot 191\) & 1a &            \(5^6 191^4\) &  \(4\) &  \(2*\) &  \(3*\) &  \(6*\) & \(3\) & \((-,-,\otimes,-)\)  & \(\alpha_2\)    & \(\mathcal{K},\mathfrak{K}_1\) \\
 862 &  \(956\) &            \(2^2\cdot 239\) & 1b &        \(5^2 2^4 239^4\) &  \(3\) &   \(1\) &   \(1\) &   \(3\) & \(4\) & \((-,\otimes,-,-)\)  & \(\beta_2\)     & \(239,\mathcal{L}\) \\
\hline
\end{tabular}
\end{center}
\end{table}

\newpage

\section{Statistical evaluation, refinements, and conclusions}
\label{s:Conclusions}

\subsection{Statistics of DPF types}
\label{ss:Statistics}
\noindent
The complete statistical evaluation of the preceding twenty Tables
\ref{tbl:PureQuinticFields50}
---
\ref{tbl:PureQuinticFields1000}
is given in Table
\ref{tbl:Statistics}.
The first ten columns show the absolute frequencies
of pure metacyclic fields \(N=\mathbb{Q}(\zeta,\sqrt[5]{D})\)
with various DPF types
for the ranges \(2\le D<n\cdot 100\) with \(1\le n\le 10\).
The eleventh column lists the relative percentages
of the five most frequent DPF types
for the complete range \(2\le D<10^3\) of normalized radicands.

Among our \(13\) differential principal factorization types,
type \(\gamma\) with \(3\)-dimensional absolute principal factorization, \(A=3\),
is clearly dominating with more than one third \((36\,\%)\) of all occurrences
in the complete range \(2\le D<10^3\),
followed by type \(\varepsilon\) with \(2\)-dimensional absolute principal factorization, \(A=2\),
which covers nearly one quarter \((23\,\%)\) of all cases.
The third place (nearly \(18\,\%\)) is occupied by type \(\beta_2\) with mixed 
absolute and intermediate principal factorization, \(A=2\), \(I=1\).

It is striking that
type \(\alpha_1\) with \(2\)-dimensional relative principal factorization, \(R=2\), and
type \(\alpha_3\) with \(2\)-dimensional intermediate principal factorization, \(I=2\),
are populated rather sparsely,
in favour of a remarkable contribution by type \(\alpha_2\) with mixed
intermediate and relative principal factorization, \(I=R=1\)
(place four with \(8\,\%\)).

The appearance of the four types \(\zeta_1\), \(\zeta_2\), \(\eta\), \(\vartheta\)
with norm representation \(N_{N/K}(Z)=\zeta\), \(Z\in U_N\),
of the primitive fifth root of unity \(\zeta=\zeta_5\)
is marginal
(\cite[Thm. 8.2]{Ma2a}),
in spite of the parametrized contribution
by all prime conductors \(f=q\equiv\pm 7\,(\mathrm{mod}\,25)\) to type \(\vartheta\),
as we shall prove in Theorem
\ref{thm:InfSimCls} (1)
in \S\
\ref{ss:Theorems}.


\renewcommand{\arraystretch}{1.0}

\begin{table}[ht]
\caption{Absolute frequencies of differential principal factorization types}
\label{tbl:Statistics}
\begin{center}
\begin{tabular}{|c|rrrrrrrrrr|r|}
\hline
 Type            &\(100\) &\(200\) &\(300\) &\(400\) &\(500\) &\(600\) &\(700\) &\(800\) &\(900\) &\(1000\)&   \(\%\) \\
\hline
 \(\alpha_1\)    &  \(1\) &  \(2\) &  \(3\) &  \(4\) &  \(5\) &  \(5\) &  \(5\) &  \(9\) &  \(9\) &  \(9\) & \\
 \(\alpha_2\)    & \(10\) & \(17\) & \(23\) & \(30\) & \(35\) & \(42\) & \(52\) & \(57\) & \(63\) & \(75\) &  \(8.3\) \\
 \(\alpha_3\)    &  \(0\) &  \(0\) &  \(0\) &  \(1\) &  \(1\) &  \(3\) &  \(5\) &  \(5\) &  \(7\) &  \(8\) & \\
 \(\beta_1\)     &  \(0\) &  \(2\) &  \(4\) &  \(7\) &  \(8\) & \(11\) & \(15\) & \(18\) & \(22\) & \(23\) & \\
 \(\beta_2\)     &  \(7\) & \(24\) & \(40\) & \(54\) & \(80\) & \(94\) &\(108\) &\(126\) &\(146\) &\(161\) & \(17.9\) \\
 \(\gamma\)      & \(25\) & \(55\) & \(88\) &\(117\) &\(148\) &\(187\) &\(222\) &\(259\) &\(290\) &\(324\) & \(36.0\) \\
\hline
 \(\delta_1\)    &  \(0\) &  \(0\) &  \(1\) &  \(1\) &  \(3\) &  \(4\) &  \(4\) &  \(4\) &  \(6\) &  \(7\) & \\
 \(\delta_2\)    &  \(8\) & \(14\) & \(19\) & \(23\) & \(31\) & \(35\) & \(38\) & \(44\) & \(51\) & \(53\) &  \(5.9\) \\
 \(\varepsilon\) & \(26\) & \(45\) & \(67\) & \(95\) &\(110\) &\(128\) &\(150\) &\(165\) &\(184\) &\(208\) & \(23.1\) \\
\hline
 \(\zeta_1\)     &  \(0\) &  \(1\) &  \(1\) &  \(1\) &  \(1\) &  \(1\) &  \(1\) &  \(1\) &  \(1\) &  \(1\) & \\
 \(\zeta_2\)     &  \(0\) &  \(0\) &  \(0\) &  \(0\) &  \(0\) &  \(1\) &  \(1\) &  \(4\) &  \(4\) &  \(5\) & \\
 \(\eta\)        &  \(1\) &  \(2\) &  \(4\) &  \(5\) &  \(5\) &  \(6\) &  \(6\) &  \(6\) &  \(6\) &  \(7\) & \\
\hline
 \(\vartheta\)   &  \(3\) &  \(6\) &  \(8\) &  \(9\) & \(11\) & \(13\) & \(15\) & \(17\) & \(18\) & \(19\) & \\
\hline
 Total           & \(81\) &\(168\) &\(258\) &\(347\) &\(438\) &\(530\) &\(622\) &\(715\) &\(807\) &\(900\) &\(100.0\) \\
\hline
\end{tabular}
\end{center}
\end{table}


\subsection{Similarity classes and prototypes}
\label{ss:Prototypes}
\noindent
In
\cite{Ma2b},
we came to the conviction
that for deeper insight into the arithmetical structure of the fields under investigation,
the prime factorization of the class field theoretic conductor \(f\)
of the abelian extension \(N/K\) over the cyclotomic field \(K=\mathbb{Q}(\zeta)\)
and the primary invariants of all involved \(5\)-class groups
must be taken in consideration.
These ideas have lead to the concept of
\textit{similarity classes} and representative \textit{prototypes},
which refines the differential principal factorization (DPF) types


Let \(t\) be the number of primes \(q_1,\ldots,q_t\in\mathbb{P}\)
distinct from \(5\) which divide the conductor \(f\).
Among these prime numbers, we separately count
\(u:=\#\lbrace 1\le i\le t\mid q_i\equiv\pm 1,\pm 7\,(\mathrm{mod}\,25)\rbrace\) free primes,
\(v:=t-u\) restrictive primes,
\(s_2:=\#\lbrace 1\le i\le t\mid q_i\equiv -1\,(\mathrm{mod}\,5)\rbrace\) \(2\)-split primes, and
\(s_4:=\#\lbrace 1\le i\le t\mid q_i\equiv +1\,(\mathrm{mod}\,5)\rbrace\) \(4\)-split primes.
The multiplicity \(m=m(f)\) is given in terms of \(t,u,v\), according to \S\
\ref{s:Formulas},
and the dimensions of various spaces of primitive ambiguous ideals over the finite field \(\mathbb{F}_5\)
are given in terms of \(t,s_2,s_4\), according to
\cite[\S\ 4]{Ma2a}.
By \(\eta=\frac{1}{2}(1+\sqrt{5})\) we denote the fundamental unit of \(K^+=\mathbb{Q}(\sqrt{5})\).
The dimensions of the spaces of absolute, intermediate and relative DPF
over \(\mathbb{F}_5\) are denoted by \(A\), \(I\) and \(R\),
identical with the additive (logarithmic) version in
\cite[Thm. 6.1]{Ma2a}.
Further, let \(M=\mathbb{Q}(\sqrt{5},\sqrt[5]{D})\) be the maximal real intermediate field of \(N/L\),
and denote by \(U_0\) the subgroup of the unit group \(U_N\) of \(N=\mathbb{Q}(\zeta,\sqrt[5]{D})\)
generated by the units of all conjugate fields of \(L=\mathbb{Q}(\sqrt[5]{D})\) and of \(K=\mathbb{Q}(\zeta)\),
where \(\zeta=\zeta_5\) is a primitive fifth root of unity.
For a number field \(F\),
let \(V_F:=v_5(\#\mathrm{Cl}(F))\)
be the \(5\)-valuation of the class number of \(F\).


\begin{definition}
\label{dfn:QuinticSimilarity}
A set of normalized fifth power free radicands \(D>1\)
is called a \textit{similarity class}
if the associated pure quintic fields \(L=\mathbb{Q}(\sqrt[5]{D})\)
share the following common multiplets of invariants:
\begin{itemize}
\item
the refined Dedekind species
\((e_0;t,u,v,m;n,s_2,s_4)\), where
\begin{equation}
\label{eqn:Dedekind5}
f^4=5^{e_0}\cdot q_1^4\cdots q_t^4 \quad \text{ with } \quad
e_0\in\lbrace 0,2,6\rbrace,\ t\ge 0,\ n=t-s_2-s_4,
\end{equation}
\item
the differential principal factorization type
\((U,\eta,\zeta;A,I,R)\), where
\begin{equation}
\label{eqn:AIR}
(U_K:N_{N/K}(U_N))=5^U \quad \text{ and } \quad U+1=A+I+R,
\end{equation}
\item
the structure of the \(5\)-class groups
\((V_L,V_M,V_N;E)\), where
\begin{equation}
\label{eqn:5ClsGrp}
(U_N:U_0)=5^E \quad \text{ and } \quad V_N=4\cdot V_L+E-5.
\end{equation}
\end{itemize}
\end{definition}


\begin{warning}
\label{wrn:QuinticSimilarity}
To reduce the number of invariants,
we abstain from defining additional counters
\(s_2^\prime:=\#\lbrace 1\le i\le t\mid q_i\equiv -1\,(\mathrm{mod}\,25)\rbrace\) and
\(s_4^\prime:=\#\lbrace 1\le i\le t\mid q_i\equiv +1\,(\mathrm{mod}\,25)\rbrace\)
for \textit{free} splitting prime divisors of the conductor \(f\).
However, we point out that occasionally a similarity class in the sense of Definition
\ref{dfn:QuinticSimilarity}
will be split in two separate classes, having the same invariants,
but distinct contributions to the counters \(u\) and \(s_4\), resp. \(s_2\).
For instance, the similarity classes \(\mathbf{\lbrack 77\rbrack}\) and \(\mathbf{\lbrack 202\rbrack}\)
with \(77=7\cdot 11\) and \(202=2\cdot 101\) share identical multiplets of invariants
\((e_0;t,u,v,m;n,s_2,s_4)=(2;2,1,1,4;1,0,1)\) (species \(1\mathrm{b}\)), \((U,\eta,\zeta;A,I,R)=(2,-,-;2,1,0)\) (type \(\beta_2\)), and \((V_L,V_M,V_N;E)=(1,1,3;4)\).
But \(u=1\) and \(n=1\) are due to \(7\), \(v=1\) and \(s_4=1\) are due to \(11\), in the former case,
whereas \(v=1\) and \(n=1\) are due to \(2\), \(u=1\) and \(s_4=1\) are due to \(101\), in the latter case.
Therefore, the contributions by primes congruent to \(\pm 1\,(\mathrm{mod}\,25)\)
will be indicated by writing \(u=1'\) and \(s_4=1'\), resp. \(s_2=1'\).
   
We also emphasize that in the rare cases of non-elementary \(5\)-class groups,
the actual structures (abelian type invariants) of the \(5\)-class groups will be taken into account,
and not only the \(5\)-valuations \(V_L,V_M,V_N\).
\end{warning}


\begin{definition}
\label{dfn:Prototype}
The minimal element \(\mathbf{M}\) of a similarity class
(with respect to the natural order of positive integers \(\mathbb{N}\))
is called the representative \textit{prototype} of the class,
which is denoted by writing its prototype in square brackets \(\mathbf{\lbrack M\rbrack}\).

\end{definition}


The remaining elements of a similarity class,
which are bigger than the prototype,
only reproduce the arithmetical invariants of the prototype
and do not provide any additional information,
exept possibly about \textit{other primary components} of the class groups,
that is the structure of \(\ell\)-class groups \(\mathrm{Cl}_{\ell}(F)\)
of the fields \(F\in\lbrace L,M,N\rbrace\) for \(\ell\in\mathbb{P}\setminus\lbrace 5\rbrace\).

Whereas there are only \(13\) DPF types of pure quintic fields,
the \textit{number of similarity classes} is obviously infinite,
since firstly the number \(t\) of primes dividing the conductor is unbounded
and secondly the number of \textit{states}, defined by the triplet \((V_L,V_M,V_N)\)
of \(5\)-valuations of class numbers, is also unlimited.

Given a fixed refined Dedekind species \((e_0;t,u,v,m;n,s_2,s_4)\),
the set of all associated normalized fifth power free radicands \(D\)
usually splits into several similarity classes defined by
distinct DPF types (\textit{type splitting}).
Occasionally it even splits further into different structures of \(5\)-class groups,
called \textit{states},
with increasing complexity of abelian type invariants (\textit{state splitting}).

The \(134\) prototypes \(2\le\mathbf{M}<10^3\) of pure quintic fields
are listed in the Tables
\ref{tbl:Prototypes1},
\ref{tbl:Prototypes2},
\ref{tbl:Prototypes3}
and
\ref{tbl:Prototypes4}.
By \(\mathbf{\lvert M\rvert}:=\#\lbrace D\in\mathbf{\lbrack M\rbrack}\mid D<B\rbrace\) we denote the number of elements
of the similarity class \(\mathbf{\lbrack M\rbrack}\)
defined by the prototype \(\mathbf{M}\),
truncated at the upper bound \(B:=10^3\)
of our systematic investigations.

\newpage

\renewcommand{\arraystretch}{1.0}

\begin{table}[ht]
\caption{\(46\) prototypes \(2\le\mathbf{M}<200\) of pure metacyclic fields}
\label{tbl:Prototypes1}
\begin{center}
\begin{tabular}{|r|rc|ccr|cccc|ccc|r|}
\hline
 No. & \(\mathbf{M}\) &          Factors & S  &              \(f^4\) &  \(m\) & \(V_L\) & \(V_M\) & \(V_N\) & \(E\) & \((1,2,4,5)\)        & T               & P & \(\mathbf{\lvert M\rvert}\) \\
\hline
   1 &   \(2\) &                         & 1b &          \(5^2 2^4\) &  \(1\) &   \(0\) &   \(0\) &   \(0\) & \(5\) & \((\times,-,-,-)\)   & \(\varepsilon\) &                 & \(71\) \\
   2 &   \(5\) &                         & 1a &              \(5^6\) &  \(1\) &   \(0\) &   \(0\) &   \(0\) & \(5\) & \((-,-,-,\otimes)\)  & \(\vartheta\)   &                 &  \(1\) \\
   3 &   \(6\) &            \(2\cdot 3\) & 1b &      \(5^2 2^4 3^4\) &  \(3\) &   \(0\) &   \(0\) &   \(1\) & \(6\) & \((\times,-,-,-)\)   & \(\gamma\)      &                 & \(77\) \\
   4 &   \(7\) &                         & 2  &              \(7^4\) &  \(1\) &   \(0\) &   \(0\) &   \(0\) & \(5\) & \((-,-,-,\otimes)\)  & \(\vartheta\)   &                 & \(18\) \\
   5 &  \(10\) &            \(2\cdot 5\) & 1a &          \(5^6 2^4\) &  \(4\) &   \(0\) &   \(0\) &   \(0\) & \(5\) & \((\times,-,-,-)\)   & \(\varepsilon\) &                 & \(31\) \\
   6 &  \(11\) &                         & 1b &         \(5^2 11^4\) &  \(1\) &   \(1\) &   \(1\) &   \(2\) & \(3\) & \((-,-,\otimes,-)\)  & \(\alpha_2\)    & \(\mathcal{L},\mathfrak{K}_1\) & \(14\) \\
   7 &  \(14\) &            \(2\cdot 7\) & 1b &      \(5^2 2^4 7^4\) &  \(4\) &   \(0\) &   \(0\) &   \(1\) & \(6\) & \((\times,-,-,-)\)   & \(\gamma\)      &                 & \(44\) \\
   8 &  \(18\) &          \(2\cdot 3^2\) & 2  &          \(2^4 3^4\) &  \(1\) &   \(0\) &   \(0\) &   \(0\) & \(5\) & \((\times,-,-,-)\)   & \(\varepsilon\) &                 & \(37\) \\
   9 &  \(19\) &                         & 1b &         \(5^2 19^4\) &  \(1\) &   \(1\) &   \(1\) &   \(2\) & \(3\) & \((-,\otimes,-,-)\)  & \(\delta_2\)    & \(\mathcal{L}\) & \(27\) \\
  10 &  \(22\) &           \(2\cdot 11\) & 1b &     \(5^2 2^4 11^4\) &  \(3\) &   \(1\) &   \(1\) &   \(3\) & \(4\) & \((-,-,(\times),-)\) & \(\beta_2\)     & \(2\cdot 5,\mathcal{K}\) & \(35\) \\
  11 &  \(30\) &     \(2\cdot 3\cdot 5\) & 1a &      \(5^6 2^4 3^4\) & \(16\) &   \(0\) &   \(0\) &   \(1\) & \(6\) & \((\times,-,-,-)\)   & \(\gamma\)      &                 & \(37\) \\
  12 &  \(31\) &                         & 1b &         \(5^2 31^4\) &  \(1\) &   \(2\) &   \(3\) &   \(5\) & \(2\) & \((-,-,\otimes,-)\)  & \(\alpha_1\)    & \(\mathfrak{K}_1,\mathfrak{K}_2\) & \(2\) \\
  13 &  \(33\) &           \(3\cdot 11\) & 1b &     \(5^2 3^4 11^4\) &  \(3\) &   \(2\) &   \(2\) &   \(4\) & \(1\) & \((-,-,\otimes,-)\)  & \(\alpha_2\)    & \(\mathcal{K},\mathfrak{K}_2\) & \(8\) \\
  14 &  \(35\) &            \(5\cdot 7\) & 1a &          \(5^6 7^4\) &  \(4\) &   \(0\) &   \(0\) &   \(1\) & \(6\) & \((-,-,-,\otimes)\)  & \(\eta\)        &                 &  \(6\) \\
  15 &  \(38\) &           \(2\cdot 19\) & 1b &     \(5^2 2^4 19^4\) &  \(3\) &   \(1\) &   \(1\) &   \(3\) & \(4\) & \((-,\otimes,-,-)\)  & \(\beta_2\)     & \(5,\mathcal{K}\) & \(44\) \\
  16 &  \(42\) &     \(2\cdot 3\cdot 7\) & 1b &  \(5^2 2^4 3^4 7^4\) & \(12\) &   \(1\) &   \(2\) &   \(5\) & \(6\) & \((\times,-,-,-)\)   & \(\gamma\)      & \(2\cdot 5,3\cdot 5^2\) & \(22\) \\
  17 &  \(55\) &           \(5\cdot 11\) & 1a &         \(5^6 11^4\) &  \(4\) &   \(1\) &   \(1\) &   \(2\) & \(3\) & \((-,-,\otimes,-)\)  & \(\alpha_2\)    & \(\mathcal{K},\mathfrak{K}_1\) & \(6\) \\
  18 &  \(57\) &           \(3\cdot 19\) & 2  &         \(3^4 19^4\) &  \(1\) &   \(1\) &   \(1\) &   \(2\) & \(3\) & \((-,\otimes,-,-)\)  & \(\delta_2\)    & \(\mathcal{K}\) & \(10\) \\
  19 &  \(66\) &    \(2\cdot 3\cdot 11\) & 1b & \(5^2 2^4 3^4 11^4\) & \(13\) &   \(1\) &   \(2\) &   \(5\) & \(6\) & \((-,-,\times,-)\)   & \(\gamma\)      & \(2\cdot 5,3\cdot 5^3\) & \(17\) \\
  20 &  \(70\) &     \(2\cdot 5\cdot 7\) & 1a &      \(5^6 2^4 7^4\) & \(16\) &   \(0\) &   \(0\) &   \(1\) & \(6\) & \((\times,-,-,-)\)   & \(\gamma\)      &                 & \(14\) \\
  21 &  \(77\) &           \(7\cdot 11\) & 1b &     \(5^2 7^4 11^4\) &  \(4\) &   \(1\) &   \(1\) &   \(3\) & \(4\) & \((-,-,(\times),-)\) & \(\beta_2\)     & \(11\cdot 5^3,\mathcal{K}\) & \(7\) \\
  22 &  \(78\) &    \(2\cdot 3\cdot 13\) & 1b & \(5^2 2^4 3^4 13^4\) & \(13\) &   \(1\) &   \(2\) &   \(5\) & \(6\) & \((\times,-,-,-)\)   & \(\gamma\)      & \(3,2\cdot 5^3\) & \(37\) \\
  23 &  \(82\) &           \(2\cdot 41\) & 2  &         \(2^4 41^4\) &  \(1\) &   \(1\) &   \(1\) &   \(2\) & \(3\) & \((-,-,\otimes,-)\)  & \(\alpha_2\)    & \(\mathcal{K},\mathfrak{K}_2\) & \(15\) \\
  24 &  \(95\) &           \(5\cdot 19\) & 1a &         \(5^6 19^4\) &  \(4\) &   \(1\) &   \(1\) &   \(2\) & \(3\) & \((-,\otimes,-,-)\)  & \(\delta_2\)    & \(\mathcal{K}\) & \(9\) \\
  25 &  \(101\) &                        & 2  &            \(101^4\) &  \(1\) &   \(1\) &   \(2\) &   \(4\) & \(5\) & \((-,-,\otimes,\otimes)\) & \(\zeta_1\) & \(\mathfrak{K}_2\) & \(1\) \\
  26 &  \(110\) &   \(2\cdot 5\cdot 11\) & 1a &     \(5^6 2^4 11^4\) & \(16\) &   \(1\) &   \(1\) &   \(3\) & \(4\) & \((-,-,(\times),-)\) & \(\beta_2\)     & \(11,\mathcal{L}\) & \(11\) \\
  27 &  \(114\) &   \(2\cdot 3\cdot 19\) & 1b & \(5^2 2^4 3^4 19^4\) & \(13\) &   \(1\) &   \(2\) &   \(5\) & \(6\) & \((-,\times,-,-)\)   & \(\gamma\)      & \(2\cdot 5^3,3\cdot 5^3\) & \(20\) \\
  28 &  \(123\) &          \(3\cdot 41\) & 1b &     \(5^2 3^4 41^4\) &  \(3\) &   \(2\) &   \(3\) &   \(6\) & \(3\) & \((-,-,\otimes,-)\)  & \(\alpha_2\)    & \(\mathcal{K},\mathfrak{K}_2\) & \(2\) \\
  29 &  \(126\) &  \(2\cdot 3^2\cdot 7\) & 2  &      \(2^4 3^4 7^4\) &  \(4\) &   \(0\) &   \(0\) &   \(1\) & \(6\) & \((\times,-,-,-)\)   & \(\gamma\)      &                 &  \(6\) \\
  30 &  \(131\) &                        & 1b &        \(5^2 131^4\) &  \(1\) &   \(2\) &   \(2\) &   \(4\) & \(1\) & \((-,-,\otimes,-)\)  & \(\alpha_2\)    & \(\mathcal{L},\mathfrak{K}_2\) & \(6\) \\
  31 &  \(132\) & \(2^2\cdot 3\cdot 11\) & 2  &     \(2^4 3^4 11^4\) &  \(3\) &   \(1\) &   \(1\) &   \(3\) & \(4\) & \((-,-,(\times),-)\) & \(\beta_2\)     & \(11,\mathcal{L}\) & \(5\) \\
  32 &  \(133\) &          \(7\cdot 19\) & 1b &     \(5^2 7^4 19^4\) &  \(4\) &   \(1\) &   \(1\) &   \(3\) & \(4\) & \((-,\otimes,-,-)\)  & \(\beta_2\)     & \(7,\mathcal{L}\) & \(7\) \\
  33 &  \(139\) &                        & 1b &        \(5^2 139^4\) &  \(1\) &   \(1\) &   \(1\) &   \(3\) & \(4\) & \((-,\otimes,-,-)\)  & \(\beta_2\)     & \(5,\mathcal{L}\) & \(4\) \\
  34 &  \(140\) &  \(2^2\cdot 5\cdot 7\) & 1a &      \(5^6 2^4 7^4\) & \(16\) &   \(1\) &   \(2\) &   \(4\) & \(5\) & \((\times,-,-,-)\)   & \(\varepsilon\) & \(7\)           & \(5\) \\
  35 &  \(141\) &          \(3\cdot 47\) & 1b &     \(5^2 3^4 47^4\) &  \(3\) &   \(1\) &   \(2\) &   \(4\) & \(5\) & \((\times,-,-,-)\)   & \(\varepsilon\) & \(47\cdot 5\)   & \(19\) \\
  36 &  \(149\) &                        & 2  &            \(149^4\) &  \(1\) &   \(1\) &   \(1\) &   \(2\) & \(3\) & \((-,\otimes,-,\times)\) & \(\delta_2\) & \(\mathcal{L}\) & \(6\) \\
  37 &  \(151\) &                        & 2  &            \(151^4\) &  \(1\) &   \(1\) &   \(1\) &   \(2\) & \(3\) & \((-,-,\otimes,\times)\) & \(\alpha_2\) & \(\mathfrak{K}_1\) & \(3\) \\
  38 &  \(154\) &   \(2\cdot 7\cdot 11\) & 1b & \(5^2 2^4 7^4 11^4\) & \(12\) &   \(2\) &   \(3\) &   \(7\) & \(4\) & \((-,-,(\times),-)\) & \(\beta_2\)     & \(2\cdot 11^2,\mathcal{K}\) & \(3\) \\
  39 &  \(155\) &          \(5\cdot 31\) & 1a &         \(5^6 31^4\) &  \(4\) &   \(2\) &   \(3\) &   \(5\) & \(2\) & \((-,-,\otimes,-)\)  & \(\alpha_1\)    & \(\mathfrak{K}_1,\mathfrak{K}_2\) & \(2\) \\
  40 &  \(171\) &        \(3^2\cdot 19\) & 1b &     \(5^2 3^4 19^4\) &  \(3\) &   \(1\) &   \(2\) &   \(4\) & \(5\) & \((-,\times,-,-)\)   & \(\varepsilon\) & \(19\cdot 5^2\)   & \(6\) \\
  41 &  \(174\) &   \(2\cdot 3\cdot 29\) & 2  &     \(2^4 3^4 29^4\) &  \(3\) &   \(1\) &   \(1\) &   \(3\) & \(4\) & \((-,\otimes,-,-)\)  & \(\beta_2\)     & \(3^2\cdot 29,\mathcal{K}\) & \(3\) \\
  42 &  \(180\) &\(2^2\cdot 3^2\cdot 5\) & 1a &      \(5^6 2^4 3^4\) & \(16\) &   \(1\) &   \(2\) &   \(4\) & \(5\) & \((\times,-,-,-)\)   & \(\varepsilon\) & \(3\)           & \(5\) \\
  43 &  \(182\) &   \(2\cdot 7\cdot 13\) & 2  &     \(2^4 7^4 13^4\) &  \(4\) &   \(1\) &   \(2\) &   \(4\) & \(5\) & \((\times,-,-,-)\)   & \(\varepsilon\) & \(7\)           & \(1\) \\
  44 &  \(186\) &   \(2\cdot 3\cdot 31\) & 1b & \(5^2 2^4 3^4 31^4\) & \(13\) &   \(2\) &   \(3\) &   \(6\) & \(3\) & \((-,-,\otimes,-)\)  & \(\beta_1\)     & \(5,\mathfrak{K}_2\) & \(7\) \\
  45 &  \(190\) &   \(2\cdot 5\cdot 19\) & 1a &     \(5^6 2^4 19^4\) & \(16\) &   \(1\) &   \(1\) &   \(3\) & \(4\) & \((-,\otimes,-,-)\)  & \(\beta_2\)     & \(5,\mathcal{K}\) & \(9\) \\
  46 &  \(191\) &                        & 1b &        \(5^2 191^4\) &  \(1\) &   \(1\) &   \(2\) &   \(4\) & \(5\) & \((-,-,\otimes,-)\)  & \(\beta_1\)     & \(5,\mathfrak{K}_1\) & \(3\) \\
\hline
\end{tabular}
\end{center}
\end{table}

\newpage

\renewcommand{\arraystretch}{1.0}

\begin{table}[ht]
\caption{\(44\) prototypes  \(200<\mathbf{M}<510\) of pure metacyclic fields}
\label{tbl:Prototypes2}
\begin{center}
\begin{tabular}{|r|rc|ccr|cccc|ccc|r|}
\hline
 No. & \(\mathbf{M}\) &               Factors & S  &                  \(f^4\) &  \(m\) & \(V_L\) & \(V_M\) & \(V_N\) & \(E\) & \((1,2,4,5)\)        & T               & P & \(\mathbf{\lvert M\rvert}\) \\
\hline
  47 &  \(202\) &              \(2\cdot 101\) & 1b &        \(5^2 2^4 101^4\) &  \(4\) &   \(1\) &   \(1\) &   \(3\) & \(4\) & \((-,\otimes,-,-)\)  & \(\beta_2\)     & \(2\cdot 5,\mathcal{K}\) & \(6\) \\
  48 &  \(203\) &               \(7\cdot 29\) & 1b &         \(5^2 7^4 29^4\) &  \(4\) &   \(1\) &   \(2\) &   \(4\) & \(5\) & \((-,\times,-,-)\)   & \(\varepsilon\) & \(29\cdot 5\) & \(2\) \\
  49 &  \(209\) &              \(11\cdot 19\) & 1b &        \(5^2 11^4 19^4\) &  \(3\) &   \(2\) &   \(3\) &   \(7\) & \(4\) & \((-,\otimes,\times,-)\) & \(\beta_2\) & \(11\cdot 5^2,\mathcal{K}_{(19)}\) & \(2\) \\
  50 &  \(210\) &  \(2\cdot 3\cdot 5\cdot 7\) & 1a &      \(5^6 2^4 3^4 7^4\) & \(64\) &   \(1\) &   \(2\) &   \(5\) & \(6\) & \((\times,-,-,-)\)   & \(\gamma\)      & \(7,3^2\cdot 5\) & \(5\) \\
  51 &  \(211\) &                             & 1b &            \(5^2 211^4\) &  \(1\) &   \(3\) &   \(5\) &   \(9\) & \(2\) & \((-,-,\otimes,-)\)  & \(\delta_1\)    & \(\mathfrak{K}_2\) & \(1\) \\
  52 &  \(218\) &              \(2\cdot 109\) & 2  &            \(2^4 109^4\) &  \(1\) &   \(1\) &   \(2\) &   \(4\) & \(5\) & \((-,\times,-,-)\)   & \(\varepsilon\) & \(2\) & \(1\) \\
  53 &  \(231\) &        \(3\cdot 7\cdot 11\) & 1b &     \(5^2 3^4 7^4 11^4\) & \(12\) &   \(1\) &   \(2\) &   \(5\) & \(6\) & \((-,-,\times,-)\)   & \(\gamma\)      & \(11,7\cdot 5^2\) & \(5\) \\
  54 &  \(247\) &              \(13\cdot 19\) & 1b &        \(5^2 13^4 19^4\) &  \(3\) &   \(1\) &   \(2\) &   \(5\) & \(6\) & \((-,-,\times,-)\)   & \(\gamma\)      &  & \(2\) \\
  55 &  \(253\) &              \(11\cdot 23\) & 1b &        \(5^2 11^4 23^4\) &  \(3\) &   \(1\) &   \(2\) &   \(4\) & \(5\) & \((-,-,\otimes,-)\)  & \(\beta_1\)     & \(11,\mathfrak{K}_1\) & \(4\) \\
  56 &  \(259\) &               \(7\cdot 37\) & 1b &         \(5^2 7^4 37^4\) &  \(4\) &   \(1\) &   \(2\) &  \(5*\) & \(6\) & \((\times,-,-,-)\)   & \(\gamma\)      &  & \(1\) \\
  57 &  \(266\) &        \(2\cdot 7\cdot 19\) & 1b &     \(5^2 2^4 7^4 19^4\) & \(12\) &   \(1\) &   \(2\) &   \(5\) & \(6\) & \((-,\times,-,-)\)   & \(\gamma\)      & \(5,2^3\cdot 7\) & \(3\) \\
  58 &  \(273\) &        \(3\cdot 7\cdot 13\) & 1b &     \(5^2 3^4 7^4 13^4\) & \(12\) &   \(2\) &   \(4\) &   \(8\) & \(5\) & \((\times,-,-,-)\)   & \(\varepsilon\) & \(7\cdot 13^2\cdot 5^4\) & \(4\) \\
  59 &  \(275\) &             \(5^2\cdot 11\) & 1a &             \(5^6 11^4\) &  \(4\) &   \(2\) &   \(2\) &   \(4\) & \(1\) & \((-,-,\otimes,-)\)  & \(\alpha_2\)    & \(\mathcal{K},\mathfrak{K}_2\) & \(2\) \\
  60 &  \(276\) &      \(2^2\cdot 3\cdot 23\) & 2  &         \(2^4 3^4 23^4\) &  \(3\) &   \(0\) &   \(0\) &   \(1\) & \(6\) & \((\times,-,-,-)\)   & \(\gamma\)      &      & \(10\)  \\
  61 &  \(281\) &                             & 1b &            \(5^2 281^4\) &  \(1\) &   \(3\) &  \(5*\) &  \(9*\) & \(2\) & \((-,-,\otimes,-)\)  & \(\alpha_1\)    & \(\mathfrak{K}_1,\mathfrak{K}_2\) & \(1\) \\
  62 &  \(285\) &        \(3\cdot 5\cdot 19\) & 1a &         \(5^6 3^4 19^4\) & \(16\) &   \(1\) &   \(2\) &   \(5\) & \(6\) & \((-,\times,-,-)\)   & \(\gamma\)      &  & \(1\) \\
  63 &  \(286\) &       \(2\cdot 11\cdot 13\) & 1b &    \(5^2 2^4 11^4 13^4\) & \(13\) &   \(2\) &   \(3\) &   \(7\) & \(4\) & \((-,-,(\times),-)\) & \(\beta_2\)     & \(11^2\cdot 5,\mathcal{K}\) & \(3\) \\
  64 &  \(287\) &               \(7\cdot 41\) & 1b &         \(5^2 7^4 41^4\) &  \(4\) &   \(2\) &   \(2\) &   \(4\) & \(1\) & \((-,-,\otimes,-)\)  & \(\alpha_2\)    & \(\mathcal{K},\mathfrak{K}_1\) & \(1\) \\
  65 &  \(290\) &        \(2\cdot 5\cdot 29\) & 1a &         \(5^6 2^4 29^4\) & \(16\) &   \(1\) &   \(2\) &   \(4\) & \(5\) & \((-,\times,-,-)\)   & \(\varepsilon\) & \(29\) & \(2\) \\
  66 &  \(298\) &              \(2\cdot 149\) & 1b &        \(5^2 2^4 149^4\) &  \(4\) &   \(1\) &   \(2\) &   \(4\) & \(5\) & \((-,\times,-,-)\)   & \(\varepsilon\) & \(5\) & \(2\) \\
  67 &  \(301\) &               \(7\cdot 43\) & 2  &             \(7^4 43^4\) &  \(4\) &   \(0\) &   \(0\) &   \(1\) & \(6\) & \((-,-,-,\otimes)\)  & \(\eta\)        &                 & \(1\) \\
  68 &  \(302\) &              \(2\cdot 151\) & 1b &        \(5^2 2^4 151^4\) &  \(4\) &   \(1\) &   \(2\) &   \(4\) & \(5\) & \((-,-,\otimes,-)\)  & \(\beta_1\)     & \(2,\mathfrak{K}_1\) & \(2\) \\
  69 &  \(319\) &              \(11\cdot 29\) & 1b &        \(5^2 11^4 29^4\) &  \(3\) &   \(2\) &   \(2\) &   \(5\) & \(2\) & \((-,\otimes,(\times),-)\) & \(\alpha_3\) & \(\mathcal{K}_{(11)},\mathcal{K}_{(29)}\) & \(3\) \\
  70 &  \(329\) &               \(7\cdot 47\) & 1b &         \(5^2 7^4 47^4\) &  \(4\) &   \(1\) &   \(2\) &   \(4\) & \(5\) & \((\times,-,-,-)\)   & \(\varepsilon\) & \(47\) & \(7\) \\
  71 &  \(330\) & \(2\cdot 3\cdot 5\cdot 11\) & 1a &     \(5^6 2^4 3^4 11^4\) & \(64\) &   \(2\) &   \(4\) &   \(9\) & \(6\) & \((-,-,\times,-)\)   & \(\gamma\)      & \(3\cdot 11,5\cdot 11^3\) & \(2\) \\
  72 &  \(341\) &              \(11\cdot 31\) & 1b &        \(5^2 11^4 31^4\) &  \(3\) &   \(3\) &   \(5\) &   \(9\) & \(2\) & \((-,-,\otimes,-)\)  & \(\alpha_1\)    & \(\mathfrak{K}_{(11),1}\mathfrak{K}_{(31),2}^3,\) & \\
     &          &                             &    &                          &        &         &         &         &       &                      &                 & \(\mathfrak{K}_{(11),2}\mathfrak{K}_{(31),1}\) & \(2\) \\
  73 &  \(348\) &      \(2^2\cdot 3\cdot 29\) & 1b &     \(5^2 2^4 3^4 29^4\) & \(13\) &   \(2\) &   \(4\) &   \(8\) & \(5\) & \((-,\times,-,-)\)   & \(\varepsilon\) & \(2\cdot 3\cdot 5\) & \(2\) \\
  74 &  \(377\) &              \(13\cdot 29\) & 1b &        \(5^2 13^4 29^4\) &  \(3\) &   \(2\) &   \(3\) &   \(6\) & \(3\) & \((-,\otimes,-,-)\)  & \(\delta_2\)    & \(\mathcal{K}\) & \(1\) \\
  75 &  \(379\) &                             & 1b &            \(5^2 379^4\) &  \(1\) &   \(1\) &   \(2\) &   \(4\) & \(5\) & \((-,\times,-,-)\)   & \(\varepsilon\) & \(5\) & \(1\) \\
  76 &  \(385\) &        \(5\cdot 7\cdot 11\) & 1a &         \(5^6 7^4 11^4\) & \(16\) &   \(1\) &   \(1\) &   \(3\) & \(4\) & \((-,\otimes,-,-)\)  & \(\beta_2\)     & \(7\cdot 11^2,\mathcal{K}\) & \(1\) \\
  77 &  \(390\) & \(2\cdot 3\cdot 5\cdot 13\) & 1a &     \(5^6 2^4 3^4 13^4\) & \(64\) &   \(1\) &   \(2\) &   \(5\) & \(6\) & \((\times,-,-,-)\)   & \(\gamma\)      & \(2,5\) & \(4\) \\
  78 &  \(398\) &              \(2\cdot 199\) & 1b &        \(5^2 2^4 199^4\) &  \(4\) &   \(1\) &   \(1\) &   \(3\) & \(4\) & \((-,\otimes,-,-)\)  & \(\beta_2\)     & \(5,\mathcal{K}\) & \(7\) \\
  79 &  \(399\) &        \(3\cdot 7\cdot 19\) & 2  &         \(3^4 7^4 19^4\) &  \(4\) &   \(1\) &   \(1\) &   \(3\) & \(4\) & \((-,\otimes,-,-)\)  & \(\beta_2\)     & \(3\cdot 19^2,\mathcal{K}\) & \(2\) \\
  80 &  \(401\) &                             & 2  &                \(401^4\) &  \(1\) &   \(2\) &   \(3\) &   \(5\) & \(2\) & \((-,-,\otimes,-)\)  & \(\alpha_1\)    & \(\mathfrak{K}_1,\mathfrak{K}_2\) & \(2\) \\
  81 &  \(418\) &       \(2\cdot 11\cdot 19\) & 2  &        \(2^4 11^4 19^4\) &  \(3\) &   \(2\) &   \(3\) &   \(6\) & \(3\) & \((-,\otimes,\otimes,-)\) & \(\alpha_2\) & \(\mathcal{K}_{(19)},\mathfrak{K}_{(11),1}\) & \(1\) \\
  82 &  \(421\) &                             & 1b &            \(5^2 421^4\) &  \(1\) &   \(1\) &   \(2\) &   \(3\) & \(4\) & \((-,-,\otimes,-)\)  & \(\delta_1\)    & \(\mathfrak{K}_2\) & \(5\) \\
  83 &  \(422\) &              \(2\cdot 211\) & 1b &        \(5^2 2^4 211^4\) &  \(3\) &   \(1\) &   \(2\) &   \(4\) & \(5\) & \((-,-,\times,-)\)   & \(\varepsilon\) & \(2\cdot 5^2\) & \(6\) \\
  84 &  \(451\) &              \(11\cdot 41\) & 2  &            \(11^4 41^4\) &  \(1\) &   \(2\) &   \(3\) &   \(6\) & \(3\) & \((-,-,\otimes,-)\)  & \(\alpha_2\)    & \(\mathcal{K}_{(11)}\mathcal{K}_{(41)}^4,\) & \\
     &          &                             &    &                          &        &         &         &         &       &                      &                 & \(\mathfrak{K}_{(11),1}\mathfrak{K}_{(41),2}^3\) & \(1\) \\
  85 &  \(462\) & \(2\cdot 3\cdot 7\cdot 11\) & 1b & \(5^2 2^4 3^4 7^4 11^4\) & \(52\) &   \(2\) &   \(4\) &   \(9\) & \(6\) & \((-,-,\times,-)\)   & \(\gamma\)      & \(5^2\cdot 7,3\cdot 5\cdot 11^2\) & \(1\) \\
  86 &  \(465\) &        \(3\cdot 5\cdot 31\) & 1a &         \(5^6 3^4 31^4\) & \(16\) &   \(2\) &   \(3\) &  \(7*\) & \(4\) & \((-,-,(\times),-)\) & \(\beta_2\)     & \(5,\mathcal{K}\) & \(1\) \\
  87 &  \(473\) &              \(11\cdot 43\) & 1b &        \(5^2 11^4 43^4\) &  \(4\) &   \(2\) &   \(3\) &  \(7*\) & \(4\) & \((-,-,(\times),-)\) & \(\beta_2\)     & \(11,\mathcal{L}\) & \(1\) \\
  88 &  \(482\) &              \(2\cdot 241\) & 2  &            \(2^4 241^4\) &  \(1\) &   \(1\) &   \(2\) &   \(4\) & \(5\) & \((-,-,\otimes,-)\)  & \(\beta_1\)     & \(241,\mathfrak{K}_2\) & \(2\) \\
  89 &  \(502\) &              \(2\cdot 251\) & 1b &        \(5^2 2^4 251^4\) &  \(4\) &  *\(2\) &  *\(4\) &  *\(8\) & \(5\) & \((-,-,\otimes,-)\)  & \(\beta_1\)     & \(251\cdot 5,\mathfrak{K}_2\) & \(1\) \\
  90 &  \(505\) &              \(5\cdot 101\) & 1a &            \(5^6 101^4\) &  \(4\) &   \(1\) &   \(1\) &   \(3\) & \(4\) & \((-,-,(\times),\otimes)\) & \(\zeta_2\) & \(\mathcal{K}\) & \(2\) \\
\hline
\end{tabular}
\end{center}
\end{table}

\newpage

\renewcommand{\arraystretch}{1.0}

\begin{table}[ht]
\caption{\(39\) prototypes  \(510<\mathbf{M}<900\) of pure metacyclic fields}
\label{tbl:Prototypes3}
\begin{center}
\begin{tabular}{|r|rc|ccr|cccc|ccc|r|}
\hline
 No. & \(\mathbf{M}\) &               Factors & S  &                  \(f^4\) &  \(m\) & \(V_L\) & \(V_M\) & \(V_N\) & \(E\) & \((1,2,4,5)\)        & T               & P & \(\mathbf{\lvert M\rvert}\) \\
\hline
  91 &  \(517\) &              \(11\cdot 47\) & 1b &        \(5^2 11^4 47^4\) &  \(3\) &   \(2\) &   \(3\) &   \(6\) & \(3\) & \((-,-,\otimes,-)\)  & \(\beta_1\)     & \(11\cdot 5^2,\mathfrak{K}_2\) & \(1\) \\
  92 &  \(532\) &      \(2^2\cdot 7\cdot 19\) & 2  &         \(2^4 7^4 19^4\) &  \(4\) &   \(1\) &   \(2\) &   \(4\) & \(5\) & \((-,\times,-,-)\)   & \(\varepsilon\) & \(2\cdot 7\) & \(1\) \\
  93 &  \(546\) & \(2\cdot 3\cdot 7\cdot 13\) & 1b & \(5^2 2^4 3^4 7^4 13^4\) & \(52\) &   \(2\) &   \(4\) &   \(9\) & \(6\) & \((\times,-,-,-)\)   & \(\gamma\)      & \(3\cdot 5\cdot 13,\)         & \\
     &          &                             &    &                          &        &         &         &         &       &                      &                 & \(5^2\cdot 7\cdot 13\) & \(3\) \\
  94 &  \(550\) &      \(2\cdot 5^2\cdot 11\) & 1a &         \(5^6 2^4 11^4\) & \(16\) &   \(2\) &   \(3\) &   \(6\) & \(3\) & \((-,-,\otimes,-)\)  & \(\alpha_2\)    & \(\mathcal{K}, \mathfrak{K}_1\) & \(1\) \\
  95 &  \(551\) &              \(19\cdot 29\) & 2  &            \(19^4 29^4\) &  \(1\) &   \(2\) &   \(2\) &   \(5\) & \(2\) & \((-,\otimes,-,-)\)  & \(\alpha_3\)    & \(\mathcal{K}_{(19)},\mathcal{K}_{(29)}\) & \(1\) \\
  96 &  \(570\) & \(2\cdot 3\cdot 5\cdot 19\) & 1a &     \(5^6 2^4 3^4 19^4\) & \(64\) &   \(1\) &   \(2\) &   \(5\) & \(6\) & \((-,\times,-,-)\)   & \(\gamma\)      & \(2^4\cdot 3,2^4\cdot 5\) & \(2\) \\
  97 &  \(574\) &        \(2\cdot 7\cdot 41\) & 2  &         \(2^4 7^4 41^4\) &  \(4\) &   \(1\) &   \(1\) &   \(3\) & \(4\) & \((-,-,(\times),-)\) & \(\beta_2\)     & \(41,\mathcal{L}\) & \(3\) \\
  98 &  \(590\) &        \(2\cdot 5\cdot 59\) & 1a &         \(5^6 2^4 59^4\) & \(16\) &   \(2\) &   \(3\) &  *\(7\) & \(4\) & \((-,\otimes,-,-)\)  & \(\beta_2\)     & \(5,\mathcal{K}\) & \(1\) \\
  99 &  \(602\) &        \(2\cdot 7\cdot 43\) & 1b &     \(5^2 2^4 7^4 43^4\) & \(16\) &   \(1\) &   \(2\) &   \(5\) & \(6\) & \((\times,-,-,-)\)   & \(\gamma\)      & \(2,7\) & \(2\) \\
 100 &  \(606\) &       \(2\cdot 3\cdot 101\) & 1b &    \(5^2 2^4 3^4 101^4\) & \(12\) &   \(1\) &   \(2\) &   \(5\) & \(6\) & \((-,-,\times,-)\)   & \(\gamma\)      & \(2\cdot 5,101\) & \(2\) \\
 101 &  \(609\) &        \(3\cdot 7\cdot 29\) & 1b &     \(5^2 3^4 7^4 29^4\) & \(12\) &   \(2\) &   \(3\) &   \(7\) & \(4\) & \((-,\otimes,-,-)\)  & \(\beta_2\)     & \(7^2\cdot 5,\mathcal{K}\) & \(1\) \\
 102 &  \(620\) &      \(2^2\cdot 5\cdot 31\) & 1a &         \(5^6 2^4 31^4\) & \(16\) &   \(2\) &  \(4*\) &  \(8*\) & \(5\) & \((-,-,\otimes,-)\)  & \(\beta_1\)     & \(5,\mathfrak{K}_1\) & \(1\) \\
 103 &  \(627\) &       \(3\cdot 11\cdot 19\) & 2  &    \(5^2 3^4 11^4 59^4\) & \(13\) &   \(3\) &   \(4\) &   \(9\) & \(2\) & \((-,\otimes,(\times),-)\) & \(\alpha_3\) & \(\mathcal{K}_{(11)},\mathcal{K}_{(19)}\) & \(1\) \\
 104 &  \(638\) &       \(2\cdot 11\cdot 29\) & 1b &    \(5^2 2^4 11^4 29^4\) & \(13\) &   \(2\) &   \(3\) &   \(7\) & \(4\) & \((-,\otimes,(\times),-)\) & \(\beta_2\) & \(2^4\cdot 11\cdot 5,\) & \\
     &          &                             &    &                          &        &         &         &         &       &                      &                 & \(\mathcal{K}_{(11)}\cdot\mathcal{K}_{(29)}^4\) & \(2\) \\
 105 &  \(649\) &              \(11\cdot 59\) & 2  &            \(11^4 59^4\) &  \(1\) &   \(2\) &   \(2\) &   \(5\) & \(2\) & \((-,\otimes,(\times),-)\) & \(\alpha_3\) & \(\mathcal{K}_{(11)},\mathcal{K}_{(59)}\) & \(2\) \\
 106 &  \(660\) &\(2^2\cdot 3\cdot 5\cdot 11\)& 1a &     \(5^6 2^4 3^4 11^4\) & \(64\) &   \(1\) &   \(2\) &   \(5\) & \(6\) & \((-,-,\times,-)\)   & \(\gamma\)      & \(2\cdot 5,5\cdot 11\) & \(2\) \\
 107 &  \(665\) &        \(5\cdot 7\cdot 19\) & 1a &         \(5^6 7^4 19^4\) & \(16\) &   \(1\) &   \(1\) &   \(3\) & \(4\) & \((-,\otimes,-,-)\)  & \(\beta_2\)     & \(7,\mathcal{K}\) & \(1\) \\
 108 &  \(671\) &              \(11\cdot 61\) & 1b &        \(5^2 11^4 61^4\) &  \(3\) &   \(3\) &   \(4\) &   \(8\) & \(1\) & \((-,-,\otimes,-)\)  & \(\alpha_2\)    & \(\mathcal{K}_{(11)}\mathcal{K}_{(61)},\) & \\
     &          &                             &    &                          &        &         &         &         &       &                      &                 & \(\mathfrak{K}_{(11),2}\mathfrak{K}_{(61),1}^2\) & \(1\) \\
 109 &  \(682\) &       \(2\cdot 11\cdot 31\) & 2  &        \(2^4 11^4 31^4\) &  \(3\) &   \(2\) &   \(3\) &   \(6\) & \(3\) & \((-,-,\otimes,-)\)  & \(\alpha_2\)    & \(\mathcal{K}_{(11)}\mathcal{K}_{(31)}^3,\) & \\
     &          &                             &    &                          &        &         &         &         &       &                      &                 & \(\mathfrak{K}_{(11),1}\mathfrak{K}_{(31),2}^3\) & \(1\) \\
 110 &  \(691\) &                             & 1b &            \(5^2 691^4\) &  \(1\) &   \(3\) &   \(4\) &   \(8\) & \(1\) & \((-,-,\otimes,-)\)  & \(\alpha_2\)    & \(\mathcal{L},\mathfrak{K}_2\) & \(1\) \\
 111 &  \(693\) &      \(3^2\cdot 7\cdot 11\) & 2  &         \(3^4 7^4 11^4\) &  \(4\) &   \(2\) &   \(3\) &   \(6\) & \(3\) & \((-,-,\otimes,-)\)  & \(\beta_1\)     & \(3\cdot 7,\mathfrak{K}_2\) & \(1\) \\
 112 &  \(695\) &              \(5\cdot 139\) & 1a &            \(5^6 139^4\) &  \(4\) &   \(1\) &   \(2\) &   \(4\) & \(5\) & \((-,\times,-,-)\)   & \(\varepsilon\) & \(139\) & \(1\) \\
 113 &  \(702\) &      \(2\cdot 3^3\cdot 13\) & 1b &     \(5^2 2^4 3^4 13^4\) & \(13\) &   \(2\) &   \(4\) &   \(8\) & \(5\) & \((\times,-,-,-)\)   & \(\varepsilon\) & \(13\cdot 5^4\) & \(2\) \\
 114 &  \(707\) &              \(7\cdot 101\) & 2  &            \(7^4 101^4\) &  \(4\) &   \(1\) &   \(1\) &   \(3\) & \(4\) & \((-,\otimes,-,\otimes)\) & \(\zeta_2\) & \(\mathcal{K}\) & \(1\) \\
 115 &  \(710\) &        \(2\cdot 5\cdot 71\) & 1a &         \(5^6 2^4 71^4\) & \(16\) &   \(2\) &   \(2\) &   \(4\) & \(1\) & \((-,-,\otimes,-)\)  & \(\alpha_2\)    & \(\mathcal{K}, \mathfrak{K}_2\) & \(4\) \\
 116 &  \(745\) &              \(5\cdot 149\) & 1a &            \(5^6 149^4\) &  \(4\) &   \(1\) &   \(1\) &   \(3\) & \(4\) & \((-,\otimes,-,\otimes)\) & \(\zeta_2\) & \(\mathcal{K}\) & \(2\) \\
 117 &  \(749\) &              \(7\cdot 107\) & 2  &            \(7^4 107^4\) &  \(4\) &   \(1\) &   \(2\) &   \(4\) & \(5\) & \((-,-,-,\times)\)   & \(\varepsilon\) &                 & \(1\) \\
 118 &  \(751\) &                             & 2  &                \(751^4\) &  \(1\) &   \(2\) &   \(2\) &   \(4\) & \(1\) & \((-,-,\otimes,\times)\) & \(\alpha_2\) & \(\mathcal{L},\mathfrak{K}_1\) & \(1\) \\
 119 &  \(770\) & \(2\cdot 5\cdot 7\cdot 11\) & 1a &     \(5^6 2^4 7^4 11^4\) & \(64\) &   \(1\) &   \(2\) &   \(5\) & \(6\) & \((-,-,\times,-)\)   & \(\gamma\)      & \(2\cdot 5,2^2\cdot 11\) & \(1\) \\ 
 120 &  \(779\) &              \(19\cdot 41\) & 1b &        \(5^2 19^4 41^4\) &  \(3\) &   \(3\) &   \(4\) &   \(8\) & \(1\) & \((-,-,\otimes,-)\)  & \(\alpha_2\)    & \(\mathcal{K}_{(19)}\mathcal{K}_{(41)},\) & \\
     &          &                             &    &                          &        &         &         &         &       &                      &                 & \(\mathfrak{K}_{(41),1}\) & \(1\) \\
 121 &  \(785\) &              \(5\cdot 157\) & 1a &            \(5^6 157^4\) &  \(4\) &   \(1\) &   \(2\) &   \(4\) & \(5\) & \((-,-,-,\times)\)   & \(\varepsilon\) &                 & \(1\) \\
 122 &  \(798\) & \(2\cdot 3\cdot 7\cdot 19\) & 1b & \(5^2 2^4 3^4 7^4 19^4\) & \(52\) &   \(2\) &   \(4\) &   \(9\) & \(6\) & \((-,\times,-,-)\)   & \(\gamma\)      & \(3\cdot 19,7\cdot 5^4\) & \(1\) \\
 123 &  \(808\) &            \(2^3\cdot 101\) & 1b &        \(5^2 2^4 101^4\) &  \(4\) &   \(2\) &   \(2\) &   \(4\) & \(1\) & \((-,-,\otimes,-)\)  & \(\alpha_2\)    & \(\mathcal{K}, \mathfrak{K}_1\) & \(2\) \\
 124 &  \(825\) &      \(3\cdot 5^2\cdot 11\) & 1a &         \(5^6 3^4 11^4\) & \(16\) &   \(1\) &   \(2\) &   \(4\) & \(5\) & \((-,-,\otimes,-)\)  & \(\beta_1\)     & \(3^2\cdot 11,\mathfrak{K}_1\) & \(1\) \\
 125 &  \(843\) &              \(3\cdot 281\) & 2  &            \(3^4 281^4\) &  \(1\) &   \(1\) &   \(2\) &   \(3\) & \(4\) & \((-,-,\otimes,-)\)  & \(\delta_1\)    & \(\mathfrak{K}_1\) & \(1\) \\
 126 &  \(858\) &\(2\cdot 3\cdot 11\cdot 13\) & 1b &\(5^2 2^4 3^4 11^4 13^4\) & \(51\) &   \(2\) &   \(4\) &   \(9\) & \(6\) & \((-,-,\times,-)\)   & \(\gamma\)      & \(2\cdot 5,13\cdot 5\) & \(1\) \\
 127 &  \(861\) &        \(3\cdot 7\cdot 41\) & 1b &     \(5^2 3^4 7^4 41^4\) & \(12\) &   \(3\) &   \(4\) &   \(8\) & \(1\) & \((-,-,\otimes,-)\)  & \(\alpha_2\)    & \(\mathcal{K},\mathfrak{K}_1\) & \(1\) \\
 128 &  \(893\) &              \(19\cdot 47\) & 2  &            \(19^4 47^4\) &  \(1\) &   \(1\) &   \(1\) &   \(3\) & \(4\) & \((-,\otimes,-,-)\)  & \(\beta_2\)     & \(19,\mathfrak{L}\) & \(1\) \\
 129 &  \(894\) &       \(2\cdot 3\cdot 149\) & 1b &    \(5^2 2^4 3^4 149^4\) & \(12\) &   \(1\) &   \(2\) &   \(5\) & \(6\) & \((\times,-,-,-)\)   & \(\gamma\)      & \(5,2\cdot 3^2\) & \(1\) \\
\hline
\end{tabular}
\end{center}
\end{table}

\newpage

\renewcommand{\arraystretch}{1.0}

\begin{table}[ht]
\caption{\(5\) prototypes  \(900<\mathbf{M}<1000\) of pure metacyclic fields}
\label{tbl:Prototypes4}
\begin{center}
\begin{tabular}{|r|rc|ccr|cccc|ccc|r|}
\hline
 No. & \(\mathbf{M}\) &               Factors & S  &                  \(f^4\) &  \(m\) & \(V_L\) & \(V_M\) & \(V_N\) & \(E\) & \((1,2,4,5)\)        & T               & P & \(\mathbf{\lvert M\rvert}\) \\
\hline
 130 &  \(902\) &       \(2\cdot 11\cdot 41\) & 1b &    \(5^2 2^4 11^4 41^4\) & \(13\) &   \(2\) &   \(3\) &   \(7\) & \(4\) & \((-,-,(\times),-)\) & \(\beta_2\)     & \(11\cdot 5,\) & \\
     &          &                             &    &                          &        &         &         &         &       &                      &                 & \(\mathcal{K}_{(11)}\cdot\mathcal{K}_{(41)}\) & \(1\) \\
 131 &  \(924\) &\(2^2\cdot 3\cdot 7\cdot 11\)& 1a &     \(2^4 3^4 7^4 11^4\) & \(12\) &   \(1\) &   \(2\) &   \(5\) & \(6\) & \((-,-,\times,-)\)   & \(\gamma\)      & \(3\cdot 7,7\cdot 11\) & \(1\) \\ 
 132 &  \(955\) &              \(5\cdot 191\) & 1a &            \(5^6 191^4\) &  \(4\) &  \(2*\) &  \(3*\) &  \(6*\) & \(3\) & \((-,-,\otimes,-)\)  & \(\alpha_2\)    & \(\mathcal{K},\mathfrak{K}_1\) & \(1\) \\
 133 &  \(957\) &       \(3\cdot 11\cdot 29\) & 2  &        \(3^4 11^4 29^4\) &  \(3\) &   \(2\) &   \(2\) &   \(5\) & \(2\) & \((-,\otimes,(\times),-)\) & \(\alpha_3\) & \(\mathcal{K}_{(11)},\mathcal{K}_{(29)}\) & \(1\) \\
 134 &  \(982\) &              \(2\cdot 491\) & 2  &            \(2^4 491^4\) &  \(1\) &   \(1\) &   \(1\) &   \(2\) & \(3\) & \((-,-,\otimes,-)\)  & \(\alpha_2\)    & \(\mathcal{K},\mathfrak{K}_1\) & \(2\) \\
\hline
\end{tabular}
\end{center}
\end{table}


\subsection{General theorems on DPF types and Polya fields}
\label{ss:Theorems}
\noindent
There is only a single finite similarity class \(\lbrack 5\rbrack=\lbrace 5\rbrace\),
characterized by the exceptional number \(t=0\) of primes \(q\ne 5\) dividing the conductor \(f\) (here \(f^4=5^6\)).
The invariants of this unique metacyclic Polya field \(N\) are given by

\begin{equation}
\label{eqn:Prototype5Theta}
\begin{aligned}
& \mathbf{\lbrack 5\rbrack}, \text{ species \(1\mathrm{a}\), }
(e_0;t,u,v,m;n,s_2,s_4)=(6;0,0,0,1;0,0,0), \\
& \text{type \(\vartheta\), }(U,\eta,\zeta;A,I,R)=(0,\times,\times;1,0,0),
\text{ and } (V_L,V_M,V_N;E)=(0,0,0;5).
\end{aligned}
\end{equation}

We conjecture that all the other similarity classes are infinite.
Precisely four of them can actually be given by
parametrized infinite sequences in a deterministic way
aside from the intrinsic probabilistic nature of the occurrence
of primes in residue classes
and of composite integers with assigned shape of prime decomposition.
This was proved in
\cite[Thm. 10.1]{Ma2a}
and
\cite[Thm. 2.1]{Ma2b}.

\begin{theorem}
\label{thm:InfSimCls}
Each of the following infinite sequences of conductors \(f=f_{N/K}\)
unambiguously determines the DPF type of the pure metacyclic fields \(N\)
in the associated multiplet with \(m=m(f)\) members.
\begin{enumerate}
\item
\(f=q\) with \(q\in\mathbb{P}\), \(q\equiv\pm 7\,(\mathrm{mod}\,25)\)
gives rise to a singulet, \(m=1\), with DPF type \(\vartheta\),
\item
\(f^4=5^2\cdot q^4\) with \(q\in\mathbb{P}\), \(q\equiv\pm 2\,(\mathrm{mod}\,5)\), \(q\not\equiv\pm 7\,(\mathrm{mod}\,25)\)
gives rise to a singulet, \(m=1\), with DPF type \(\varepsilon\),
\item
\(f^4=5^6\cdot q^4\) with \(q\in\mathbb{P}\), \(q\equiv\pm 2\,(\mathrm{mod}\,5)\), \(q\not\equiv\pm 7\,(\mathrm{mod}\,25)\)
gives rise to a quartet, \(m=4\), with homogeneous DPF type \((\varepsilon,\varepsilon,\varepsilon,\varepsilon)\),
\item
\(f=q_1\cdot q_2\) with \(q_i\in\mathbb{P}\), \(q_i\equiv\pm 2\,(\mathrm{mod}\,5)\), \(q_i\not\equiv\pm 7\,(\mathrm{mod}\,25)\)
gives rise to a singulet, \(m=1\), with DPF type \(\varepsilon\).
\end{enumerate}
\end{theorem}

In fact, the shape of the conductors in Theorem
\ref{thm:InfSimCls}
does not only determine the refined Dedekind species and the DPF type,
but also the structure of the \(5\)-class groups of the fields \(L\), \(M\) and \(N\).

\begin{corollary}
\label{cor:InfSimCls}
The invariants of the similarity classes
defined by the four infinite sequences of conductors in Theorem
\ref{thm:InfSimCls}
are given as follows, in the same order:
\begin{equation}
\label{eqn:Prototype7Theta}
\begin{aligned}
& \mathbf{\lbrack 7\rbrack}, \text{ species \(2\), }
(e_0;t,u,v,m;n,s_2,s_4)=(0;1,1,0,1;1,0,0), \\
& \text{type \(\vartheta\), }(U,\eta,\zeta;A,I,R)=(0,\times,\times;1,0,0),
\text{ and } (V_L,V_M,V_N;E)=(0,0,0;5);
\end{aligned}
\end{equation}
\begin{equation}
\label{eqn:Prototype2Epsilon}
\begin{aligned}
& \mathbf{\lbrack 2\rbrack}, \text{ species \(1\mathrm{b}\), }
(e_0;t,u,v,m;n,s_2,s_4)=(2;1,0,1,1;1,0,0), \\
& \text{type \(\varepsilon\), }(U,\eta,\zeta;A,I,R)=(1,\times,-;2,0,0),
\text{ and } (V_L,V_M,V_N;E)=(0,0,0;5);
\end{aligned}
\end{equation}
\begin{equation}
\label{eqn:Prototype10Epsilon}
\begin{aligned}
& \mathbf{\lbrack 10\rbrack}, \text{ species \(1\mathrm{a}\), }
(e_0;t,u,v,m;n,s_2,s_4)=(6;1,0,1,4;1,0,0), \\
& \text{type \(\varepsilon\), }(U,\eta,\zeta;A,I,R)=(1,\times,-;2,0,0),
\text{ and } (V_L,V_M,V_N;E)=(0,0,0;5);
\end{aligned}
\end{equation}
\begin{equation}
\label{eqn:Prototype18Epsilon}
\begin{aligned}
& \mathbf{\lbrack 18\rbrack}, \text{ species \(2\), }
(e_0;t,u,v,m;n,s_2,s_4)=(0;2,0,2,1;2,0,0), \\
& \text{type \(\varepsilon\), }(U,\eta,\zeta;A,I,R)=(1,\times,-;2,0,0),
\text{ and } (V_L,V_M,V_N;E)=(0,0,0;5).
\end{aligned}
\end{equation}
The pure metacyclic fields \(N\) associated with these four similarity classes are Polya fields.
\end{corollary}

\begin{remark}
\label{rmk:InfSimCls}
The statements concerning \(5\)-class groups in Corollary
\ref{cor:InfSimCls}
were proved by Parry in
\cite[Thm. IV, p. 481]{Pa},
where Formula (10) gives the shape of radicands
associated with the conductors in our Theorem
\ref{thm:InfSimCls}.
\end{remark}

\begin{proof}
(of Theorem
\ref{thm:InfSimCls}
and Corollary
\ref{cor:InfSimCls})
It only remains to show the claims for the composite radicands
associated with conductors \(f^4=5^6\cdot q^4\) and \(f=q_2\cdot q_2\).
See
\cite[Thm. 10.6]{Ma2a}.
\end{proof}


For similarity classes distinct from the four infinite classes in Theorem
\ref{thm:InfSimCls}
we cannot provide deterministic criteria for the DPF type
and for the homogeneity of multiplets with \(m>1\).
In general, the members of a multiplet belong to distinct similarity classes,
thus giving rise to \textit{heterogeneous} DPF types.
We explain these phenomena with the simplest cases
where only two DPF types are involved
(type splitting).

\begin{theorem}
\label{thm:EpsilonEta}
Each of the following infinite sequences of conductors \(f=f_{N/K}\)
admits precisely two DPF types of the pure metacyclic fields \(N\)
in the associated quartet with \(m=4\) members.
\begin{enumerate}
\item
\(f^4=5^6\cdot q^4\) with \(q\in\mathbb{P}\), \(q\equiv\pm 7\,(\mathrm{mod}\,25)\)
gives rise to a quartet with possibly heterogeneous DPF type \((\varepsilon^x,\eta^y)\), \(x+y=4\),
\item
\(f=q_1\cdot q_2\) with \(q_i\in\mathbb{P}\), \(q_i\equiv\pm 7\,(\mathrm{mod}\,25)\)
gives rise to a quartet with possibly heterogeneous DPF type \((\varepsilon^x,\eta^y)\), \(x+y=4\).
\end{enumerate}
\end{theorem}

\begin{example}
\label{exm:EpsilonEta}
It is not easy to find complete quartets,
whose members are spread rather widely.
The smallest quartet \((35,175,245,4375)=(5\cdot 7,5^2\cdot 7,5\cdot 7^2,5^4\cdot 7)\)
belonging to the first infinite sequence
contains the member \(D=4375\) outside of the range of our systematic computations.
We have determined its DPF type separately
and thus discovered a homogeneous quartet of type \((\eta,\eta,\eta,\eta)\).
However, we cannot generally exclude the occurrence of heterogeneous quartets.
\end{example}

\begin{corollary}
\label{cor:EpsilonEta}
The invariants of the similarity classes
defined by the two infinite sequences of conductors in Theorem
\ref{thm:EpsilonEta}
are given as follows, in the same order.
The statements concerning \(5\)-class groups are only conjectural.
Each sequence splits into two similarity classes. \\
The classes for \(f^4=5^6\cdot q^4\) are:
\begin{equation}
\label{eqn:Prototype35Eta}
\begin{aligned}
& \mathbf{\lbrack 35\rbrack}, \text{ species \(1\mathrm{a}\), }
(e_0;t,u,v,m;n,s_2,s_4)=(6;1,1,0,4;1,0,0), \\
& \text{type \(\eta\), }(U,\eta,\zeta;A,I,R)=(1,-,\times;2,0,0),
\text{ and } (V_L,V_M,V_N;E)=(0,0,1;6);
\end{aligned}
\end{equation}
\begin{equation}
\label{eqn:Prototype785Epsilon}
\begin{aligned}
& \mathbf{\lbrack 785\rbrack}, \text{ species \(1\mathrm{a}\), }
(e_0;t,u,v,m;n,s_2,s_4)=(6;1,1,0,4;1,0,0), \\
& \text{type \(\varepsilon\), }(U,\eta,\zeta;A,I,R)=(1,\times,-;2,0,0),
\text{ and } (V_L,V_M,V_N;E)=(1,2,4;5).
\end{aligned}
\end{equation}
The classes for \(f=q_1\cdot q_2\) are:
\begin{equation}
\label{eqn:Prototype301Eta}
\begin{aligned}
& \mathbf{\lbrack 301\rbrack}, \text{ species \(2\), }
(e_0;t,u,v,m;n,s_2,s_4)=(0;2,2,0,4;2,0,0), \\
& \text{type \(\eta\), }(U,\eta,\zeta;A,I,R)=(1,-,\times;2,0,0),
\text{ and } (V_L,V_M,V_N;E)=(0,0,1;6);
\end{aligned}
\end{equation}
\begin{equation}
\label{eqn:Prototype749Epsilon}
\begin{aligned}
& \mathbf{\lbrack 749\rbrack}, \text{ species \(2\), }
(e_0;t,u,v,m;n,s_2,s_4)=(0;2,2,0,4;2,0,0), \\
& \text{type \(\varepsilon\), }(U,\eta,\zeta;A,I,R)=(1,\times,-;2,0,0),
\text{ and } (V_L,V_M,V_N;E)=(1,2,4;5).
\end{aligned}
\end{equation}
All pure metacyclic fields \(N\) associated with these four similarity classes are Polya fields.
\end{corollary}

\begin{proof}
(of Theorem
\ref{thm:EpsilonEta}
and Corollary
\ref{cor:EpsilonEta})
\end{proof}

\begin{remark}
\label{rmk:EpsilonEta}
The statements on \(5\)-class groups in Corollary
\ref{cor:EpsilonEta}
have been verified for all examples with \(2\le D<1000\)
by our computations.
In particular,
the occurrence of the radicands
\(D=749=7\cdot 107\) and \(D=785=5\cdot 157\), both with \(V_L=1\),
proves the impossibility of the general claim \(5\nmid h(L)\)
for the two situations mentioned in
\cite[Lem. 3.3 (ii) and (iv), p. 204]{Ii}
and
\cite[Thm. 5 (ii) and (iv), p. 5]{Ky2},
partially also indicated in
\cite[Thm. IV (11), p. 481]{Pa}.
\end{remark}


\begin{theorem}
\label{thm:EpsilonGamma}
Each of the following infinite sequences of conductors \(f=f_{N/K}\)
admits precisely two DPF types of the pure metacyclic fields \(N\)
in the associated hexadecuplet with \(m=16\) members.
\begin{enumerate}
\item
\(f^4=5^6\cdot q_1^4q_2^4\) with \(q_i\in\mathbb{P}\), \(q_i\equiv\pm 2\,(\mathrm{mod}\,5)\), both \(q_i\not\equiv\pm 7\,(\mathrm{mod}\,25)\)
gives rise to a hexadecuplet with possibly heterogeneous DPF type \((\varepsilon^x,\gamma^y)\), \(x+y=16\),
\item
\(f^4=5^6\cdot q_1^4q_2^4\) with \(q_i\in\mathbb{P}\), \(q_i\equiv\pm 2\,(\mathrm{mod}\,5)\), only one \(q_i\equiv\pm 7\,(\mathrm{mod}\,25)\)
gives rise to a hexadecuplet with possibly heterogeneous DPF type \((\varepsilon^x,\gamma^y)\), \(x+y=16\).
\end{enumerate}
\end{theorem}

\begin{example}
\label{exm:EpsilonGamma}
It is difficult to find complete hexadecuplets,
whose members are spread rather widely.
The smallest hexadecuplet
\[(30,60,90,120,150,180,240,270,360,540,600,720,810,1350,1620,3750)=\]
\[=(2\cdot 3\cdot 5,\ 2^2\cdot 3\cdot 5,\ 2\cdot 3^2\cdot 5,\ 2^3\cdot 3\cdot 5,\ 
2\cdot 3\cdot 5^2,\ 2^2\cdot 3^2\cdot 5,\ 2^4\cdot 3\cdot 5,\ 2\cdot 3^3\cdot 5,\]
\[2^3\cdot 3^2\cdot 5,\ 2^2\cdot 3^3\cdot 5,\ 2^3\cdot 3\cdot 5^2,\ 2^4\cdot 3^2\cdot 5,\ 
2\cdot 3^4\cdot 5,\ 2\cdot 3^3\cdot 5^2,\ 2^2\cdot 3^4\cdot 5,\ 2\cdot 3\cdot 5^4)\]
belonging to the first infinite sequence
contains the members \(D=1350,1620,3750\) outside of the range of our systematic computations.
We have determined their DPF type separately
and thus discovered a heterogeneous hexadecuplet (in the same order) of type
\[(\varepsilon^3,\gamma^{13})=(\gamma,\gamma,\gamma,\gamma,\gamma,\varepsilon,\varepsilon,\gamma,\gamma,\gamma,\gamma,\gamma,\gamma,\varepsilon,\gamma,\gamma).\]
\end{example}

\begin{corollary}
\label{cor:EpsilonGamma}
The invariants of the similarity classes
defined by the two infinite sequences of conductors in Theorem
\ref{thm:EpsilonGamma}
are given as follows, in the same order.
The statements concerning \(5\)-class groups are only conjectural.
Each sequence splits into two similarity classes. \\
The classes for \(f^4=5^6\cdot q_1^4q_2^4\), both \(q_i\not\equiv\pm 7\,(\mathrm{mod}\,25)\) are:
\begin{equation}
\label{eqn:Prototype30Gamma}
\begin{aligned}
& \mathbf{\lbrack 30\rbrack}, \text{ species \(1\mathrm{a}\), }
(e_0;t,u,v,m;n,s_2,s_4)=(6;2,0,2,16;2,0,0), \\
& \text{type \(\gamma\), }(U,\eta,\zeta;A,I,R)=(2,-,-;3,0,0),
\text{ and } (V_L,V_M,V_N;E)=(0,0,1;6);
\end{aligned}
\end{equation}
\begin{equation}
\label{eqn:Prototype180Epsilon}
\begin{aligned}
& \mathbf{\lbrack 180\rbrack}, \text{ species \(1\mathrm{a}\), }
(e_0;t,u,v,m;n,s_2,s_4)=(6;2,0,2,16;2,0,0), \\
& \text{type \(\varepsilon\), }(U,\eta,\zeta;A,I,R)=(1,\times,-;2,0,0),
\text{ and } (V_L,V_M,V_N;E)=(1,2,4;5).
\end{aligned}
\end{equation}
The classes for \(f^4=5^6\cdot q_1^4q_2^4\), only one \(q_i\equiv\pm 7\,(\mathrm{mod}\,25)\) are:
\begin{equation}
\label{eqn:Prototype70Gamma}
\begin{aligned}
& \mathbf{\lbrack 70\rbrack}, \text{ species \(1\mathrm{a}\), }
(e_0;t,u,v,m;n,s_2,s_4)=(6;2,1,1,16;2,0,0), \\
& \text{type \(\gamma\), }(U,\eta,\zeta;A,I,R)=(2,-,-;3,0,0),
\text{ and } (V_L,V_M,V_N;E)=(0,0,1;6);
\end{aligned}
\end{equation}
\begin{equation}
\label{eqn:Prototype140Epsilon}
\begin{aligned}
& \mathbf{\lbrack 140\rbrack}, \text{ species \(1\mathrm{a}\), }
(e_0;t,u,v,m;n,s_2,s_4)=(6;2,1,1,16;2,0,0), \\
& \text{type \(\varepsilon\), }(U,\eta,\zeta;A,I,R)=(1,\times,-;2,0,0),
\text{ and } (V_L,V_M,V_N;E)=(1,2,4;5).
\end{aligned}
\end{equation}
Only the pure metacyclic fields \(N\) of type \(\gamma\) associated with
\eqref{eqn:Prototype30Gamma}
and
\eqref{eqn:Prototype70Gamma}
are Polya fields.
\end{corollary}

\begin{proof}
(of Theorem
\ref{thm:EpsilonGamma}
and Corollary
\ref{cor:EpsilonGamma})
See
\cite[Thm. 10.7]{Ma2a}.
\end{proof}


\begin{theorem}
\label{thm:Polya1or24mod25}
A pure metacyclic field \(N=\mathbb{Q}(\zeta_5,\sqrt[5]{\ell})\)
with prime radicand \(\ell\equiv\pm 1\,(\mathrm{mod}\,25)\)
has a prime conductor \(f=\ell\), and possesses the Polya property,
regardless of its DPF type and the complexity of its \(5\)-class group structure.
\end{theorem}

\begin{proof}
This is an immediate consequence of
\cite[Thm. 10.5 and Thm. 6.1]{Ma2a},
taking into account that we have the value \(t=1\)
for the number of primes dividing the conductor in the present situation,
and thus the estimate in
\cite[Cor. 4.1]{Ma2a}
yields \(1\le A\le\min(3,t)=\min(3,1)=1\).
For the Polya property we must have \(A=t=1\), according to
\cite[Thm. 10.5]{Ma2a},
which admits the DPF types
\(\alpha_1\), \(\alpha_2\), \(\alpha_3\), \(\delta_1\), \(\delta_2\), \(\zeta_1\), \(\zeta_2\) or \(\vartheta\)
\cite[Thm. 1.3 and Tbl. 1]{Ma2a}.
However,
DPF type \(\alpha_3\) is excluded by 
\cite[Cor. 4.2]{Ma2a},
since the requirement \(s_2+s_4\ge 2\) cannot be fulfilled in our situation where
either \(s_2=0\) and \(s_4=1\) for \(\ell\equiv +1\,(\mathrm{mod}\,25)\)
or \(s_2=1\) and \(s_4=0\) for \(\ell\equiv -1\,(\mathrm{mod}\,25)\).
\end{proof}


\begin{theorem}
\label{thm:Polya1or4mod5}
A pure metacyclic field \(N=\mathbb{Q}(\zeta_5,\sqrt[5]{\ell})\)
with prime radicand \(\ell\equiv\pm 1\,(\mathrm{mod}\,5)\) but \(\ell\not\equiv\pm 1\,(\mathrm{mod}\,25)\)
has a composite conductor \(f^4=5^2\cdot\ell^4\), and
the following conditions are equivalent:
\begin{enumerate}
\item
\(N\) possesses the Polya property.
\item
\((\exists\,\alpha\in L=\mathbb{Q}(\sqrt[5]{\ell}))\,N_{L/\mathbb{Q}}(\alpha)=5\).
\item
The prime ideal \(\mathfrak{p}\in\mathbb{P}_L\) with \(5\mathcal{O}_L=\mathfrak{p}^5\) is principal.
\item
\(N\) is of DPF type either \(\beta_1\) or \(\beta_2\) or \(\varepsilon\).
\end{enumerate}
\end{theorem}

\begin{proof}
This is a consequence of
\cite[Thm. 10.5 and Thm. 6.1]{Ma2a},
taking into account that the prime \(5\) is not included in the current definition of the counter \(t\)
(with value \(t=1\) in the present situation),
and thus the estimate in
\cite[Cor. 4.1]{Ma2a}
must be replaced by \(1\le A\le\min(3,t+1)=\min(3,2)=2\).
For the Polya property we must have \(A=t+1=2\)
\cite[Thm. 10.5]{Ma2a},
which determines the DPF types \(\beta_1\), \(\beta_2\), \(\varepsilon\) or \(\eta\)
\cite[Thm. 1.3 and Tbl. 1]{Ma2a}.
However,
DPF type \(\eta\) is excluded by the prime \(\ell\not\equiv\pm 1,\pm 7\,(\mathrm{mod}\,25)\)
dividing the conductor
(\cite[Thm. 8.1]{Ma2a}).
\end{proof}


Inspired by the last two theorems,
it is worth ones while
to summarize, for each kind of prime radicands,
what is known about the possibilities for differential principal factorizations.

\begin{theorem}
\label{thm:PrimeRadicands}
Let \(N=\mathbb{Q}(\zeta_5,\sqrt[5]{D})\) be a pure metacyclic field
with prime radicand \(D\in\mathbb{P}\).
\begin{enumerate}
\item
If \(D=q\) with \(q\equiv\pm 7\,(\mathrm{mod}\,25)\) or \(q=5\),
then \(N\) is of type \(\vartheta\).
\item
If \(D=\ell\) with \(\ell\equiv -1\,(\mathrm{mod}\,25)\),
then \(N\) is of one of the types \(\delta_2,\zeta_2,\vartheta\).
\item
If \(D=\ell\) with \(\ell\equiv +1\,(\mathrm{mod}\,25)\),
then \(N\) is of one of the types \(\alpha_1,\alpha_2,\delta_1,\delta_2,\zeta_1,\zeta_2,\vartheta\).
\item
If \(D=q\) with \(q\equiv\pm 2\,(\mathrm{mod}\,5)\) but \(q\not\equiv\pm 7\,(\mathrm{mod}\,25)\),
then \(N\) is of type \(\varepsilon\).
\item
If \(D=\ell\) with \(\ell\equiv -1\,(\mathrm{mod}\,5)\) but \(\ell\not\equiv -1\,(\mathrm{mod}\,25)\),
then \(N\) is of one of the types \(\beta_2,\delta_2,\varepsilon\).
\item
If \(D=\ell\) with \(\ell\equiv +1\,(\mathrm{mod}\,5)\) but \(\ell\not\equiv +1\,(\mathrm{mod}\,25)\),
then \(N\) is of one of the types \(\alpha_1,\alpha_2,\beta_1,\beta_2,\delta_1,\delta_2,\varepsilon\).
\end{enumerate}
A pure metacyclic field with prime radicand can never be of any of the types \(\alpha_3,\gamma,\eta\).
\end{theorem}

\begin{proof}
By making use of the bounds
\cite[\S\ 4]{Ma2a}
for \(\mathbb{F}_5\)-dimensions
of spaces of differential principal factors (DPF),
\begin{equation}
\label{eqn:Dimensions}
\begin{aligned}
1 \le A \le & \min(3,t), \\
0 \le I \le & \min(2,2(s_2+s_4)), \\
0 \le R \le & \min(2,4s_4),
\end{aligned}
\end{equation}
we can determine the possible DPF types of
pure metacyclic fields \(N=\mathbb{Q}(\zeta_5,\sqrt[5]{D})\)
with prime radicands \(D\in\mathbb{P}\).
We start with a few general observations.

Firstly, if \(D\equiv\pm 1,\pm 7\,(\mathrm{mod}\,25)\), resp. \(D=5\), is prime,
then \(N\) is of Dedekind species \(2\), resp. \(1\mathrm{a}\),
with prime power conductor \(f=D\), resp.  \(f^4=5^6\), and \(t=1\),
whence \(A=1\) and the types \(\beta_1,\beta_2,\gamma,\varepsilon,\eta\) with \(A\ge 2\) are forbidden.
However, if \(D\not\equiv\pm 1,\pm 7\,(\mathrm{mod}\,25)\) and \(D\ne 5\) is prime,
then the congruence requirement eliminates the types \(\zeta_1,\zeta_2,\eta,\vartheta\),
the field \(N\) is of Dedekind species \(1\mathrm{b}\)
with composite conductor \(f^4=5^2\cdot D^4\), and \(t=2\),
whence \(1\le A\le 2\) and type \(\gamma\) with \(A=3\) is discouraged.
So, the types \(\gamma\) and \(\eta\) are generally forbidden for prime radicands.

Secondly, for a prime radicand \(D\equiv\pm 1\,(\mathrm{mod}\,5)\) which splits in \(M\),
the space of radicals \(\Delta=\langle\sqrt[5]{D}\rangle\)
is a \(1\)-dimensional subspace of absolute DPF
contained in the \(2\)-dimensional space \(\Delta\oplus\Delta^\prime\)
of differential factors generated by the two prime ideals of \(M\) over \(D\).
Consequently, in this special situation there arises an additional constraint
\(I\le 1\) for the dimension of the space of intermediate DPF,
which must be contained in the \(1\)-dimensional complement \(\Delta^\prime\).
This generally excludes type \(\alpha_3\) with \(I=2\) for prime radicands.

\begin{enumerate}
\item
If \(D=q\) with \(q\equiv\pm 7\,(\mathrm{mod}\,25)\),
then \(t=1\), \(s_2=s_4=0\), and thus \(A=1\), \(I=R=0\).
These conditions eliminate the types \(\alpha_1,\alpha_2,\alpha_3,\beta_1,\beta_2,\gamma,\delta_1,\delta_2,\varepsilon,\zeta_1,\zeta_2,\eta\)
with either \(A\ge 2\) or \(I\ge 1\) or \(R\ge 1\),
and only type \(\vartheta\) remains admissible.
\item
If \(D=\ell\) with \(\ell\equiv -1\,(\mathrm{mod}\,25)\),
then \(t=s_2=1\), \(s_4=0\), and thus \(A=1\), \(0\le I\le 1\), \(R=0\),
whence the types \(\alpha_1,\alpha_2,\alpha_3,\beta_1,\beta_2,\gamma,\delta_1,\varepsilon,\zeta_1,\eta\)
with either \(A\ge 2\) or \(I=2\) or \(R\ge 1\) are excluded,
and only the types \(\delta_2,\zeta_2,\vartheta\) remain admissible.
\item
If \(D=\ell\) with \(\ell\equiv +1\,(\mathrm{mod}\,25)\),
then \(t=s_4=1\), \(s_2=0\), and thus \(A=1\), \(0\le I\le 1\), \(0\le R\le 2\),
whence the types \(\alpha_3,\beta_1,\beta_2,\gamma,\varepsilon,\eta\)
with either \(A\ge 2\) or \(I=2\) are excluded,
and only the types \(\alpha_1,\alpha_2,\delta_1,\delta_2,\zeta_1,\zeta_2,\vartheta\) remain admissible.
\item
If \(D=q\) with \(q\equiv\pm 2\,(\mathrm{mod}\,5)\) but \(q\not\equiv\pm 7\,(\mathrm{mod}\,25)\),
then \(t=2\), \(s_2=s_4=0\), and thus \(1\le A\le 2\), \(I=R=0\).
These conditions eliminate the types \(\alpha_1,\alpha_2,\alpha_3,\beta_1,\beta_2,\gamma,\delta_1,\delta_2,\zeta_1,\zeta_2\)
with either \(A=3\) or \(I\ge 1\) or \(R\ge 1\),
and only the types \(\varepsilon,\eta,\vartheta\) remain admissible.
However, the congruence requirement modulo \(25\) discourages the types \(\eta,\vartheta\),
and only type \(\varepsilon\) is possible.
\item
If \(D=\ell\) with \(\ell\equiv -1\,(\mathrm{mod}\,5)\) but \(\ell\not\equiv -1\,(\mathrm{mod}\,25)\),
then \(t=2\), \(s_2=1\), \(s_4=0\), and thus \(1\le A\le 2\), \(0\le I\le 1\), \(R=0\),
whence the types \(\alpha_1,\alpha_2,\alpha_3,\beta_1,\gamma,\delta_1,\zeta_1\)
with either \(A=3\) or \(I=2\) or \(R\ge 1\) are forbidden.
The types \(\zeta_2,\eta,\vartheta\) are excluded by congruence conditions,
and only the types \(\beta_2,\delta_2,\varepsilon\) remain admissible.
\item
If \(D=\ell\) with \(\ell\equiv +1\,(\mathrm{mod}\,5)\) but \(\ell\not\equiv +1\,(\mathrm{mod}\,25)\)
then \(t=2\), \(s_2=0\), \(s_4=1\), and thus \(1\le A\le 2\), \(0\le I\le 1\), \(0\le R\le 2\),
whence the types \(\alpha_3,\gamma\)
with either \(A=3\) or \(I=2\) are forbidden.
The types \(\zeta_1,\zeta_2,\eta,\vartheta\) are excluded by congruence conditions,
and only the types \(\alpha_1,\alpha_2,\beta_1,\beta_2,\delta_1,\delta_2,\varepsilon\) remain admissible.
\qedhere
\end{enumerate}
\end{proof}


\begin{example}
\label{exm:PrimeRadicands}
Concerning numerical realizations of Theorem
\ref{thm:PrimeRadicands},
we refer to Corollary
\ref{cor:InfSimCls}
for the parametrized infinite sequences
\(\mathbf{\lbrack 7\rbrack}\) and \(\mathbf{\lbrack 2\rbrack}\)
which realize item (1) and (4).
(See also Tables
\ref{tbl:Theta}
and
\ref{tbl:Epsilon}
for the types \(\vartheta\) and \(\varepsilon\).)
In all the other cases, there occurs \textit{type splitting}:

The similarity class
\(\mathbf{\lbrack 149\rbrack}\)
partially realizes item (2).
(See Table
\ref{tbl:Delta2}
for the type \(\delta_2\).)
Outside the range of our systematic investigations, we found that
the similarity class \(\mathbf{\lbrack 1049\rbrack}\)
realizes type \(\zeta_2\).
Realizations of the type \(\vartheta\) are unknown up to now.

The similarity classes
\(\mathbf{\lbrack 401\rbrack}\), \(\mathbf{\lbrack 151\rbrack}\) and \(\mathbf{\lbrack 101\rbrack}\)
partially realize item (3).
(See Tables
\ref{tbl:Alpha1},
\ref{tbl:Alpha2}
and
\ref{tbl:Zeta1}
for the types \(\alpha_1\), \(\alpha_2\) and \(\zeta_1\).)
Outside the range of our systematic investigations, we found that
the similarity class \(\mathbf{\lbrack 1151\rbrack}\), resp. \(\mathbf{\lbrack 3251\rbrack}\),
realizes type \(\delta_1\), resp. \(\delta_2\).
Realizations of the types \(\zeta_2\) and \(\vartheta\) are unknown up to now.

The similarity classes
\(\mathbf{\lbrack 139\rbrack}\), \(\mathbf{\lbrack 19\rbrack}\) and \(\mathbf{\lbrack 379\rbrack}\)
completely realize item (5).
(See Tables
\ref{tbl:Beta2},
\ref{tbl:Delta2}
and
\ref{tbl:Epsilon}
for the types \(\beta_2\), \(\delta_2\) and \(\varepsilon\).)

The similarity classes
\(\mathbf{\lbrack 31\rbrack}\), \(\mathbf{\lbrack 11\rbrack}\), \(\mathbf{\lbrack 191\rbrack}\) and \(\mathbf{\lbrack 211\rbrack}\)
partially realize item (6).
(See Tables
\ref{tbl:Alpha1},
\ref{tbl:Alpha2},
\ref{tbl:Beta1}
and
\ref{tbl:Delta1}
for the types \(\alpha_1\), \(\alpha_2\), \(\beta_1\) and \(\delta_1\).)
Realizations of the types \(\beta_2\), \(\delta_2\) and \(\varepsilon\) are unknown up to now.
\end{example}


\subsection{Non-elementary \(5\)-class groups}
\label{ss:NonElem}
\noindent
Although most of the \(5\)-class groups of pure metacyclic fields \(N\),
maximal real subfields \(M\) and pure quintic subfields \(L\)
are elementary abelian,
there occur sparse examples with non-elementary structure.
For instance, we have only \(8\) occurrences within the range \(2\le D<10^3\) of our computations: \\
(1) \(\mathrm{Cl}_5(N)\simeq C_{25}\times C_5^3\), \((V_L,V_M,V_N;E)=(1,2,5*;6)\) for \(D=259=7\cdot 37\) (type \(\gamma\)), \\
(2) \(\mathrm{Cl}_5(N)\simeq C_{25}\times C_5^7\), \(\mathrm{Cl}_5(M)\simeq C_{25}\times C_5^3\),
\((V_L,V_M,V_N;E)=(3,5*,9*;2)\) for \(D=281\) prime \\ (type \(\alpha_1\)), \\
(3) \(\mathrm{Cl}_5(N)\simeq C_{25}\times C_5^5\), \((V_L,V_M,V_N;E)=(2,3,7*;4)\) for \(D=465=3\cdot 5\cdot 31\) (type \(\beta_2\)), \\
(4) \(\mathrm{Cl}_5(N)\simeq C_{25}\times C_5^5\), \((V_L,V_M,V_N;E)=(2,3,7*;4)\) for \(D=473=11\cdot 43\) (type \(\beta_2\)), \\
(5) \(\mathrm{Cl}_5(N)\simeq C_{25}\times C_5^6\), \(\mathrm{Cl}_5(M)\simeq C_{25}\times C_5^2\), \(\mathrm{Cl}_5(L)\simeq C_{25}\),
\((V_L,V_M,V_N;E)=(2*,4*,8*;5)\) \\ for \(D=502=2\cdot 251\) (type \(\beta_1\)), \\
(6) \(\mathrm{Cl}_5(N)\simeq C_{25}\times C_5^5\), \((V_L,V_M,V_N;E)=(2,3,7*;4)\) for \(D=590=2\cdot 5\cdot 59\) (type \(\beta_2\)), \\
(7) \(\mathrm{Cl}_5(N)\simeq C_{25}^2\times C_5^4\), \(\mathrm{Cl}_5(M)\simeq C_{25}\times C_5^2\),
\((V_L,V_M,V_N;E)=(2,4*,8*;5)\) for \(D=620=2^2\cdot 5\cdot 31\) \\ (type \(\beta_1\)), \\
(8) \(\mathrm{Cl}_5(N)\simeq C_{25}\times C_5^4\), \(\mathrm{Cl}_5(M)\simeq C_{25}\times C_5\), \(\mathrm{Cl}_5(L)\simeq C_{25}\),
\((V_L,V_M,V_N;E)=(2*,3*,6*;3)\) \\ for \(D=955=5\cdot 191\) (type \(\alpha_2\)).

However, outside the range of systematic computations, we additionally found: \\
(a) \(\mathrm{Cl}_5(N)\simeq C_{25}^3\times C_5\), \(\mathrm{Cl}_5(M)\simeq C_{25}\times C_5\), \(\mathrm{Cl}_5(L)\simeq C_{25}\),
\((V_L,V_M,V_N;E)=(2*,3*,7*;4)\) \\ for \(D=1049\) prime (type \(\zeta_2\)), \\
(b) \(\mathrm{Cl}_5(N)\simeq C_{25}^2\times C_5^6\), \(\mathrm{Cl}_5(M)\simeq C_{25}\times C_5^3\),
\((V_L,V_M,V_N;E)=(3,5*,10*;3)\) for \(D=3001\) prime \\ (type \(\alpha_2\)), \\
(c) \(\mathrm{Cl}_5(N)\simeq C_{25}^5\times C_5^4\), \(\mathrm{Cl}_5(M)\simeq C_{25}^2\times C_5^3\), \(\mathrm{Cl}_5(L)\simeq C_{25}\times C_5^2\),
\((V_L,V_M,V_N;E)=(4*,7*,14*;3)\) \\ for \(D=3251\) prime (type \(\delta_2\)), \\
(d) \(\mathrm{Cl}_5(N)\simeq C_{25}^2\times C_5^2\), \(\mathrm{Cl}_5(M)\simeq C_{25}\times C_5\), \(\mathrm{Cl}_5(L)\simeq C_{25}\),
\((V_L,V_M,V_N;E)=(2*,3*,6*;3)\) \\ for \(D=5849\) prime (type \(\delta_2\)).

We point out that in all of the last four examples, the normal field \(N\) is a Polya field,
since the radicands \(D\) are primes \(\ell\equiv\pm 1\,(\mathrm{mod}\,25)\), the conductors are primes \(f=\ell\),
and thus all primitive ambiguous ideals are principal, generated by the radical \(\delta=\sqrt[5]{D}\) and its powers.
Consequently, there seems to be no upper bound for the complexity of \(5\)-class groups \(\mathrm{Cl}_5(N)\)
of pure metacyclic Polya fields \(N\) in Theorem
\ref{thm:Polya1or24mod25}.


\subsection{Refinement of DPF types by similarity classes}
\label{ss:Refinement}
\noindent
Based on the definition of similarity classes and prototypes in \S\
\ref{ss:Prototypes},
on the explicit listing of all prototypes in the range between \(2\) and \(10^3\) in the Tables
\ref{tbl:Prototypes1}
---
\ref{tbl:Prototypes4},
and on theoretical foundations in \S\
\ref{ss:Theorems},
we are now in the position to establish
the intended refinement of our \(13\) differential principal factorization types
into similarity classes in the Tables
\ref{tbl:Alpha1}
---
\ref{tbl:Theta},
as far as the range of our computations for normalized radicands \(2\le D<10^3\)
is concerned.
The cardinalities \(\mathbf{\lvert M\rvert}\) refine the statistical evaluation in Table
\ref{tbl:Statistics}.

DPF types are characterized by the multiplet \((U,\eta,\zeta;A,I,R)\),
refined Dedekind species, S, by the multiplet \((e_0;t,u,v,m;n,s_2,s_4)\),
and \(5\)-class groups by the multiplet \((V_L,V_M,V_N;E)\).


\renewcommand{\arraystretch}{1.0}

\begin{table}[ht]
\caption{Splitting of type \(\alpha_1\), \((U,\eta,\zeta;A,I,R)=(2,-,-;1,0,2)\): \(5\) similarity classes}
\label{tbl:Alpha1}
\begin{center}
\begin{tabular}{|r|cc|cccr|ccc|cccc|rr|}
\hline
 No. & S  & \(e_0\) & \(t\) & \(u\) & \(v\) & \(m\) & \(n\) & \(s_2\) & \(s_4\) & \(V_L\) & \(V_M\) & \(V_N\) & \(E\) &   \(\mathbf{M}\) & \(\mathbf{\lvert M\rvert}\) \\
\hline
   1 & 1b &   \(2\) & \(1\) & \(0\) & \(1\) & \(1\) & \(0\) &   \(0\) &   \(1\) &   \(2\) &   \(3\) &   \(5\) & \(2\) &  \(\mathbf{31}\) &                       \(2\) \\
   2 & 1a &   \(6\) & \(1\) & \(0\) & \(1\) & \(4\) & \(0\) &   \(0\) &   \(1\) &   \(2\) &   \(3\) &   \(5\) & \(2\) & \(\mathbf{155}\) &                       \(2\) \\
   3 & 1b &   \(2\) & \(1\) & \(0\) & \(1\) & \(1\) & \(0\) &   \(0\) &   \(1\) &   \(3\) &  \(5*\) &  \(9*\) & \(2\) & \(\mathbf{281}\) &                       \(1\) \\
   4 & 1b &   \(2\) & \(2\) & \(0\) & \(2\) & \(3\) & \(0\) &   \(0\) &   \(2\) &   \(3\) &   \(5\) &   \(9\) & \(2\) & \(\mathbf{341}\) &                       \(2\) \\
   5 & 2  &   \(0\) & \(1\) & \(1'\)& \(0\) & \(1\) & \(0\) &   \(0\) &   \(1'\)&   \(2\) &   \(3\) &   \(5\) & \(2\) & \(\mathbf{401}\) &                       \(2\) \\
\hline
\end{tabular}
\end{center}
\end{table}

DPF type \(\alpha_1\) splits into
\(3\) similarity classes in the ground state \((V_L,V_M,V_N)=(2,3,5)\) and
\(2\) similarity classes in the first excited state \((V_L,V_M,V_N)=(3,5,9)\).
Summing up the partial frequencies \(6+3\) of these states in Table
\ref{tbl:Alpha1}
yields the modest absolute frequency \(9\) of type \(\alpha_1\) in the range \(2\le D<10^3\),
as given in Table
\ref{tbl:Statistics}.
The logarithmic subfield unit index of type \(\alpha_1\)
is restricted to the single value \(E=2\).
Type \(\alpha_1\) is the unique type with
\(2\)-dimensional relative principal factorization, \(R=2\).


\renewcommand{\arraystretch}{1.0}

\begin{table}[ht]
\caption{Splitting of type \(\alpha_2\), \((U,\eta,\zeta;A,I,R)=(2,-,-;1,1,1)\): \(22\) similarity classes}
\label{tbl:Alpha2}
\begin{center}
\begin{tabular}{|r|cc|cccr|ccc|cccc|rr|}
\hline
 No. & S  & \(e_0\) & \(t\) & \(u\) & \(v\) &  \(m\) & \(n\) & \(s_2\) & \(s_4\) & \(V_L\) & \(V_M\) & \(V_N\) & \(E\) &   \(\mathbf{M}\) & \(\mathbf{\lvert M\rvert}\) \\
\hline
   1 & 1b &   \(2\) & \(1\) & \(0\) & \(1\) &  \(1\) & \(0\) &   \(0\) &   \(1\) &   \(1\) &   \(1\) &   \(2\) & \(3\) &  \(\mathbf{11}\) &                      \(14\) \\
   2 & 1b &   \(2\) & \(2\) & \(0\) & \(2\) &  \(3\) & \(1\) &   \(0\) &   \(1\) &   \(2\) &   \(2\) &   \(4\) & \(1\) &  \(\mathbf{33}\) &                       \(8\) \\
   3 & 1a &   \(6\) & \(1\) & \(0\) & \(1\) &  \(4\) & \(0\) &   \(0\) &   \(1\) &   \(1\) &   \(1\) &   \(2\) & \(3\) &  \(\mathbf{55}\) &                       \(6\) \\
   4 & 1b &   \(2\) & \(2\) & \(0\) & \(2\) &  \(3\) & \(1\) &   \(0\) &   \(1\) &   \(1\) &   \(1\) &   \(2\) & \(3\) &  \(\mathbf{82}\) &                      \(15\) \\
   5 & 1b &   \(2\) & \(2\) & \(0\) & \(2\) &  \(3\) & \(1\) &   \(0\) &   \(1\) &   \(2\) &   \(3\) &   \(6\) & \(3\) & \(\mathbf{123}\) &                       \(2\) \\
   6 & 1b &   \(2\) & \(1\) & \(0\) & \(1\) &  \(1\) & \(0\) &   \(0\) &   \(1\) &   \(2\) &   \(2\) &   \(4\) & \(1\) & \(\mathbf{131}\) &                       \(6\) \\
   7 & 2  &   \(0\) & \(1\) & \(1'\)& \(0\) &  \(1\) & \(0\) &   \(0\) &   \(1'\)&   \(1\) &   \(1\) &   \(2\) & \(3\) & \(\mathbf{151}\) &                       \(3\) \\
   8 & 1a &   \(6\) & \(1\) & \(0\) & \(1\) &  \(4\) & \(0\) &   \(0\) &   \(1\) &   \(2\) &   \(2\) &   \(4\) & \(1\) & \(\mathbf{275}\) &                       \(2\) \\
   9 & 1b &   \(2\) & \(2\) & \(1\) & \(1\) &  \(4\) & \(1\) &   \(0\) &   \(1\) &   \(2\) &   \(2\) &   \(4\) & \(1\) & \(\mathbf{287}\) &                       \(1\) \\
  10 & 2  &   \(0\) & \(3\) & \(0\) & \(3\) &  \(3\) & \(1\) &   \(1\) &   \(1\) &   \(2\) &   \(3\) &   \(6\) & \(3\) & \(\mathbf{418}\) &                       \(1\) \\
  11 & 2  &   \(0\) & \(2\) & \(0\) & \(2\) &  \(1\) & \(0\) &   \(0\) &   \(2\) &   \(2\) &   \(3\) &   \(6\) & \(3\) & \(\mathbf{451}\) &                       \(1\) \\
  12 & 1a &   \(6\) & \(2\) & \(0\) & \(2\) & \(16\) & \(1\) &   \(0\) &   \(1\) &   \(2\) &   \(3\) &   \(6\) & \(3\) & \(\mathbf{550}\) &                       \(1\) \\
  13 & 1b &   \(2\) & \(2\) & \(0\) & \(2\) &  \(3\) & \(0\) &   \(0\) &   \(2\) &   \(3\) &   \(4\) &   \(8\) & \(1\) & \(\mathbf{671}\) &                       \(1\) \\
  14 & 2  &   \(0\) & \(3\) & \(0\) & \(3\) &  \(3\) & \(1\) &   \(0\) &   \(2\) &   \(2\) &   \(3\) &   \(6\) & \(3\) & \(\mathbf{682}\) &                       \(1\) \\
  15 & 1b &   \(2\) & \(1\) & \(0\) & \(1\) &  \(1\) & \(0\) &   \(0\) &   \(1\) &   \(3\) &   \(4\) &   \(8\) & \(1\) & \(\mathbf{691}\) &                       \(1\) \\
  16 & 1a &   \(6\) & \(2\) & \(0\) & \(2\) & \(16\) & \(1\) &   \(0\) &   \(1\) &   \(2\) &   \(2\) &   \(4\) & \(1\) & \(\mathbf{710}\) &                       \(4\) \\
  17 & 2  &   \(0\) & \(1\) & \(1'\)& \(0\) &  \(1\) & \(0\) &   \(0\) &   \(1'\)&   \(2\) &   \(2\) &   \(4\) & \(1\) & \(\mathbf{751}\) &                       \(1\) \\
  18 & 1b &   \(2\) & \(2\) & \(0\) & \(2\) &  \(3\) & \(0\) &   \(1\) &   \(1\) &   \(3\) &   \(4\) &   \(8\) & \(1\) & \(\mathbf{779}\) &                       \(1\) \\
  19 & 1b &   \(2\) & \(2\) & \(1'\)& \(1\) &  \(4\) & \(1\) &   \(0\) &   \(1'\)&   \(2\) &   \(2\) &   \(4\) & \(1\) & \(\mathbf{808}\) &                       \(2\) \\
  20 & 1b &   \(2\) & \(3\) & \(1\) & \(2\) & \(12\) & \(2\) &   \(0\) &   \(1\) &   \(3\) &   \(4\) &   \(8\) & \(1\) & \(\mathbf{861}\) &                       \(1\) \\
  21 & 1a &   \(6\) & \(1\) & \(0\) & \(1\) &  \(4\) & \(0\) &   \(0\) &   \(1\) &  \(2*\) &  \(3*\) &  \(6*\) & \(3\) & \(\mathbf{955}\) &                       \(1\) \\
  22 & 2  &   \(0\) & \(2\) & \(0\) & \(2\) &  \(1\) & \(1\) &   \(0\) &   \(1\) &   \(1\) &   \(1\) &   \(2\) & \(3\) & \(\mathbf{982}\) &                       \(2\) \\
\hline
\end{tabular}
\end{center}
\end{table}

The logarithmic subfield unit index of DPF type \(\alpha_2\)
can take two values, either \(E=3\) or \(E=1\).
Type \(\alpha_2\) with \(E=3\) splits into
\(5\) similarity classes in the ground state \((V_L,V_M,V_N)=(1,1,2)\) and
\(6\) similarity classes in the first excited state \((V_L,V_M,V_N)=(2,3,6)\).
Type \(\alpha_2\) with \(E=1\) splits into
\(7\) similarity classes in the ground state \((V_L,V_M,V_N)=(2,2,4)\) and
\(4\) similarity classes in the first excited state \((V_L,V_M,V_N)=(3,4,8)\).
Summing up the partial frequencies \(40+7\), resp. \(24+4\), of these states in Table
\ref{tbl:Alpha2}
yields the considerable absolute frequency \(75\) of type \(\alpha_2\) in the range \(2\le D<10^3\),
as given in Table
\ref{tbl:Statistics}.
Type \(\alpha_2\) is the unique type with
mixed intermediate and relative principal factorization, \(I=R=1\).


\renewcommand{\arraystretch}{1.0}

\begin{table}[ht]
\caption{Splitting of type \(\alpha_3\), \((U,\eta,\zeta;A,I,R)=(2,-,-;1,2,0)\): \(5\) similarity classes}
\label{tbl:Alpha3}
\begin{center}
\begin{tabular}{|r|cc|cccr|ccc|cccc|rr|}
\hline
 No. & S  & \(e_0\) & \(t\) & \(u\) & \(v\) &  \(m\) & \(n\) & \(s_2\) & \(s_4\) & \(V_L\) & \(V_M\) & \(V_N\) & \(E\) &   \(\mathbf{M}\) & \(\mathbf{\lvert M\rvert}\) \\
\hline
   1 & 1b &   \(2\) & \(2\) & \(0\) & \(2\) &  \(3\) & \(0\) &   \(1\) &   \(1\) &   \(2\) &   \(2\) &   \(5\) & \(2\) & \(\mathbf{319}\) &                       \(3\) \\
   2 & 2  &   \(0\) & \(2\) & \(0\) & \(2\) &  \(1\) & \(0\) &   \(2\) &   \(0\) &   \(2\) &   \(2\) &   \(5\) & \(2\) & \(\mathbf{551}\) &                       \(1\) \\
   3 & 1b &   \(2\) & \(3\) & \(0\) & \(3\) & \(13\) & \(1\) &   \(1\) &   \(1\) &   \(3\) &   \(4\) &   \(9\) & \(2\) & \(\mathbf{627}\) &                       \(1\) \\
   4 & 2  &   \(0\) & \(2\) & \(0\) & \(2\) &  \(1\) & \(0\) &   \(1\) &   \(1\) &   \(2\) &   \(2\) &   \(5\) & \(2\) & \(\mathbf{649}\) &                       \(2\) \\
   5 & 2  &   \(0\) & \(3\) & \(0\) & \(3\) &  \(3\) & \(1\) &   \(1\) &   \(1\) &   \(2\) &   \(2\) &   \(5\) & \(2\) & \(\mathbf{957}\) &                       \(1\) \\
\hline
\end{tabular}
\end{center}
\end{table}

DPF type \(\alpha_3\) splits into
\(4\) similarity classes in the ground state \((V_L,V_M,V_N)=(2,2,5)\) and
\(1\) similarity class in the first excited state \((V_L,V_M,V_N)=(3,4,9)\).
Summing up the partial frequencies \(7+1\) of these states in Table
\ref{tbl:Alpha3}
yields the modest absolute frequency \(8\) of type \(\alpha_3\) in the range \(2\le D<10^3\),
as given in Table
\ref{tbl:Statistics}.
The logarithmic subfield unit index of type \(\alpha_3\)
is restricted to the unique value \(E=2\).
Type \(\alpha_3\) is the unique type with
\(2\)-dimensional intermediate principal factorization, \(I=2\).


\renewcommand{\arraystretch}{1.0}

\begin{table}[ht]
\caption{Splitting of type \(\beta_1\), \((U,\eta,\zeta;A,I,R)=(2,-,-;2,0,1)\): \(10\) similarity classes}
\label{tbl:Beta1}
\begin{center}
\begin{tabular}{|r|cc|cccr|ccc|cccc|rr|}
\hline
 No. & S  & \(e_0\) & \(t\) & \(u\) & \(v\) &  \(m\) & \(n\) & \(s_2\) & \(s_4\) & \(V_L\) & \(V_M\) & \(V_N\) & \(E\) &   \(\mathbf{M}\) & \(\mathbf{\lvert M\rvert}\) \\
\hline
   1 & 1b &   \(2\) & \(3\) & \(0\) & \(3\) & \(13\) & \(2\) &   \(0\) &   \(1\) &   \(2\) &   \(3\) &   \(6\) & \(3\) & \(\mathbf{186}\) &                       \(7\) \\
   2 & 1b &   \(2\) & \(1\) & \(0\) & \(1\) &  \(1\) & \(0\) &   \(0\) &   \(1\) &   \(1\) &   \(2\) &   \(4\) & \(5\) & \(\mathbf{191}\) &                       \(3\) \\
   3 & 1b &   \(2\) & \(2\) & \(0\) & \(2\) &  \(3\) & \(1\) &   \(0\) &   \(1\) &   \(1\) &   \(2\) &   \(4\) & \(5\) & \(\mathbf{253}\) &                       \(4\) \\
   4 & 1b &   \(2\) & \(2\) & \(1'\)& \(1\) &  \(4\) & \(1\) &   \(0\) &   \(1'\)&   \(1\) &   \(2\) &   \(4\) & \(5\) & \(\mathbf{302}\) &                       \(2\) \\
   5 & 2  &   \(0\) & \(2\) & \(0\) & \(2\) &  \(1\) & \(1\) &   \(0\) &   \(1\) &   \(1\) &   \(2\) &   \(4\) & \(5\) & \(\mathbf{482}\) &                       \(2\) \\
   6 & 1b &   \(2\) & \(2\) & \(1'\)& \(1\) &  \(4\) & \(1\) &   \(0\) &   \(1'\)&  \(2*\) &  \(4*\) &  \(8*\) & \(5\) & \(\mathbf{502}\) &                       \(1\) \\
   7 & 1b &   \(2\) & \(2\) & \(0\) & \(2\) &  \(3\) & \(1\) &   \(0\) &   \(1\) &   \(2\) &   \(3\) &   \(6\) & \(3\) & \(\mathbf{517}\) &                       \(1\) \\
   8 & 1a &   \(6\) & \(2\) & \(0\) & \(2\) & \(16\) & \(1\) &   \(0\) &   \(1\) &   \(2\) &  \(4*\) &  \(8*\) & \(5\) & \(\mathbf{620}\) &                       \(1\) \\
   9 & 2  &   \(0\) & \(3\) & \(1\) & \(2\) &  \(4\) & \(2\) &   \(0\) &   \(1\) &   \(2\) &   \(3\) &   \(6\) & \(3\) & \(\mathbf{693}\) &                       \(1\) \\
  10 & 1a &   \(6\) & \(2\) & \(0\) & \(2\) & \(16\) & \(1\) &   \(0\) &   \(1\) &   \(1\) &   \(2\) &   \(4\) & \(5\) & \(\mathbf{825}\) &                       \(1\) \\
\hline
\end{tabular}
\end{center}
\end{table}

The logarithmic subfield unit index of DPF type \(\beta_1\)
can take two values, either \(E=3\) or \(E=5\).
Type \(\beta_1\) with \(E=3\) consists of
\(3\) similarity classes in the ground state \((V_L,V_M,V_N)=(2,3,6)\).
Type \(\beta_1\) with \(E=5\) splits into
\(5\) similarity classes in the ground state \((V_L,V_M,V_N)=(1,2,4)\) and
\(2\) similarity classes in the first excited state \((V_L,V_M,V_N)=(2,4,8)\).
Summing up the partial frequencies \(9\), resp. \(12+2\), of these states in Table
\ref{tbl:Beta1}
yields the modest absolute frequency \(23\) of type \(\beta_1\) in the range \(2\le D<10^3\),
as given in Table
\ref{tbl:Statistics}.
Type \(\beta_1\) is the unique type with
mixed absolute and relative principal factorization, \(A=2\) and \(R=1\).


\renewcommand{\arraystretch}{1.0}

\begin{table}[ht]
\caption{Splitting of type \(\beta_2\), \((U,\eta,\zeta;A,I,R)=(2,-,-;2,1,0)\): \(25\) similarity classes}
\label{tbl:Beta2}
\begin{center}
\begin{tabular}{|r|cc|cccr|ccc|cccc|rr|}
\hline
 No. & S  & \(e_0\) & \(t\) & \(u\) & \(v\) &  \(m\) & \(n\) & \(s_2\) & \(s_4\) & \(V_L\) & \(V_M\) & \(V_N\) & \(E\) &    \(\mathbf{M}\) & \(\mathbf{\lvert M\rvert}\) \\
\hline
   1 & 1b &   \(2\) & \(2\) & \(0\) & \(2\) &  \(3\) & \(1\) &   \(0\) &   \(1\) &   \(1\) &   \(1\) &   \(3\) & \(4\) &   \(\mathbf{22}\) &                      \(35\) \\
   2 & 1b &   \(2\) & \(2\) & \(0\) & \(2\) &  \(3\) & \(1\) &   \(1\) &   \(0\) &   \(1\) &   \(1\) &   \(3\) & \(4\) &   \(\mathbf{38}\) &                      \(44\) \\
   3 & 1b &   \(2\) & \(2\) & \(1\) & \(1\) &  \(4\) & \(1\) &   \(0\) &   \(1\) &   \(1\) &   \(1\) &   \(3\) & \(4\) &   \(\mathbf{77}\) &                       \(7\) \\
   4 & 1a &   \(6\) & \(2\) & \(0\) & \(2\) & \(16\) & \(1\) &   \(0\) &   \(1\) &   \(1\) &   \(1\) &   \(3\) & \(4\) &  \(\mathbf{110}\) &                      \(11\) \\
   5 & 2  &   \(0\) & \(3\) & \(0\) & \(3\) &  \(3\) & \(2\) &   \(0\) &   \(1\) &   \(1\) &   \(1\) &   \(3\) & \(4\) &  \(\mathbf{132}\) &                       \(5\) \\
   6 & 1b &   \(2\) & \(2\) & \(1\) & \(1\) &  \(4\) & \(1\) &   \(1\) &   \(0\) &   \(1\) &   \(1\) &   \(3\) & \(4\) &  \(\mathbf{133}\) &                       \(7\) \\
   7 & 1b &   \(2\) & \(1\) & \(0\) & \(1\) &  \(1\) & \(0\) &   \(1\) &   \(0\) &   \(1\) &   \(1\) &   \(3\) & \(4\) &  \(\mathbf{139}\) &                       \(4\) \\
   8 & 1b &   \(2\) & \(3\) & \(1\) & \(2\) & \(12\) & \(2\) &   \(0\) &   \(1\) &   \(2\) &   \(3\) &   \(7\) & \(4\) &  \(\mathbf{154}\) &                       \(3\) \\
   9 & 2  &   \(0\) & \(3\) & \(0\) & \(3\) &  \(3\) & \(2\) &   \(1\) &   \(0\) &   \(1\) &   \(1\) &   \(3\) & \(4\) &  \(\mathbf{174}\) &                       \(3\) \\
  10 & 1a &   \(6\) & \(2\) & \(0\) & \(2\) & \(16\) & \(1\) &   \(1\) &   \(0\) &   \(1\) &   \(1\) &   \(3\) & \(4\) &  \(\mathbf{190}\) &                       \(9\) \\
  11 & 1b &   \(2\) & \(2\) & \(1'\)& \(1\) &  \(4\) & \(1\) &   \(0\) &   \(1'\)&   \(1\) &   \(1\) &   \(3\) & \(4\) &  \(\mathbf{202}\) &                       \(6\) \\
  12 & 1b &   \(2\) & \(2\) & \(0\) & \(2\) &  \(3\) & \(0\) &   \(1\) &   \(1\) &   \(2\) &   \(3\) &   \(7\) & \(4\) &  \(\mathbf{209}\) &                       \(2\) \\
  13 & 1b &   \(2\) & \(3\) & \(0\) & \(3\) & \(13\) & \(2\) &   \(0\) &   \(1\) &   \(2\) &   \(3\) &   \(7\) & \(4\) &  \(\mathbf{286}\) &                       \(3\) \\
  14 & 1a &   \(6\) & \(2\) & \(1\) & \(1\) & \(16\) & \(1\) &   \(0\) &   \(1\) &   \(1\) &   \(1\) &   \(3\) & \(4\) &  \(\mathbf{385}\) &                       \(1\) \\
  15 & 1b &   \(2\) & \(2\) & \(1'\)& \(1\) &  \(4\) & \(1\) &   \(1'\)&   \(0\) &   \(1\) &   \(1\) &   \(3\) & \(4\) &  \(\mathbf{398}\) &                       \(7\) \\
  16 & 2  &   \(0\) & \(3\) & \(1\) & \(2\) &  \(4\) & \(2\) &   \(1\) &   \(0\) &   \(1\) &   \(1\) &   \(3\) & \(4\) &  \(\mathbf{399}\) &                       \(2\) \\
  17 & 1a &   \(6\) & \(2\) & \(0\) & \(2\) & \(16\) & \(1\) &   \(0\) &   \(1\) &   \(2\) &   \(3\) &  \(7*\) & \(4\) &  \(\mathbf{465}\) &                       \(1\) \\
  18 & 1b &   \(2\) & \(2\) & \(1\) & \(1\) &  \(4\) & \(1\) &   \(0\) &   \(1\) &   \(2\) &   \(3\) &  \(7*\) & \(4\) &  \(\mathbf{473}\) &                       \(1\) \\
  19 & 2  &   \(0\) & \(3\) & \(1\) & \(2\) &  \(4\) & \(2\) &   \(0\) &   \(1\) &   \(1\) &   \(1\) &   \(3\) & \(4\) &  \(\mathbf{574}\) &                       \(3\) \\
  20 & 1a &   \(6\) & \(2\) & \(0\) & \(2\) & \(16\) & \(1\) &   \(1\) &   \(0\) &   \(2\) &   \(3\) &  \(7*\) & \(4\) &  \(\mathbf{590}\) &                       \(1\) \\
  21 & 1b &   \(2\) & \(3\) & \(1\) & \(2\) & \(12\) & \(2\) &   \(1\) &   \(0\) &   \(2\) &   \(3\) &   \(7\) & \(4\) &  \(\mathbf{609}\) &                       \(1\) \\
  22 & 1b &   \(2\) & \(3\) & \(0\) & \(3\) & \(13\) & \(1\) &   \(1\) &   \(1\) &   \(2\) &   \(3\) &   \(7\) & \(4\) &  \(\mathbf{638}\) &                       \(2\) \\
  23 & 1a &   \(6\) & \(2\) & \(1\) & \(1\) & \(16\) & \(1\) &   \(1\) &   \(0\) &   \(1\) &   \(1\) &   \(3\) & \(4\) &  \(\mathbf{665}\) &                       \(1\) \\
  24 & 2  &   \(0\) & \(2\) & \(0\) & \(2\) &  \(1\) & \(1\) &   \(1\) &   \(0\) &   \(1\) &   \(1\) &   \(3\) & \(4\) &  \(\mathbf{893}\) &                       \(1\) \\
  25 & 1b &   \(2\) & \(3\) & \(0\) & \(3\) & \(13\) & \(1\) &   \(0\) &   \(2\) &   \(2\) &   \(3\) &   \(7\) & \(4\) &  \(\mathbf{902}\) &                       \(1\) \\
\hline
\end{tabular}
\end{center}
\end{table}

DPF type \(\beta_2\) splits into
\(16\) similarity classes in the ground state \((V_L,V_M,V_N)=(1,1,3)\) and
\(9\) similarity classes in the first excited state \((V_L,V_M,V_N)=(2,3,7)\).
Summing up the partial frequencies \(146+15\) of these states in Table
\ref{tbl:Beta2}
yields the high absolute frequency \(161\) of type \(\beta_2\) in the range \(2\le D<10^3\),
as given in Table
\ref{tbl:Statistics}.
The logarithmic subfield unit index of type \(\beta_2\)
is restricted to the unique value \(E=4\).
Type \(\beta_2\) is the unique type with
mixed absolute and intermediate principal factorization, \(A=2\) and \(I=1\).


\renewcommand{\arraystretch}{1.0}

\begin{table}[ht]
\caption{Splitting of type \(\gamma\), \((U,\eta,\zeta;A,I,R)=(2,-,-;3,0,0)\): \(29\) similarity classes}
\label{tbl:Gamma}
\begin{center}
\begin{tabular}{|r|cc|cccr|ccc|cccc|rr|}
\hline
 No. & S  & \(e_0\) & \(t\) & \(u\) & \(v\) &  \(m\) & \(n\) & \(s_2\) & \(s_4\) & \(V_L\) & \(V_M\) & \(V_N\) & \(E\) &   \(\mathbf{M}\) & \(\mathbf{\lvert M\rvert}\) \\
\hline
   1 & 1b &   \(2\) & \(2\) & \(0\) & \(2\) &  \(3\) & \(2\) &   \(0\) &   \(0\) &   \(0\) &   \(0\) &   \(1\) & \(6\) &   \(\mathbf{6}\) &                      \(77\) \\
   2 & 1b &   \(2\) & \(2\) & \(1\) & \(1\) &  \(4\) & \(2\) &   \(0\) &   \(0\) &   \(0\) &   \(0\) &   \(1\) & \(6\) &  \(\mathbf{14}\) &                      \(44\) \\
   3 & 1a &   \(6\) & \(2\) & \(0\) & \(2\) & \(16\) & \(2\) &   \(0\) &   \(0\) &   \(0\) &   \(0\) &   \(1\) & \(6\) &  \(\mathbf{30}\) &                      \(37\) \\
   4 & 1b &   \(2\) & \(3\) & \(1\) & \(2\) & \(12\) & \(3\) &   \(0\) &   \(0\) &   \(1\) &   \(2\) &   \(5\) & \(6\) &  \(\mathbf{42}\) &                      \(22\) \\
   5 & 1b &   \(2\) & \(3\) & \(0\) & \(3\) & \(13\) & \(2\) &   \(0\) &   \(1\) &   \(1\) &   \(2\) &   \(5\) & \(6\) &  \(\mathbf{66}\) &                      \(17\) \\
   6 & 1a &   \(6\) & \(2\) & \(1\) & \(1\) & \(16\) & \(2\) &   \(0\) &   \(0\) &   \(0\) &   \(0\) &   \(1\) & \(6\) &  \(\mathbf{70}\) &                      \(14\) \\
   7 & 1b &   \(2\) & \(3\) & \(0\) & \(3\) & \(13\) & \(3\) &   \(0\) &   \(0\) &   \(1\) &   \(2\) &   \(5\) & \(6\) &  \(\mathbf{78}\) &                      \(37\) \\
   8 & 1b &   \(2\) & \(3\) & \(0\) & \(3\) & \(13\) & \(2\) &   \(1\) &   \(0\) &   \(1\) &   \(2\) &   \(5\) & \(6\) & \(\mathbf{114}\) &                      \(20\) \\
   9 & 2  &   \(0\) & \(3\) & \(1\) & \(2\) &  \(4\) & \(3\) &   \(0\) &   \(0\) &   \(0\) &   \(0\) &   \(1\) & \(6\) & \(\mathbf{126}\) &                       \(6\) \\
  10 & 1a &   \(6\) & \(3\) & \(1\) & \(2\) & \(64\) & \(3\) &   \(0\) &   \(0\) &   \(1\) &   \(2\) &   \(5\) & \(6\) & \(\mathbf{210}\) &                       \(5\) \\
  11 & 1b &   \(2\) & \(3\) & \(1\) & \(2\) & \(12\) & \(2\) &   \(0\) &   \(1\) &   \(1\) &   \(2\) &   \(5\) & \(6\) & \(\mathbf{231}\) &                       \(5\) \\
  12 & 1b &   \(2\) & \(2\) & \(0\) & \(2\) &  \(3\) & \(1\) &   \(1\) &   \(0\) &   \(1\) &   \(2\) &   \(5\) & \(6\) & \(\mathbf{247}\) &                       \(2\) \\
  13 & 1b &   \(2\) & \(2\) & \(1\) & \(1\) &  \(4\) & \(2\) &   \(0\) &   \(0\) &   \(1\) &   \(2\) &   \(5\) & \(6\) & \(\mathbf{259}\) &                       \(1\) \\
  14 & 1b &   \(2\) & \(3\) & \(1\) & \(2\) & \(12\) & \(2\) &   \(1\) &   \(0\) &   \(1\) &   \(2\) &   \(5\) & \(6\) & \(\mathbf{266}\) &                       \(3\) \\
  15 & 2  &   \(0\) & \(3\) & \(0\) & \(3\) &  \(3\) & \(3\) &   \(0\) &   \(0\) &   \(0\) &   \(0\) &   \(1\) & \(6\) & \(\mathbf{276}\) &                      \(10\) \\
  16 & 1a &   \(6\) & \(2\) & \(0\) & \(2\) & \(16\) & \(1\) &   \(1\) &   \(0\) &   \(1\) &   \(2\) &   \(5\) & \(6\) & \(\mathbf{285}\) &                       \(1\) \\
  17 & 1a &   \(6\) & \(3\) & \(0\) & \(3\) & \(64\) & \(2\) &   \(0\) &   \(1\) &   \(2\) &   \(4\) &   \(9\) & \(6\) & \(\mathbf{330}\) &                       \(2\) \\
  18 & 1a &   \(6\) & \(3\) & \(0\) & \(3\) & \(64\) & \(3\) &   \(0\) &   \(0\) &   \(1\) &   \(2\) &   \(5\) & \(6\) & \(\mathbf{390}\) &                       \(4\) \\
  19 & 1b &   \(2\) & \(4\) & \(1\) & \(3\) & \(52\) & \(3\) &   \(0\) &   \(1\) &   \(2\) &   \(4\) &   \(9\) & \(6\) & \(\mathbf{462}\) &                       \(1\) \\
  20 & 1b &   \(2\) & \(4\) & \(1\) & \(3\) & \(52\) & \(4\) &   \(0\) &   \(0\) &   \(2\) &   \(4\) &   \(9\) & \(6\) & \(\mathbf{546}\) &                       \(3\) \\
  21 & 1a &   \(6\) & \(3\) & \(0\) & \(3\) & \(64\) & \(2\) &   \(1\) &   \(0\) &   \(1\) &   \(2\) &   \(5\) & \(6\) & \(\mathbf{570}\) &                       \(2\) \\
  22 & 1b &   \(2\) & \(3\) & \(2\) & \(1\) & \(16\) & \(3\) &   \(0\) &   \(0\) &   \(1\) &   \(2\) &   \(5\) & \(6\) & \(\mathbf{602}\) &                       \(2\) \\
  23 & 1b &   \(2\) & \(3\) & \(1'\)& \(2\) & \(12\) & \(2\) &   \(0\) &   \(1'\)&   \(1\) &   \(2\) &   \(5\) & \(6\) & \(\mathbf{606}\) &                       \(2\) \\
  24 & 1a &   \(6\) & \(3\) & \(0\) & \(3\) & \(64\) & \(2\) &   \(0\) &   \(1\) &   \(1\) &   \(2\) &   \(5\) & \(6\) & \(\mathbf{660}\) &                       \(2\) \\
  25 & 1a &   \(6\) & \(3\) & \(1\) & \(2\) & \(64\) & \(2\) &   \(0\) &   \(1\) &   \(1\) &   \(2\) &   \(5\) & \(6\) & \(\mathbf{770}\) &                       \(1\) \\
  26 & 1b &   \(2\) & \(4\) & \(1\) & \(3\) & \(52\) & \(3\) &   \(1\) &   \(0\) &   \(2\) &   \(4\) &   \(9\) & \(6\) & \(\mathbf{798}\) &                       \(1\) \\
  27 & 1b &   \(2\) & \(4\) & \(0\) & \(4\) & \(51\) & \(3\) &   \(0\) &   \(1\) &   \(2\) &   \(4\) &   \(9\) & \(6\) & \(\mathbf{858}\) &                       \(1\) \\
  28 & 1b &   \(2\) & \(3\) & \(1'\)& \(2\) & \(12\) & \(2\) &   \(1'\)&   \(0\) &   \(1\) &   \(2\) &   \(5\) & \(6\) & \(\mathbf{894}\) &                       \(1\) \\
  29 & 2  &   \(0\) & \(4\) & \(1\) & \(3\) & \(12\) & \(3\) &   \(0\) &   \(1\) &   \(1\) &   \(2\) &   \(5\) & \(6\) & \(\mathbf{924}\) &                       \(1\) \\
\hline
\end{tabular}
\end{center}
\end{table}

DPF type \(\gamma\) splits into
\(6\) similarity classes in the ground state \((V_L,V_M,V_N)=(0,0,1)\),
\(16\) similarity classes in the first excited state \((V_L,V_M,V_N)=(1,2,5)\), and
\(5\) similarity classes in the second excited state \((V_L,V_M,V_N)=(2,4,9)\).
Summing up the partial frequencies \(188+128+8\) of these states in Table
\ref{tbl:Gamma}
yields the maximal absolute frequency \(324\) of type \(\gamma\)
among all \(13\) types in the range \(2\le D<10^3\),
as given in Table
\ref{tbl:Statistics}.
The logarithmic subfield unit index of type \(\gamma\)
is restricted to the unique value \(E=6\).
Type \(\gamma\) is the unique type with
\(3\)-dimensional absolute principal factorization, \(A=3\).
However, the \(1\)-dimensional subspace \(\Delta\) is formed by radicals,
and only the complementary \(2\)-dimensional subspace is non-trivial.


\renewcommand{\arraystretch}{1.0}

\begin{table}[ht]
\caption{Splitting of type \(\delta_1\), \((U,\eta,\zeta;A,I,R)=(1,\times,-;1,0,1)\): \(3\) similarity classes}
\label{tbl:Delta1}
\begin{center}
\begin{tabular}{|r|cc|cccr|ccc|cccc|rr|}
\hline
 No. & S  & \(e_0\) & \(t\) & \(u\) & \(v\) & \(m\) & \(n\) & \(s_2\) & \(s_4\) & \(V_L\) & \(V_M\) & \(V_N\) & \(E\) &   \(\mathbf{M}\) & \(\mathbf{\lvert M\rvert}\) \\
\hline
   1 & 1b &   \(2\) & \(1\) & \(0\) & \(1\) & \(1\) & \(0\) &   \(0\) &   \(1\) &   \(3\) &   \(5\) &   \(9\) & \(2\) & \(\mathbf{211}\) &                       \(1\) \\
   2 & 1b &   \(2\) & \(1\) & \(0\) & \(1\) & \(1\) & \(0\) &   \(0\) &   \(1\) &   \(1\) &   \(2\) &   \(3\) & \(4\) & \(\mathbf{421}\) &                       \(5\) \\
   3 & 2  &   \(0\) & \(2\) & \(0\) & \(2\) & \(1\) & \(1\) &   \(0\) &   \(1\) &   \(1\) &   \(2\) &   \(3\) & \(4\) & \(\mathbf{843}\) &                       \(1\) \\
\hline
\end{tabular}
\end{center}
\end{table}

The logarithmic subfield unit index of DPF type \(\delta_1\)
can take two values, either \(E=4\) or \(E=2\).
Type \(\delta_1\) with \(E=4\) splits into
\(2\) similarity classes in the ground state \((V_L,V_M,V_N)=(1,2,3)\).
Type \(\delta_1\) with \(E=2\) consists of
\(1\) similarity class in the ground state \((V_L,V_M,V_N)=(3,5,9)\).
Summing up the partial frequencies \(6+1\) of these states in Table
\ref{tbl:Delta1}
yields the modest absolute frequency \(7\) of type \(\delta_1\) in the range \(2\le D<10^3\),
as given in Table
\ref{tbl:Statistics}.
Type \(\delta_1\) is a type with
\(1\)-dimensional relative principal factorization, \(R=1\).


\renewcommand{\arraystretch}{1.0}

\begin{table}[ht]
\caption{Splitting of type \(\delta_2\), \((U,\eta,\zeta;A,I,R)=(1,\times,-;1,1,0)\): \(5\) similarity classes}
\label{tbl:Delta2}
\begin{center}
\begin{tabular}{|r|cc|cccr|ccc|cccc|rr|}
\hline
 No. & S  & \(e_0\) & \(t\) & \(u\) & \(v\) & \(m\) & \(n\) & \(s_2\) & \(s_4\) & \(V_L\) & \(V_M\) & \(V_N\) & \(E\) &   \(\mathbf{M}\) & \(\mathbf{\lvert M\rvert}\) \\
\hline
   1 & 1b &   \(2\) & \(1\) & \(0\) & \(1\) & \(1\) & \(0\) &   \(1\) &   \(0\) &   \(1\) &   \(1\) &   \(2\) & \(3\) &  \(\mathbf{19}\) &                      \(27\) \\
   2 & 2  &   \(0\) & \(2\) & \(0\) & \(2\) & \(1\) & \(1\) &   \(1\) &   \(0\) &   \(1\) &   \(1\) &   \(2\) & \(3\) &  \(\mathbf{57}\) &                      \(10\) \\
   3 & 1a &   \(6\) & \(1\) & \(0\) & \(1\) & \(4\) & \(0\) &   \(1\) &   \(0\) &   \(1\) &   \(1\) &   \(2\) & \(3\) &  \(\mathbf{95}\) &                       \(9\) \\
   4 & 2  &   \(0\) & \(1\) & \(1'\)& \(0\) & \(1\) & \(0\) &   \(1'\)&   \(0\) &   \(1\) &   \(1\) &   \(2\) & \(3\) & \(\mathbf{149}\) &                       \(6\) \\
   5 & 1b &   \(2\) & \(2\) & \(0\) & \(2\) & \(3\) & \(1\) &   \(1\) &   \(0\) &   \(2\) &   \(3\) &   \(6\) & \(3\) & \(\mathbf{377}\) &                       \(1\) \\
\hline
\end{tabular}
\end{center}
\end{table}

DPF type \(\delta_2\) splits into
\(4\) similarity classes in the ground state \((V_L,V_M,V_N)=(1,1,2)\) and
\(1\) similarity class in the first excited state \((V_L,V_M,V_N)=(2,3,6)\).
Summing up the partial frequencies \(52+1\) of these states in Table
\ref{tbl:Delta2}
yields the considerable absolute frequency \(53\) of type \(\delta_2\) in the range \(2\le D<10^3\),
as given in Table
\ref{tbl:Statistics}.
The logarithmic subfield unit index of type \(\delta_2\)
is restricted to the unique value \(E=3\).
Type \(\delta_2\) is a type with
\(1\)-dimensional intermediate principal factorization, \(I=1\).


\renewcommand{\arraystretch}{1.0}

\begin{table}[ht]
\caption{Splitting of type \(\varepsilon\), \((U,\eta,\zeta;A,I,R)=(1,\times,-;2,0,0)\): \(22\) similarity classes}
\label{tbl:Epsilon}
\begin{center}
\begin{tabular}{|r|cc|cccr|ccc|cccc|rr|}
\hline
 No. & S  & \(e_0\) & \(t\) & \(u\) & \(v\) &  \(m\) & \(n\) & \(s_2\) & \(s_4\) & \(V_L\) & \(V_M\) & \(V_N\) & \(E\) &   \(\mathbf{M}\) & \(\mathbf{\lvert M\rvert}\) \\
\hline
   1 & 1b &   \(2\) & \(1\) & \(0\) & \(1\) &  \(1\) & \(1\) &   \(0\) &   \(0\) &   \(0\) &   \(0\) &   \(0\) & \(5\) &   \(\mathbf{2}\) &                      \(71\) \\
   2 & 1a &   \(6\) & \(1\) & \(0\) & \(1\) &  \(4\) & \(1\) &   \(0\) &   \(0\) &   \(0\) &   \(0\) &   \(0\) & \(5\) &  \(\mathbf{10}\) &                      \(31\) \\
   3 & 2  &   \(0\) & \(2\) & \(0\) & \(2\) &  \(1\) & \(2\) &   \(0\) &   \(0\) &   \(0\) &   \(0\) &   \(0\) & \(5\) &  \(\mathbf{18}\) &                      \(37\) \\
   4 & 1a &   \(6\) & \(2\) & \(1\) & \(1\) & \(16\) & \(2\) &   \(0\) &   \(0\) &   \(1\) &   \(2\) &   \(4\) & \(5\) & \(\mathbf{140}\) &                       \(5\) \\
   5 & 1b &   \(2\) & \(2\) & \(0\) & \(2\) &  \(3\) & \(2\) &   \(0\) &   \(0\) &   \(1\) &   \(2\) &   \(4\) & \(5\) & \(\mathbf{141}\) &                      \(19\) \\
   6 & 1b &   \(2\) & \(2\) & \(0\) & \(2\) &  \(3\) & \(1\) &   \(1\) &   \(0\) &   \(1\) &   \(2\) &   \(4\) & \(5\) & \(\mathbf{171}\) &                       \(6\) \\
   7 & 1a &   \(6\) & \(2\) & \(0\) & \(2\) & \(16\) & \(2\) &   \(0\) &   \(0\) &   \(1\) &   \(2\) &   \(4\) & \(5\) & \(\mathbf{180}\) &                       \(5\) \\
   8 & 2  &   \(0\) & \(3\) & \(1\) & \(2\) &  \(4\) & \(2\) &   \(0\) &   \(0\) &   \(1\) &   \(2\) &   \(4\) & \(5\) & \(\mathbf{182}\) &                       \(1\) \\
   9 & 1b &   \(2\) & \(2\) & \(1\) & \(1\) &  \(4\) & \(1\) &   \(1\) &   \(0\) &   \(1\) &   \(2\) &   \(4\) & \(5\) & \(\mathbf{203}\) &                       \(2\) \\
  10 & 2  &   \(0\) & \(2\) & \(0\) & \(2\) &  \(1\) & \(1\) &   \(1\) &   \(0\) &   \(1\) &   \(2\) &   \(4\) & \(5\) & \(\mathbf{218}\) &                       \(1\) \\
  11 & 1b &   \(2\) & \(3\) & \(1\) & \(2\) & \(12\) & \(3\) &   \(0\) &   \(0\) &   \(2\) &   \(4\) &   \(8\) & \(5\) & \(\mathbf{273}\) &                       \(4\) \\
  12 & 1a &   \(6\) & \(2\) & \(0\) & \(2\) & \(16\) & \(1\) &   \(1\) &   \(0\) &   \(1\) &   \(2\) &   \(4\) & \(5\) & \(\mathbf{290}\) &                       \(2\) \\
  13 & 1b &   \(2\) & \(2\) & \(0\) & \(2\) &  \(3\) & \(1\) &   \(1\) &   \(0\) &   \(1\) &   \(2\) &   \(4\) & \(5\) & \(\mathbf{298}\) &                       \(2\) \\
  14 & 1b &   \(2\) & \(2\) & \(1\) & \(1\) &  \(4\) & \(2\) &   \(0\) &   \(0\) &   \(1\) &   \(2\) &   \(4\) & \(5\) & \(\mathbf{329}\) &                       \(7\) \\
  15 & 1b &   \(2\) & \(3\) & \(0\) & \(3\) & \(13\) & \(2\) &   \(1\) &   \(0\) &   \(2\) &   \(4\) &   \(8\) & \(5\) & \(\mathbf{348}\) &                       \(2\) \\
  16 & 1b &   \(2\) & \(1\) & \(0\) & \(1\) &  \(1\) & \(0\) &   \(1\) &   \(0\) &   \(1\) &   \(2\) &   \(4\) & \(5\) & \(\mathbf{379}\) &                       \(1\) \\
  17 & 1b &   \(2\) & \(2\) & \(0\) & \(2\) &  \(3\) & \(1\) &   \(0\) &   \(1\) &   \(1\) &   \(2\) &   \(4\) & \(5\) & \(\mathbf{422}\) &                       \(6\) \\
  18 & 2  &   \(0\) & \(3\) & \(1\) & \(2\) &  \(4\) & \(2\) &   \(1\) &   \(0\) &   \(1\) &   \(2\) &   \(4\) & \(5\) & \(\mathbf{532}\) &                       \(1\) \\
  19 & 1a &   \(6\) & \(1\) & \(0\) & \(1\) &  \(4\) & \(0\) &   \(1\) &   \(0\) &   \(1\) &   \(2\) &   \(4\) & \(5\) & \(\mathbf{695}\) &                       \(1\) \\
  20 & 1b &   \(2\) & \(3\) & \(0\) & \(3\) & \(13\) & \(3\) &   \(0\) &   \(0\) &   \(2\) &   \(4\) &   \(8\) & \(5\) & \(\mathbf{702}\) &                       \(2\) \\
  21 & 2  &   \(0\) & \(2\) & \(2\) & \(0\) &  \(4\) & \(2\) &   \(0\) &   \(0\) &   \(1\) &   \(2\) &   \(4\) & \(5\) & \(\mathbf{749}\) &                       \(1\) \\
  22 & 1a &   \(6\) & \(1\) & \(1\) & \(0\) &  \(4\) & \(1\) &   \(0\) &   \(0\) &   \(1\) &   \(2\) &   \(4\) & \(5\) & \(\mathbf{785}\) &                       \(1\) \\
\hline
\end{tabular}
\end{center}
\end{table}

DPF type \(\varepsilon\) splits into
\(3\) similarity classes in the ground state \((V_L,V_M,V_N)=(0,0,0)\),
\(16\) similarity classes in the first excited state \((V_L,V_M,V_N)=(1,2,4)\), and
\(3\) similarity classes in the second excited state \((V_L,V_M,V_N)=(2,4,8)\).
Summing up the partial frequencies \(139+61+8\) of these states in Table
\ref{tbl:Epsilon}
yields the high absolute frequency \(208\) of type \(\varepsilon\) in the range \(2\le D<10^3\),
as given in Table
\ref{tbl:Statistics}.
The logarithmic subfield unit index of type \(\varepsilon\)
is restricted to the unique value \(E=5\).
Type \(\varepsilon\) is a type with
\(2\)-dimensional absolute principal factorization, \(A=2\).


\renewcommand{\arraystretch}{1.0}

\begin{table}[ht]
\caption{No splitting of type \(\zeta_1\), \((U,\eta,\zeta;A,I,R)=(1,-,\times;1,0,1)\): \(1\) similarity class}
\label{tbl:Zeta1}
\begin{center}
\begin{tabular}{|r|cc|cccr|ccc|cccc|rr|}
\hline
 No. & S  & \(e_0\) & \(t\) & \(u\) & \(v\) & \(m\) & \(n\) & \(s_2\) & \(s_4\) & \(V_L\) & \(V_M\) & \(V_N\) & \(E\) &   \(\mathbf{M}\) & \(\mathbf{\lvert M\rvert}\) \\
\hline
   1 & 2  &   \(0\) & \(1\) & \(1'\)& \(0\) & \(1\) & \(0\) &   \(0\) &   \(1'\)&   \(1\) &   \(2\) &   \(4\) & \(5\) & \(\mathbf{101}\) &                       \(1\) \\
\hline
\end{tabular}
\end{center}
\end{table}

The logarithmic subfield unit index of DPF type \(\zeta_1\)
is restricted to the unique value \(E=5\).
Type \(\zeta_1\) consists of
\(1\) similarity class in the ground state \((V_L,V_M,V_N)=(1,2,4)\).
The frequency \(1\) of this state in Table
\ref{tbl:Zeta1}
coincides with the negligible absolute frequency \(1\) of type \(\zeta_1\) in the range \(2\le D<10^3\),
as given in Table
\ref{tbl:Statistics}.
Type \(\zeta_1\) is a type with
\(1\)-dimensional relative principal factorization, \(R=1\).


\renewcommand{\arraystretch}{1.0}

\begin{table}[ht]
\caption{Splitting of type \(\zeta_2\), \((U,\eta,\zeta;A,I,R)=(1,-,\times;1,1,0)\): \(3\) similarity classes}
\label{tbl:Zeta2}
\begin{center}
\begin{tabular}{|r|cc|cccr|ccc|cccc|rr|}
\hline
 No. & S  & \(e_0\) & \(t\) & \(u\) & \(v\) & \(m\) & \(n\) & \(s_2\) & \(s_4\) & \(V_L\) & \(V_M\) & \(V_N\) & \(E\) &   \(\mathbf{M}\) & \(\mathbf{\lvert M\rvert}\) \\
\hline
   1 & 1a &   \(6\) & \(1\) & \(1'\)& \(0\) & \(4\) & \(0\) &   \(0\) &   \(1'\)&   \(1\) &   \(1\) &   \(3\) & \(4\) & \(\mathbf{505}\) &                       \(2\) \\
   2 & 2  &   \(0\) & \(2\) & \(2'\)& \(0\) & \(4\) & \(1\) &   \(0\) &   \(1'\)&   \(1\) &   \(1\) &   \(3\) & \(4\) & \(\mathbf{707}\) &                       \(1\) \\
   3 & 1a &   \(6\) & \(1\) & \(1\) & \(0\) & \(4\) & \(0\) &   \(1\) &   \(0\) &   \(1\) &   \(1\) &   \(3\) & \(4\) & \(\mathbf{745}\) &                       \(2\) \\
\hline
\end{tabular}
\end{center}
\end{table}

DPF type \(\zeta_2\) consists of \(3\) similarity classes.
The modest absolute frequency \(5\) of type \(\zeta_2\) in the range \(2\le D<10^3\),
given in Table
\ref{tbl:Statistics},
is the sum \(2+1+2\) of partial frequencies in Table
\ref{tbl:Zeta2}.
Type \(\zeta_2\) only occurs with logarithmic subfield unit index \(E=4\).
It is a type with
\(1\)-dimensional intermediate principal factorization, \(I=1\).


\renewcommand{\arraystretch}{1.0}

\begin{table}[ht]
\caption{Splitting of type \(\eta\), \((U,\eta,\zeta;A,I,R)=(1,-,\times;2,0,0)\): \(2\) similarity classes}
\label{tbl:Eta}
\begin{center}
\begin{tabular}{|r|cc|cccr|ccc|cccc|rr|}
\hline
 No. & S  & \(e_0\) & \(t\) & \(u\) & \(v\) & \(m\) & \(n\) & \(s_2\) & \(s_4\) & \(V_L\) & \(V_M\) & \(V_N\) & \(E\) &   \(\mathbf{M}\) & \(\mathbf{\lvert M\rvert}\) \\
\hline
   1 & 1a &   \(6\) & \(1\) & \(1\) & \(0\) & \(4\) & \(1\) &   \(0\) &   \(0\) &   \(0\) &   \(0\) &   \(1\) & \(6\) &  \(\mathbf{35}\) &                       \(6\) \\
   2 & 2  &   \(0\) & \(2\) & \(2\) & \(0\) & \(4\) & \(2\) &   \(0\) &   \(0\) &   \(0\) &   \(0\) &   \(1\) & \(6\) & \(\mathbf{301}\) &                       \(1\) \\
\hline
\end{tabular}
\end{center}
\end{table}

DPF type \(\eta\) splits in \(2\) similarity classes,
\(\lbrack\mathbf{35}\rbrack\) and \(\lbrack\mathbf{301}\rbrack\).
The modest absolute frequency \(7\) of type \(\eta\) in the range \(2\le D<10^3\),
given in Table
\ref{tbl:Statistics},
is the sum \(6+1\) of partial frequencies in Table
\ref{tbl:Eta}.
Type \(\eta\) only occurs with logarithmic subfield unit index \(E=6\).
It is a type with
\(2\)-dimensional absolute principal factorization, \(A=2\).
However, it should be pointed out that
outside of the range of our systematic investigations
we found an excited state \((V_L,V_M,V_N)=(1,2,5)\)
for the similarity class \(\mathbf{\lbrack 1505\rbrack}\),
where \(1505=5\cdot 7\cdot 43\) has three prime divisors,
additionally to the ground state \((V_L,V_M,V_N)=(0,0,1)\).


\renewcommand{\arraystretch}{1.0}

\begin{table}[ht]
\caption{Splitting of type \(\vartheta\), \((U,\eta,\zeta;A,I,R)=(0,\times,\times;1,0,0)\): \(2\) similarity classes}
\label{tbl:Theta}
\begin{center}
\begin{tabular}{|r|cc|cccr|ccc|cccc|rr|}
\hline
 No. & S  & \(e_0\) & \(t\) & \(u\) & \(v\) & \(m\)  & \(n\) & \(s_2\) & \(s_4\) & \(V_L\) & \(V_M\) & \(V_N\) & \(E\) & \(\mathbf{M}\) & \(\mathbf{\lvert M\rvert}\) \\
\hline
   1 & 1a &   \(6\) & \(0\) & \(0\) & \(0\) & \(1\)  & \(0\) &   \(0\) &   \(0\) &   \(0\) &   \(0\) &   \(0\) & \(5\) & \(\mathbf{5}\) &                       \(1\) \\
   2 & 2  &   \(0\) & \(1\) & \(1\) & \(0\) & \(1\)  & \(1\) &   \(0\) &   \(0\) &   \(0\) &   \(0\) &   \(0\) & \(5\) & \(\mathbf{7}\) &                      \(18\) \\
\hline
\end{tabular}
\end{center}
\end{table}

DPF type \(\vartheta\) splits into
the unique finite similarity class \(\lbrack\mathbf{5}\rbrack\) with only a single element
and the infinite parametrized sequence \(\lbrack\mathbf{7}\rbrack\)
consisting of all prime radicands \(D=q\) congruent to \(\pm 7\,(\mathrm{mod}\,25)\).
The small absolute frequency \(19\) of type \(\vartheta\) in the range \(2\le D<10^3\),
given in Table
\ref{tbl:Statistics},
is the sum \(\mathbf{\lvert 5\rvert}+\mathbf{\lvert 7\rvert}=1+18\) in Table
\ref{tbl:Theta}.
Since no theoretical argument disables the occurrence of type \(\vartheta\)
for composite radicands \(D\) with prime factors \(5\) and \(q\equiv\pm 7\,(\mathrm{mod}\,25)\),
we conjecture that such cases will appear in bigger ranges with \(D>10^3\).
Type \(\vartheta\) only occurs with logarithmic subfield unit index \(E=5\),
and is the unique type where
every unit of \(K\) occurs as norm of a unit of \(N\), that is \(U=0\).


\subsection{Increasing dominance of DPF type \(\gamma\) for \(T\to\infty\)}
\label{ss:ManyPrmDiv}
\noindent
In this final section, we want to show that the careful book keeping
of similarity classes with representative prototypes in the Tables
\ref{tbl:Alpha1}
--
\ref{tbl:Theta}
is useful for the quantitative illumination of many other phenomena.
For an explanation, we select the phenomenon of \textit{absolute} principal factorizations.

The statistical distribution of DPF types in Table
\ref{tbl:Statistics}
has proved that \textit{type} \(\gamma\) with \(324\) occurrences, that is \(36\,\%\),
among all \(900\) fields \(N=\mathbb{Q}(\zeta_5,\sqrt[5]{D})\)
with normalized radicands in the range \(2\le D<10^3\)
is doubtlessly the \textit{high champion} of all DPF types.
This means that there is a clear trend towards
the maximal possible extent of \(3\)-dimensional spaces of
\textit{absolute} principal factorizations, \(A=3\),
in spite of the disadvantage that the estimate \(1\le A\le\min(3,T)\)
in the formulas (4.3) and (4.4) of
\cite[Cor. 4.1]{Ma2a}
prohibits type \(\gamma\) for conductors \(f\) with \(T\le 2\) prime divisors.

For the following investigation,
we have to recall that the number \(T\) of all prime factors of \(f^4=5^{e_0}\cdot q_1^4\ldots q_t^4\)
is given by \(T=t+1\) for fields of Dedekind's species \(1\), where \(e_0\in\lbrace 2,6\rbrace\),
and by \(T=t\) for fields of Dedekind's species \(2\), where \(e_0=0\).

Conductors \(f\) with \(T=4\) prime factors occur in six tables, \\
\(1\) case of type \(\alpha_2\) in a single similarity class of Table
\ref{tbl:Alpha2}, \\
\(1\) case of type \(\alpha_3\) in a single similarity class of Table
\ref{tbl:Alpha3}, \\
\(7\) cases of type \(\beta_1\) in a single similarity class of Table
\ref{tbl:Beta1}, \\
\(10\) cases of type \(\beta_2\) in \(5\) similarity classes of Table
\ref{tbl:Beta2}, \\
\(126\) cases of type \(\gamma\) in \(16\) similarity classes of Table
\ref{tbl:Gamma}, \\
\(8\) cases of type \(\varepsilon\) in \(3\) similarity classes of Table
\ref{tbl:Epsilon}, \\
that is, a total of \(153\) cases,
with respect to the complete range \(2\le D<10^3\) of our computations.
Consequently, we have an increase of type \(\gamma\)
from \(36.0\,\%\), with respect to the entire database,
to \(\frac{126}{153}=82.4\,\%\), with respect to \(T=4\).

The feature is even aggravated for conductors \(f\) with \(T=5\) prime factors,
which exclusively occur in Table
\ref{tbl:Gamma}.
There are \(4\) similarity classes with \(T=5\), namely
\(\mathbf{\lbrack 462\rbrack}\), \(\mathbf{\lbrack 546\rbrack}\), \(\mathbf{\lbrack 798\rbrack}\), \(\mathbf{\lbrack 858\rbrack}\),
with a total of \(6\) elements, all \((100\,\%)\) with associated fields of type \(\gamma\).


\section{Acknowledgements}
\label{s:Thanks}

\noindent
We gratefully acknowledge that our research was supported by the Austrian Science Fund (FWF):
projects J 0497-PHY and P 26008-N25.
This work is dedicated to the memory of Charles J. Parry (\(\dagger\) 25 December 2010)
who suggested a numerical investigation of pure quintic number fields.


\begin{figure}[ht]
\caption{Lattices of subfields of \(N\) and of subgroups of \(G=\mathrm{Gal}(N/\mathbb{Q})\)}
\label{fig:GaloisCorrespondence}

{\tiny

\setlength{\unitlength}{1.0cm}
\begin{picture}(15,7)(-11,-8.5)

\put(-10,-2){\makebox(0,0)[cb]{Degree}}
\put(-10,-4){\vector(0,1){1.5}}
\put(-10,-9){\line(0,1){5}}
\multiput(-10.1,-9)(0,1){6}{\line(1,0){0.2}}

\put(-10.2,-4){\makebox(0,0)[rc]{\(20\)}}
\put(-9.8,-4){\makebox(0,0)[lc]{metacyclic}}
\put(-10.2,-5){\makebox(0,0)[rc]{\(10\)}}
\put(-9.8,-5){\makebox(0,0)[lc]{conjugates of maximal real}}
\put(-10.2,-6){\makebox(0,0)[rc]{\(5\)}}
\put(-9.8,-6){\makebox(0,0)[lc]{conjugates of pure quintic}}
\put(-10.2,-7){\makebox(0,0)[rc]{\(4\)}}
\put(-9.8,-7){\makebox(0,0)[lc]{cyclotomic}}
\put(-10.2,-8){\makebox(0,0)[rc]{\(2\)}}
\put(-9.8,-8){\makebox(0,0)[lc]{real quadratic}}
\put(-10.2,-9){\makebox(0,0)[rc]{\(1\)}}
\put(-9.8,-9){\makebox(0,0)[lc]{rational}}

{\normalsize
\put(-1,-2){\makebox(0,0)[cc]{Galois Correspondence}}
\put(-1,-2.5){\makebox(0,0)[cc]{\(\longleftrightarrow\)}}
}



\put(-6,-9){\circle*{0.2}}
\put(-6,-9.2){\makebox(0,0)[ct]{\(\mathbb{Q}\)}}

\put(-6,-9){\line(2,1){2}}

\put(-4,-8){\circle*{0.2}}
\put(-4,-8.2){\makebox(0,0)[ct]{\(K^+\)}}

\put(-4,-8){\line(2,1){2}}

\put(-2,-7){\circle*{0.2}}
\put(-2,-7.2){\makebox(0,0)[ct]{\(K\)}}



\put(0,-9){\circle*{0.2}}
\put(0,-9.2){\makebox(0,0)[ct]{\(G=\langle\sigma,\tau\rangle\)}}

\put(0,-9){\line(2,1){2}}

\put(2,-8){\circle*{0.2}}
\put(2,-8.2){\makebox(0,0)[ct]{\(\langle\sigma,\tau^2\rangle\)}}

\put(2,-8){\line(2,1){2}}

\put(4,-7){\circle*{0.2}}
\put(4,-7.2){\makebox(0,0)[ct]{\(\langle\sigma\rangle\)}}


\put(-6,-9){\line(0,1){3}}
\put(-4,-8){\line(0,1){3}}
\put(-2,-7){\line(0,1){3}}

\put(0,-9){\line(0,1){3}}
\put(2,-8){\line(0,1){3}}
\put(4,-7){\line(0,1){3}}



\put(-6,-6){\circle{0.2}}
\put(-6,-5.8){\makebox(0,0)[cb]{\(L\)}}
\put(-5.8,-6.2){\makebox(0,0)[lc]{\(L^{\sigma},\ldots,L^{\sigma^4}\)}}

\put(-6,-6){\line(2,1){2}}

\put(-4,-5){\circle{0.2}}
\put(-4,-4.8){\makebox(0,0)[cb]{\(M\)}}
\put(-3.8,-5.2){\makebox(0,0)[lc]{\(M^{\sigma},\ldots,M^{\sigma^4}\)}}

\put(-4,-5){\line(2,1){2}}

\put(-2,-4){\circle*{0.2}}
\put(-2,-3.8){\makebox(0,0)[cb]{\(N\)}}



\put(0,-6){\circle{0.2}}
\put(0,-5.8){\makebox(0,0)[cb]{\(\langle\tau\rangle\)}}
\put(0.2,-6.2){\makebox(0,0)[lc]{\(\langle\sigma^{-i}\tau\sigma^i\rangle\)}}
\put(0.2,-6.5){\makebox(0,0)[lc]{\(1\le i\le 4\)}}

\put(0,-6){\line(2,1){2}}

\put(2,-5){\circle{0.2}}
\put(2,-4.8){\makebox(0,0)[cb]{\(\langle\tau^2\rangle\)}}
\put(2.2,-5.2){\makebox(0,0)[lc]{\(\langle\sigma^{-i}\tau^2\sigma^i\rangle\)}}
\put(2.2,-5.5){\makebox(0,0)[lc]{\(1\le i\le 4\)}}

\put(2,-5){\line(2,1){2}}

\put(4,-4){\circle*{0.2}}
\put(4,-3.8){\makebox(0,0)[cb]{\(1\)}}


\end{picture}

}

\end{figure}
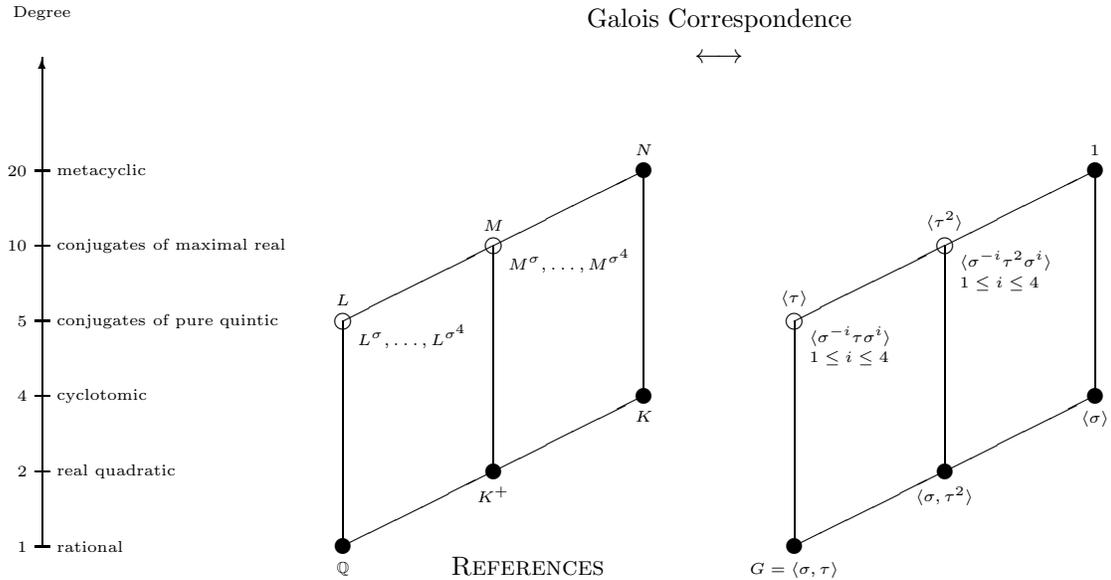




\begin{thebibliography}{XX}
%
\bibitem{AMITA}
S. Aouissi, D. C. Mayer, M. C. Ismaili, M. Talbi and A. Azizi, \\
\textit{\(3\)-rank of ambiguous class groups in cubic Kummer extensions}, 
arXiv:1804.00767v3 [math.NT]
9 Aug 2018.
%
\bibitem{AMI}
S. Aouissi, D. C. Mayer and M. C. Ismaili,
\textit{Structure of relative genus fields of cubic Kummer extensions}, \\
arXiv:1808.04678v1 [math.NT]
14 Aug 2018.
%
\bibitem{BaCo1}
P. Barrucand and H. Cohn, 
\textit{A rational genus, class number divisibility, and unit theory for pure cubic fields}, \\
J. Number Theory
\textbf{2}
(1970),
no. 1,
7--21. 
%
\bibitem{BaCo2}
P. Barrucand and H. Cohn, 
\textit{Remarks on principal factors in a relative cubic field}, \\
J. Number Theory
\textbf{3}
(1971),
no. 2,
226--239. 
%
\bibitem{BCP}
W. Bosma, J. Cannon, and C. Playoust,
\textit{The Magma algebra system. I. The user language}, \\
J. Symbolic Comput.
\textbf{24}
(1997),
235--265.
%
\bibitem{BCFS}
W. Bosma, J. J. Cannon, C. Fieker, and A. Steels (eds.),
\textit{Handbook of Magma functions} \\
(Edition 2.24,
Sydney,
2018).
%
\bibitem{GAP}
GAP Developer Group,
GAP -- \textit{Groups, Algorithms, and Programming --- a System for Computational Discrete Algebra},
Version 4.10.0,
Aachen, Braunschweig, Fort Collins, St. Andrews,
2018,\\
\verb+(http://www.gap-system.org)+.
%
\bibitem{Ii}
K. Iimura,
\textit{A criterion for the class number of a pure quintic field to be divisible by \(5\)}, \\
J. Reine Angew. Math.
\textbf{292}
(1977),
201--210.
%
\bibitem{Is}
M. Ishida,
\textit{Class numbers of algebraic number fields of Eisenstein type}, \\
J. Number Theory
\textbf{2}
(1970),
404--413.
%
\bibitem{Ky1}
H. Kobayashi,
\textit{Class numbers of pure quintic fields},
J. Number Theory
\textbf{160}
(2016),
no. 1,
463--477.
%
\bibitem{Ky2}
H. Kobayashi,
\textit{Class numbers of pure quintic fields}, \\
Ph.D. thesis,
Osaka University Knowledge Archive (OUKA),
DOI 10.18910/56049.
%
\bibitem{MAGMA}
MAGMA Developer Group,
MAGMA Computational Algebra System,
Version 2.24-2,
Sydney,
2018, \\
\verb+(http://magma.maths.usyd.edu.au)+.
%
\bibitem{Ma0}
D. C. Mayer, 
\textit{Classification of dihedral fields},
Preprint,
Dept. of Comp. Science,
Univ. of Manitoba,
1991. 
%
\bibitem{Ma1}
D. C. Mayer, 
\textit{Discriminants of metacyclic fields}, 
Canad. Math. Bull. 
\textbf{36}
(1)
(1993),
103--107.
%
\bibitem{Ma2a}
D. C. Mayer, 
\textit{Differential principal factors and Polya property of pure metacyclic fields},
Preprint,
2018. 
%
\bibitem{Ma2b}
D. C. Mayer, 
\textit{Similarity classes and prototypes of pure metacyclic fields},
Preprint,
2018. 
%
\bibitem{PARI}
PARI Developer Group, PARI/GP, Version 64 bit 2.11.0,
Bordeaux,
2018, \\
\verb+(http://pari.math.u-bordeaux.fr)+.
%
\bibitem{Pa}
C. J. Parry,
\textit{Class number relations in pure quintic fields},
Symposia Mathematica
\textbf{15}
(1975),
475--485, \\
Academic Press,
London.
(Convegno di Strutture in Corpi Algebrici,
INDAM,
Rome,
1973.)
%
\bibitem{Vo}
G. F. Voronoi, 
\textit{Ob odnom obobshchenii algorifma nepreryvnykh drobei}
(On a generalization of the algorithm of continued fractions), 
Doctoral Dissertation,
1896,
Warsaw
(Russian). 
%
\end{thebibliography}
\end{document}